\documentclass[a4paper,12pt]{amsart}

\usepackage{bbm}

\usepackage[utf8]{inputenc}
\usepackage[english]{babel}             
\usepackage{geometry}
\usepackage{lipsum}
\usepackage{parskip}
\usepackage{enumitem}
\usepackage{amsfonts,amsmath,amssymb,amsthm}    
\usepackage{mathtools}
\usepackage{bbm}
\usepackage{setspace}
\usepackage{esint}
\usepackage{lmodern}                    
\usepackage{graphicx}                   
\usepackage{xcolor}                     
\usepackage{hyperref}                   
\usepackage{bookmark}
\usepackage[backend=biber,style=numeric,bibencoding=utf8,isbn=false,url=false,sorting=nty,doi=false,giveninits=true,maxbibnames=10,maxalphanames=10]{biblatex}  
\AtEveryBibitem{\clearfield{month}}
\usepackage{csquotes}                   
\usepackage{enumitem}                   
\usepackage{parskip}
\usepackage[margin=1cm,labelfont=bf]{caption}   
\usepackage{float}

\savegeometry{paper}                    

\graphicspath{{./files/figures/}}       

\AtBeginBibliography{\small}            

\allowdisplaybreaks                     

\setlist[enumerate,1]{label=(\roman*)}  

\numberwithin{equation}{section}

\DeclareFieldFormat{extraalpha}{#1}     
\DeclareLabelalphaTemplate{
  \labelelement{
    \field[final]{shorthand}
    \field{label}
    \field[strwidth=2 , strside=left,ifnames=1]{labelname}
    \field[strside=left,ifnames=2-,varwidthlist=true]{labelname}
  }
}
\DeclareLabelalphaNameTemplate{
  \namepart[use=true, pre=true, strwidth=1, compound]{prefix}
  \namepart{family} 
}

\renewbibmacro{in:}{}                   
\DeclareFieldFormat[article,periodical]{volume}{\mkbibbold{#1}}
\DeclareFieldFormat[article,periodical]{journaltitle}{#1\isdot}
\DeclareFieldFormat[article,periodical]{title}{#1}
\newcommand{\mcalA}{\mathcal{A}}

\newcommand{\mcalF}{\mathcal{F}}

\newcommand{\mcalH}{\mathcal{H}}

\newcommand{\mcalT}{\mathcal{T}}

\newcommand{\naturals}{\mathbb{N}}
\newcommand{\N}{\naturals}

\newcommand{\reals}{\mathbb{R}}
\newcommand{\R}{\reals}
\newcommand{\complex}{\mathbb{C}}
\newcommand{\C}{\complex}

\newcommand{\field}{\mathbb{F}}

\newcommand{\eps}{\varepsilon}

\newcommand{\gs}{>}

\newcommand{\conj}[1]{\overline{#1}}
\newcommand{\norm}[1]{\left\|  #1  \right\|}
\newcommand{\abs}[1]{\left| #1 \right|}
\newcommand{\vol}[1]{\abs{#1}}
\newcommand{\of}[1]{\left( #1 \right)}

\newcommand{\inv}[1]{\frac{1}{#1}}

\newcommand{\inset}[1]{\left \{ #1 \right \}}
\newcommand{\angles}[1]{\left\langle #1 \right\rangle}
\newcommand{\inner}[1]{\angles{#1}}

\newcommand{\pe}{\partial_\e}

\newcommand{\charf}[1]{\mathbbm{1}_{#1}}

\DeclareMathOperator{\tRe}{Re}

\DeclareMathOperator{\diff}{d\!}

\newcommand{\theoremname}{Theorem}
\newcommand{\lemmaname}{Lemma}
\newcommand{\propositionname}{Proposition}
\newcommand{\definitionname}{Definition}
\newcommand{\corollaryname}{Corollary}
\newcommand{\claimname}{Claim}

\newcommand{\examplename}{Example}
\newcommand{\remarkname}{Remark}
\newcommand{\commentsname}{Comments}
\newcommand{\notename}{Note}
\newcommand{\notationname}{Notation}

\newcounter{allcounter}[section]

\newtheorem{numthm}[allcounter]{\theoremname}
\newtheorem{theorem}{\theoremname}

\newtheorem*{lemma}{\lemmaname}                     
\newtheorem{numlemma}[allcounter]{\lemmaname}

\newtheorem{numprop}[allcounter]{\propositionname}

\newtheorem{numcor}[allcounter]{\corollaryname}

\newtheorem{numdefn}[allcounter]{\definitionname}

\newtheorem*{theorem*}{\theoremname}

\theoremstyle{definition}
\newtheorem{remark}{\remarkname}                   

\makeatletter
\AtEveryBibitem{\def\@currentlabel{\thefield{labelnumber}}\label{labelnumber-\thefield{entrykey}}}
\makeatother

\savegeometry{paper}
\allowdisplaybreaks
\setlist[enumerate,1]{label=(\roman*)}
\numberwithin{equation}{section}

\newcommand{\disk}{\mathbb{D}}
\newcommand\ee{v^*}
\newcommand{\Fock}{\mcalF^2}

\newcommand{\Fnorm}[1]{\norm{#1}_{\Fock}}
\newcommand{\Om}{\Omega}
\newcommand{\om}{\omega}
\newcommand{\dso}{\delta_{s_0}}

\newcommand{\Berg}{\mathbf{B}_{\alpha}^2}
\newcommand{\K}{\mathcal K}
\newcommand{\e}{\varepsilon}
\newcommand{\p}{\partial}
\newcommand{\mc}{\mathcal}
\newcommand{\ff}{\mathbf{f}}
\newcommand{\mb}{\mathbb}
\newcommand{\B}{\mathbb{B}}
\newcommand{\vhi}{\varphi}
\DeclareMathOperator*{\ddiv}{div}
\renewcommand{\Re}{\operatorname{Re}}
\renewcommand{\Im}{\operatorname{Im}}

\title[Uniform Stability of Faber-Krahn for Wavelet transforms]{Uniform stability of concentration inequalities and applications}

\author{Jaime Gómez}
\address{Mathematics section, EPFL, Lausanne, Switzerland.}
\email{jaime.gomezramirez@epfl.ch}
\author{David Kalaj}
\address{University of Montenegro, Faculty of natural sciences and mathematics, Podgorica, Cetinjski put b.b. 81000 Podgorica, Montenegro }
\email{davidk@ucg.ac.me}
\author{Petar Melentijevi\'c}
\address{University of Belgrade, Faculty of Mathematics, 11000, Belgrade, Serbia}
\email{petarmel@matf.bg.ac.rs}
\author{Jo\~ao P. G. Ramos}
\address{Department of Mathematics, Faculty of Sciences of the University of Lisbon, Lisbon, Portugal}
\email{joaopgramos95@gmail.com}


\begin{document}
\begin{abstract}
We prove a sharp quantitative version of recent Faber-Krahn inequalities for the continuous Wavelet transforms associated to a certain family of Cauchy wavelet windows \cite{RamosTilli}. Our results are \emph{uniform} on the parameters of the family of Cauchy wavelets, and asymptotically \emph{sharp in both directions}.

As a corollary of our results, we are able to recover not only the original result for the short-time Fourier transform as a limiting procedure, but also a \emph{new} concentration result for functions in Hardy spaces. This is a \emph{completely novel} result about optimal concentration of Poisson extensions, and our proof automatically comes with a \emph{sharp stability} version of that inequality. 

Our techniques highlight the intertwining of geometric and complex-analytic arguments involved in the context of concentration inequalities. In particular, in the process of deriving uniform results, we obtain a refinement over the proof of the result in \cite{Inv2}, further improving the current understanding of the geometry of near extremals in all contexts under consideration. 
\end{abstract}

\maketitle

\section{Introduction}

Given a fixed function $g \in L^2(\R),$ the \emph{Wavelet transform} with window $g$ is defined as
\begin{equation}\label{eq:wavelet-transform}
W_gf(x,y) = \frac{1}{y^{1/2}} \int_{\R} f(t)\overline{ g\left( \frac{t-x}{y}\right) }\, \diff t,
\quad \forall f \in L^2(\R),
\end{equation}
whenever $x \in \R, y > 0.$ This operator was introduced first by I. Daubechies and T. Paul in \cite{DaubechiesPaul}, where the authors were concerned to its relationship with time-frequency localization, in analogy to the same family of operators of a similar kind introduced in the case of the short-time Fourier transform by Daubechies \cite{Daubechies} and Berezin \cite{Berezin}. Together with the short-time Fourier transform, the Wavelet transform is one of the most standard ways to jointly encompass time-frequency and time-scale information of a given signal. We refer the reader to \cite{Kailath,Bello,Pfander} for uses of the wavelet transform in more applied contexts, and in particular to \cite{Abreu2012,Abreu2021, AbreuGrochRomero, AbreuPerRomero, GroechenigBook, WongWaveletBook} 
 and the references therein for a more recent mathematical point of view, especially regarding the problem of obtaining information from a domain from information on its localization operator. 

In this manuscript, we shall be interested in the continuous wavelet transform; how much of its mass, in an  $L^2(\C_+,y^{-2} \, dx \, dy)-$sense, can be concentrated on certain subsets of the upper half space, and stability estimates of such inequalities. Indeed, fix $\beta > 0.$ We then define $\psi_{\beta} \in L^2(\R)$ to be such that \[
\widehat{\psi_{\beta}}(t) = \frac{1}{c_{\beta}} \charf{[0,+\infty)} t^{\beta} e^{-t},
\]
where one lets $c_{\beta} = \int_0^{\infty} t^{2\beta - 1} e^{-2t} dt = 2^{2\beta -1}\Gamma(2\beta).$ Here, we normalize the Fourier transform as
\[
\widehat{f}(\xi) = \frac{1}{(2\pi)^{1/2}} \int_{\R} f(t) e^{-it \xi} \diff t.
\]
We define then the \emph{$\beta$-optimal concentration rate} on a set $\Omega \subset \C_+$ as 
\[
\lambda_{\beta}(\Om) := \sup \left\{ \int_{\Om} |W_{\overline{\psi_{\beta}}} f(x,y)|^2 \,\frac{ \diff x \diff y}{y^2} \colon f \in H^2(\C_+), \|f\|_2 = 1 \right\}.
\]
The constant $\lambda_{\beta}(\Om)$ measures the maximal Wavelet concentration of order $\beta >0$ in $\Om$. Regarding evaluating the quantities above and analyzing how large they can be, the analogue of this problem has appeared especially in the context of the Short-time Fourier transform in \cite{Abreu2021, NicolaTilli, NicolaTilli2}. For the presently studied case of the wavelet transform, we mention, in particular, the paper by L. D. Abreu and M. D\"orfler \cite{Abreu2012}, where the authors pose this question explicitly in their last remark. This problem has been recently settled in \cite{RamosTilli}, where the authors obtained the following result: 

\begin{theorem*}[Theorem~1.1 in \cite{RamosTilli}] It holds that
	\begin{equation}\label{eq:first-theorem}
  \lambda_{\beta}(\Om) \le \lambda_{\beta}(\Om^*),
	\end{equation}
	where $\Om^* \subset \C_+$ denotes any pseudohyperbolic disc so that $\nu(\Om) = \nu(\Om^*),$ where $\diff \nu(x,y) = \frac{\diff x \, \diff y}{y^2}.$ Moreover, equality holds in \eqref{eq:first-theorem} if and only if $\Om$ is a pseudohyperbolic disc of measure $\nu(\Om).$
\end{theorem*}

 In the context of the Short-time Fourier transform, we highlight that the analogous result to the one above was obtained in \cite{NicolaTilli} (see also \cite{NicolaTilli2} for a further generalization), and in the context of higher-dimensional Wavelet transforms it has been explored by the second and fourth authors of this manuscript \cite{KalajRamos}. We also highlight the papers \cite{Kalaj1,Kalaj2024,KalajMelentijevic, Llinares, Kulikov,KlukovNicolaOrtegaTilli,Melentijevic} and the references therein as further sources of instances where the ideas used in order to originally prove the main result in \cite{NicolaTilli} have been successfully applied. 

 The main purpose of this paper, on the other hand, will be to provide a \emph{sharp} stability version of the result above. Indeed, this has been done in \cite{Inv2} in a collaboration of the first and fourth authors of this paper with A. Guerra and P. Tilli in the case of the Short-time Fourier transform. We note that it has been recently brought to our attention that M.A. Garc\'ia-Ferrero and J. Ortega-Cerd\`a \cite{GarciaOrtega} have independently successfully achieved a similar result for the sharp \emph{polynomial} concentration problem, as considered by \cite{KlukovNicolaOrtegaTilli}, recovering also the original case considered in \cite{Inv2} as a limiting case -- see also \cite{NicolaRiccardi} for a different yet related stability result in the context of Hilbert-Schmidt norms of localization operators. A main difference of our main result and the ones in \cite{GarciaOrtega} is the presence of another limitting parameter in our analysis - namely, the one which is yielded when taking the Cauchy windows $\psi_{\beta}$ as defined before to be degenerate. Effectively, the first main result of this manuscript reads as follows: 

 \begin{numthm}\label{thm:main-wavelet} Let $\beta > 0$ and $\Omega \subset \C_+$ be such that $\nu(\Omega) = s$, where $\diff\nu(x,y) = \frac{\diff x \diff y}{y^2}$. Then we have that 
 \begin{equation}\label{eqn:StabilityFunction}
        \inf_{\substack{\abs{c} = \|f\|_2 ,\\ z \in \C_+}} \frac{\|f-c\pi_z \psi_{\beta}\|_2}{\|f\|_2} \leq C \of{1 + \frac{2\beta+1}{2\beta} \left[ \of{1+\frac{s}{4\pi}}^{2\beta} -1\right]}^{1/2} \eta(f;\Omega,\beta)^{1/2},
    \end{equation}
where we let $\pi_z \psi_{\beta}(t) = y^{-1/2} \psi_{\beta}\left( \frac{t-x}{y}\right)$ for $z = x + iy$, and 
\[
\eta(f;\Omega,\beta) = \frac{\|f\|_2 \left( 1 - \left( 1 + \frac{s}{4\pi}\right)^{-2\beta}\right) - \int_{\Omega} |W_{\psi_\beta}f(x,y)|^2 \, \diff\nu(x,y)}{\|f\|_2 \left( 1 - \left( 1 + \frac{s}{4\pi}\right)^{-2\beta}\right)}
\]
denotes the \emph{Wavelet deficit of order $\beta$}. Moreover, if we define the \emph{hyperbolic asymmetry} of $\Omega$ to be 
\begin{equation}\label{eqn:HyperbolicAsymmetry}
    \mcalA_{\C_+}(\Omega) := \inf \inset{ \frac{\nu(\Omega \Delta B(x,r))}{\nu(\Omega)} \colon \nu(B(x,r)) = \nu(\Omega) , \ x \in \C_+ },
\end{equation}
then there is an \emph{explicit} constant $K(s,\beta) > 0$ such that 
\[
\mcalA_{\C_+}(\Omega)  \le K(s,\beta) \eta(f;\Omega,\beta)^{1/2}. 
\]
 \end{numthm}

As usual in this context, the first main step in order to prove Theorem \ref{thm:main-wavelet} will be to translate it in terms of certain \emph{Bergman spaces}. Indeed, as noted in several preceding references in the literature \cite{Abreu2012,RamosTilli}, the Wavelet transform with window $\psi_{\beta}$ is directly related to the \emph{Bergman transform of order $2\beta -1$} by 
\[
B_{\alpha}f(z) = \frac{1}{y^{\frac{\alpha}{2} +1}} W_{\psi_{\frac{\alpha+1}{2}}} f(-x,y) = c_{\alpha} \int_0^{+\infty} t^{\frac{\alpha+1}{2}} \widehat{f}(t) e^{i z t} \diff t,
\]
where $z =  x + i y.$ Since this operator defines an isometry between $H^2(\C_+)$ and the space $\mathbf{B}_{\alpha}^2(\C_+)$ composed of analytic functions $g : \C_+ \to \C$ such that 
\[
\|g\|_{\mathbf{B}_{\alpha}^2(\C_+)}^2 = \int_{\C_+} |g(z)|^2 y^{\alpha} \diff \nu(z) < +\infty,
\]
we may work on the latter space instead. By mapping the upper half plane to the unit disk in a standard way, we may rephrase Theorem \ref{thm:main-wavelet} in the context of \emph{Bergman spaces of the unit disk}. We define, to that extent, the function 
\[
\theta_{\alpha}(s) = 1 - \left( 1 + \frac{s}{\pi}\right)^{-\alpha - 1}.
\]
The hyperbolic measure $d\mu$ on the hyperbolic disk is defined hence as 
\[
\mu(\Omega) = \int_{\Omega} (1-|z|^2)^{-2} \diff A(z). 
\]
The deficit now reads
\begin{equation}\label{eqn:BergmanDeficit}
    \delta(f;\Omega,\alpha) = \frac{\norm{f}_{\Berg} \theta_{\alpha}(s) - \int_\Omega |f(z)|^2 (1-|z|^2)^{\alpha + 2} \diff \mu(z)}{\norm{f}_{\Berg} \theta_{\alpha}(s)} , \quad \text{ where we set }\mu(\Om) = s,
\end{equation}
and the asymmetry is defined to be 
\begin{equation}\label{eqn:BergmanAsymmetry}
    \mcalA_{\disk}(\Omega)  = \inf \inset{ \frac{\mu(\Omega \Delta B(x,r))}{\mu(\Omega)} \colon \mu(B(x,r)) = \mu(\Omega) , \ x \in \disk } .
\end{equation}

In these terms, Theorem \ref{thm:main-wavelet} may be stated equivalently in the following form: 

\begin{numthm}\label{thm:BergmanStability}
    Let $\alpha > -1$ and $\Omega \subset \disk$ such that $\mu(\Omega) = s > 0$, and let $f \in \Berg(\disk)$. Then there is an absolute, computable constant $C >0$ such that
    \begin{equation}\label{eqn:BergmanStabilityFunction}
        \inf_{\substack{\abs{c} = \norm{f}_{\Berg} ,\\ \om \in \disk}} \frac{\norm{f-c \ff_{\om}}_{\mathbf{B}_{\alpha}^2}}{\norm{f}_{\mathbf{B}_{\alpha}^2}} \leq C \of{1 + \frac{\alpha+2}{\alpha+1} \left[ \of{1+\frac{s}{\pi}}^{\alpha+1} -1\right]}^{1/2} \delta(f;\Omega,\alpha)^{1/2} ,
    \end{equation}
    where $\ff_{\om}$ is a multiple of the reproducing kernel for $\Berg(\disk)$, explicitly given by 
    \[ \ff_{\om}(z) = \sqrt{\frac{\alpha + 1}{\pi}} \cdot \frac{(1-|\om|^2)^{\frac{\alpha+2}{2}}} {(1 - \overline{\om} z)^{\alpha+2}}.\]
    Moreover, for some explicit constant $K(s,\alpha)$ we also have
     \begin{equation}\label{eqn:BergmanStabilitySet}
         \mcalA_{\disk} (\Omega) \leq K(s,\alpha) \delta(f;\Omega,\alpha)^{1/2}.
     \end{equation}
\end{numthm}

We note that both Theorem \ref{thm:main-wavelet} and Theorem \ref{thm:BergmanStability} are \emph{sharp} in terms of the dependence rate on $\beta$ and $\alpha$, respectively. This statement should be interpreted in the following sense (for the case of Theorem \ref{thm:BergmanStability}, with the other one having an entirely analogous formulation): as $\alpha \to \infty$, we are able to recover the main result in \cite{Inv2} from \eqref{eqn:BergmanStabilityFunction} -- which, as noted in \cite{Inv2} itself, is sharp not only on the order of exponent, but in terms of the best possible exponential dependency on the area of the sets under considerations; we refer the reader also to \cite{GarciaOrtega} where a similar result has been recently obtained for the \emph{positive curvature} case. Furthermore, if we let $\alpha \to -1$, we obtain that, for any fixed $\Omega \subset \disk$, the right-hand side of \eqref{eqn:BergmanStabilityFunction} converges towards a constant that depends only on $\mu(\Omega)$, a behaviour that can be attained for the left-hand side of \eqref{eqn:BergmanStabilityFunction} by choosing a suitable sequence of almost-optimal functions. 

This last part is the main additional cause of difficulty when comparing the results in this manuscript with the ones in \cite{Inv2} and \cite{GarciaOrtega}. Indeed, in \cite{Inv2}, one does not deal with specific parameters, and in \cite{GarciaOrtega}, the only relevant case for polynomial concentration inequalities is that of \emph{integer exponents}, which removes the obstruction of the lack of an endpoint result in Theorems \ref{thm:main-wavelet} and \ref{thm:BergmanStability}. 

In order to deal with this additional problem we need to improve the results and techniques from \cite{Inv2}. As a matter of fact, a main idea in the proof of the main results in \cite{Inv2} was to prove that $\mathbf{u}^*$ and $\mathbf{v}^*$ cannot be equal for points `too close' to the origin, where $\mathbf{v}^*(s) = e^{-s}$ and $\mathbf{u}^*$ denotes the one-dimensional rearrangement of a function $\mathbf{u}(z) = |F(z)|^2 e^{-\pi |z|^2}$. In order to prove the optimal control on the constant in \eqref{eqn:BergmanStabilityFunction}, we need a sort of converse of the principle above: if the point where they intersect is suitably controlled `from above' when we take $\alpha \to -1$ in Theorem \ref{thm:BergmanStability}, we are able to conclude our result. 

On the other hand, this last fact is significantly more challenging than the proof of the main theorems and lemmas in \cite{Inv2}, which are already technically involved. As a matter of fact, one would need a careful asymptotic analysis of the behavior of constants in \cite[Lemma~2.1]{Inv2}. As a first step towards this goal - and consequently towards a \emph{complete} understanding of the geometry of near-extremals to the Nicola-Tilli inequality \cite{NicolaTilli}, we analyze what happens when one replaces the function $u^*$ by a suitably close upper bound given by the Bergman case analogue of Lemma 2.1 in \cite{Inv2}, and control the intersection of the latter function with the Bergman analogue of $\mathbf{v}^*$. By using an underlying \emph{convexity} of the difference of the two comparison functions, we are able to explicitly and sharply estimate the intersection points under consideration, which yields Theorem \ref{thm:BergmanStability}, and also Theorem \ref{thm:main-wavelet} as a result. 

\vspace{2mm}

As a consequence of Theorem \ref{thm:BergmanStability} for $\alpha \to \infty$ and a rescaling argument, we are able to recover the main result in \cite{Inv2}: 

\begin{numthm}[Theorem~1.5~in~\cite{Inv2}]\label{thm:GGRT-main} There is an absolute, computable constant $C>0$ such that, for all measurable sets of finite measure $\Omega \subset \C$ and $F \in \mathcal{F}^2(\C)$, we have 
\begin{equation}\label{eqn:StabilityFockFunction} 
        \inf_{\substack{\abs{c} = \norm{F}_{\Fock} ,\\ z \in \C}} \frac{\norm{F-cF_{z}}_{\Fock}}{\norm{F}_{\Fock} } \leq C \left( e^{|\Omega|} \delta_{\Fock}(F;\Omega)\right)^{1/2}, 
\end{equation}
where 
\begin{equation}\label{eqn:DeficitFock}
\delta_{\Fock} (F;\Omega) = 1 - \frac{\int_{\Omega} |F(z)|^2 e^{-\pi |z|^2} \, \diff z}{(1-e^{-|\Om|})\|F\|_{\Fock}}.
\end{equation}
Moreover, for some universal explicit constant $K(|\Om|)$ we also have 
\[
\mathcal{A}(\Omega) \le K(|\Om|) \delta_{\Fock} (F;\Omega)^{1/2}, 
\]
where we let the Euclidean asymmetry be defined as 
\[
\mathcal{A}(\Om) := \inf \inset{ \frac{|\Omega \Delta B(x,r)|}{|\Omega|} \colon |B(x,r)| = |\Omega| , \ x \in \C } .
\]
\end{numthm}

This is perhaps not surprising, as the main results in \cite{GarciaOrtega} manage to obtain the same result as a consequence of their methods. On the other hand, the fact that our results are also sharp in terms of dependence rate on $\alpha \to -1$ yield a \emph{new} concentration result for Hardy spaces, even in sharp quantitative form. In fact, even the qualitative version of this result, corresponding to the inequality
\begin{equation}
\int_{\Om} |f(z)|^2 (1-|z|^2)^{-1} \, \diff z \le \pi \, \log\left( 1+ \frac{s}{\pi} \right) \cdot \|f\|_{H^2}^2, \quad \text{ for all } f \in H^2(\disk),
\end{equation}
is new. In order to introduce such a result, we say that $f \in H^2(\disk)$ if 
\[
\|f\|_{H^2}^2 = \sup_{r \in (0,1)} \frac{1}{2\pi} \int_0^{2\pi} |f(re^{i \theta})|^2 \, \diff \theta < + \infty. 
\]
It is known that, if $f \in H^2(\disk)$, then it has a measurable finite radial limit $f(e^{i\theta})$ for almost all $\theta \in (0,2\pi]$, for which we have that 
\[
\|f\|_{H^2}^2 = \frac{1}{2\pi} \int_0^{2\pi} |f(e^{i \theta})|^2 \, \diff \theta. 
\]
\begin{numthm}\label{thm:HardyStability} There is an explicitly computable constant $C>0$ such that for all $\Omega \subset \disk$ with $\mu(\Omega) = s$ and all $f \in H^2(\disk)$, we have 
\begin{equation}\label{eqn:StabilityFunctionHardy}
\inf_{\substack{\abs{c} = \norm{f}_{H^2} ,\\ \om \in \disk}} \frac{ \|f - c \textbf{g}_\om\|_{H^2}}{\|f\|_{H^2}} \le C  \cdot \left[ 1 + \log^{1/2} \left( 1 + \frac{s}{\pi}\right)\right] \delta_{H^2}(f;\Omega)^{1/2},
\end{equation}
where 
\begin{equation}\label{eqn:DeficitHardy} 
\delta_{H^2} (f;\Omega) = 1 - \frac{\int_{\Omega} |f(z)|^2 (1-|z|^2)^{-1} \, \diff z}{\pi \log\left( 1 + \frac{s}{\pi}\right) \|f\|_{H^2}^2},
\end{equation}
and $\textbf{g}_\om = \frac{(1-|\om|^2)^{1/2}}{1- \overline{\om} z}$. Moreover, we additionally have that, for an explicit continuous function $N:(0,\infty) \to (0,\infty)$, 
\begin{equation}\label{eqn:AsymmetryHardy} 
\mathcal{A}_{\disk}(\Omega) \le N(s) \cdot \delta_{H^2}(f;\Omega)^{1/2}. 
\end{equation} 
\end{numthm}

This last result may be interpreted as follows: take any $H^2$ function on the unit circle - which may be identified with the space of functions $f \in L^2(-1/2,1/2)$ which have square-summable Fourier coefficients, and whose Fourier coefficients of negative order vanish. Consider its harmonic extension to the unit disk. Theorem \ref{thm:HardyStability} shows not only a \emph{sharp} bound for how much this Poisson extension can be concentrated on a given subset of the unit disk, but it also says that, whenever one is close to achieving that sharp bound, one is close to being a multiple the Poisson kernel itself, and furthermore the order of magnitude of the distance to the set of extremals is \emph{sharp}. 

The rest of this paper is organized as follows: in Section \ref{sec:Prelim} below we set out the groundwork for the proof of our main results, discussing the basic properties of Bergman spaces, the equivalence of Theorems \ref{thm:main-wavelet} and \ref{thm:BergmanStability}, and recalling the most useful notational convention used. In Section \ref{sec:MainProof}, we prove Theorem \ref{thm:BergmanStability}, which itself implies Theorem \ref{thm:main-wavelet}. We then discuss the sharpness of our results in Section \ref{sec:Sharpness}, and in Section \ref{sec:Fock}, we discuss how to obtain the Bargmann--Fock result from \cite{Inv2} as a consequence of our results. Finally, in Section \ref{sec:Hardy}, we prove Theorem \ref{thm:HardyStability}. 

\subsection*{Acknowledgements} The authors are thankful to M.A. Garc\'ia-Ferrero and J. Ortega-Cerd\`a for pointing us to their work \cite{GarciaOrtega} during the preparation of this manuscript and discussing the relationships between their main results and the ones here. We would also like to thank A. Guerra for several comments and suggestions that helped improve the exposition.  The third author acknowledges partial financial support by MPNTP grant 174017 Serbia. 

\section{Preliminaries}\label{sec:Prelim}

\subsection{Bergman spaces on \texorpdfstring{$\C_+$}{C+} and \texorpdfstring{$\disk$}{D} and the Bergman transform}
For every $\alpha>-1$, we define the Bergmann space $\mathbf{B}_{\alpha}^2 (\disk)$ of the unit disk as the Hilbert space of all functions
$f\colon\disk\to \C$ which are holomorphic in $\disk$ and satisfy that
\[
\| f\|_{\mathbf{B}_{\alpha}^2}^2 := \int_\disk |f(z)|^2 (1-|z|^2)^\alpha \diff z <+\infty.
\]
Analogously, the Bergman space of the upper half plane $\mathbf{B}_{\alpha}^2(\C_+)$ is defined as the set of analytic functions in $\C_+$ such that
\[
\|f\|_{\mathbf{B}_{\alpha}^2(\C_+)}^2 = \int_{\C_+} |f(z)|^2 y^{\alpha} \diff \nu(z),
\]
where $y = \text{Im}(z).$ These two spaces defined above do not only share similarities in their definition, but indeed it can be shown that they are \emph{isomorphic:} if one defines
\[
T_{\alpha}f(w) = \frac{2^{\alpha/2}}{(1-w)^{\alpha+2}} f \left(\frac{w+1}{i(w-1)} \right),
\]
then $T_{\alpha}$ maps $\mathbf{B}_{\alpha}^2(\C_+)$ to $\mathbf{B}_{\alpha}^2(\disk)$ as a \emph{unitary isomorphism.} For this reason, dealing with one space or the other is tantamount, an important fact in the proof of the main theorem below. Let us then focus on the case of $\disk$, and thus we abbreviate $\mathbf{B}_{\alpha}^2(\disk) = \mathbf{B}_{\alpha}^2$ from now on. The weighted $L^2$ norm defining this space is induced by the scalar product
\[
\langle f,g\rangle_\alpha := \int_\disk f(z)\overline{g(z)} (1-|z|^2)^\alpha \diff z.
\]
Here and throughout, $\diff z$ denotes the bidimensional Lebesgue measure on $\disk$.

An orthonormal basis of $\mathbf{B}_{\alpha}^2$ is given by the normalized monomials
$ z^n/\sqrt{c_n}$ ($n=0,1,2,\ldots$), where
\[
c_n = \int_\disk |z|^{2n}(1-|z|^2)^\alpha \diff z=
2\pi \int_0^1 r^{2n+1}(1-r^2)^\alpha \diff r=
\frac{\Gamma(\alpha+1)\Gamma(n+1)}{\Gamma(2+\alpha+n)}\pi.
\]
 Hence, a function $f \in \Berg$ satisfies the expansion 
 \begin{equation}\label{eqn:Bergman-exp} 
 f(z) = \sum_{k \ge 0} a_k \frac{z^k}{\sqrt{c_k}},
 \end{equation} 
 where the sequence $\{a_k\}_{k \ge 0} \in \ell^2(\N_{\ge 0})$ with $\|a_k\|_{\ell^2} = 1.$ 
 
 Now, we shall connect the first two subsections above by relating the wavelet transform to Bergman spaces, through the so-called \emph{Bergman transform.} For more detailed information, see, for instance \cite{Abreu} or \cite[Section~4]{Abreu2012}. 

Fix $\alpha > -1.$ Recall that the function $\psi_{\alpha} \in H^2(\C_+)$ satisfies
\[
\widehat{\psi_{\alpha}}(t) = \frac{1}{c_{\alpha}} \charf{[0,+\infty)} t^{\alpha} e^{-t},
\]
where $c_{\alpha} > 0$ is chosen so that $\| \widehat{\psi_{\alpha}} \|_{L^2(\R^+,t^{-1})}^2 =1.$ The \emph{Bergman transform of order $\alpha$} is, as stated in the introduction,  given by
\[
B_{\alpha}f(z) = \frac{1}{s^{\frac{\alpha}{2} +1}} W_{{\psi_{\frac{\alpha+1}{2}}}} f(-x,s) = c_{\alpha} \int_0^{+\infty} t^{\frac{\alpha+1}{2}} \widehat{f}(t) e^{i z t} \diff t.
\]
From this definition, it is immediate that $B_{\alpha}$ defines an analytic function whenever $f \in H^2(\C_+).$ Moreover, it follows directly from the isometric properties of the Wavelet transform that $B_{\alpha}$ is a unitary map between $H^2(\C_+)$ and $\mathbf{B}_{\alpha}^2(\C_+).$ We also note that $B_{\alpha}$ is actually an \emph{isomorphism} between $H^2(\C_+)$ and $\mathbf{B}_{\alpha}^2(\C_+)$ - see, for instance, \cite[Section~2.3]{RamosTilli} for a proof. 

\subsection{Equivalence between Theorem \ref{thm:main-wavelet} and \ref{thm:BergmanStability}} We quickly show that the two versions of our main result, as stated in the introduction, are actually equivalent to one another. 

Indeed, by the definition of the Bergman transform from before, we have that
\[
\int_{\Omega} |W_{\overline{\psi_{\beta}}} f(x,y)|^2 \, \frac{ \diff x \diff y}{y^2} = \int_{\tilde{\Omega}} |B_{\alpha}f(z)|^{2} y^{\alpha} \diff x \diff y,
\]
where $\tilde{\Omega} =\{ z = x + iy\colon -x+iy \in \Omega\}$ and $\alpha = 2\beta - 1.$ On the other hand, we may further apply the map $T_{\alpha}$ above to $B_{\alpha}f;$ this implies that
\[
\int_{\tilde{\Omega}} |B_{\alpha}f(z)|^{2} y^{\alpha} \diff x \diff y = \int_{\Omega'} |T_{\alpha}(B_{\alpha}f)(w)|^2 (1-|w|^2)^{\alpha} \diff w,
\]
where $\Omega'$ is the image of $\tilde{\Omega}$ under the map $z \mapsto \frac{z-i}{z+i}$ on the upper half plane $\C_+.$ Notice that, from this relationship, we have
\begin{align*}
\mu(\Omega') = \int_{\Omega'} (1-|w|^2)^{-2} \diff w  = \int_\disk \charf{\Omega'} \left( \frac{w+1}{i(w-1)} \right) (1-|w|^2)^{-2} \diff w = \frac{\nu(\Omega)}{4}. 
\end{align*}
This shows that, upon applying $B_{\alpha}$ and $T_{\alpha}$, the right-hand side of \eqref{eqn:StabilityFunction} equals that of \eqref{eqn:BergmanStabilityFunction}, and that quantifying over all $\Omega \subset \C_+$ with $\nu(\Omega) = s$ is the same as quantifying over all $\Omega' \subset \disk$ with $\mu(\Omega') = s/4.$ 

Furthermore, it follows from the mapping properties of the Bargmann transform (see, e.g., \cite[Eq.~(2.5)]{RamosTilli}) that $ B_{\alpha}(\pi_z \psi_{\alpha})(w)$ is a multiple of $\left( \frac{1}{w-\overline{z}}\right)^{\alpha + 2}$, which is mapped to a multiple of a function of the form $(1-\overline{\om}z)^{-\alpha-2}$ under the action of $T_{\alpha}$. Since the function $\mathbf{f}_\om$ in the statement of Theorem \ref{thm:BergmanStability} has $\Berg$-norm equal to $1$, it follows at once that by applying once more $T_{\alpha}$ and $B_{\alpha}$ to the left-hand side of \eqref{eqn:StabilityFunction}, one obtains the left-hand side of \eqref{eqn:BergmanStabilityFunction}, as desired. 

\subsection{Notation and notions from [\ref{labelnumber-RamosTilli}]} We fix, for the most part of the paper, the parameter $\alpha > -1.$ Recalling the notation in \cite{RamosTilli}, given $f \in \Berg$ with $\norm{f}_{\Berg} = 1$, we introduce the function
\begin{equation}\label{eqn:BergmanU}
    u(z) \coloneqq \abs{f(z)}^2 (1-\abs{z}^2)^{\alpha+2} ,
\end{equation}
and its distribution function
\begin{equation}
    \rho(t) \coloneqq \mu \of{\inset{u > t}}.
\end{equation}
Its inverse function will be denoted by $u^*(s)$, 
\begin{equation}\label{eqn:BergmanU*}
    \mu \of{\inset{u > u^*(s)}} = s,
\end{equation}
and we shall suppose without loss of generality that this reaches its maximum at $u^*(0) = T = \max_{\disk} u$. We further denote by $v^*(s) = \frac{\alpha+1}{\pi} \left(1 + \frac{s}{\pi} \right)^{-\alpha-2},$ which will play a similar role in the proof below as the function $\mathbf{v}^*(s) = e^{-s}$ does in \cite{Inv2}.

\section{Proof of Theorem \ref{thm:BergmanStability}}\label{sec:MainProof}

\subsection{Function stability}

In this part we address the proof of \eqref{eqn:BergmanStabilityFunction} by means of an estimate on the distribution function $\rho(t)$.


\begin{numlemma}\label{lemma:BergmanSuperLevelSetsEstimate}
There exists $\Tilde{c}_0 >0$ such that for all $c_0 \geq \tilde{c}_0$ there is $C_0 \in (c_0,1)$ and a constant $C > 0$ with the following property. If $f \in \Berg$ is such that $\norm{f}_{\Berg} = 1$ and $a_0^2 \geq C_0$, then
    \begin{equation}\label{eqn:BergmanSuperLevelSetsEstimate}
        \rho(t) \leq \pi \left(\frac{ 1+K_0 (1-a_0^2)}{a_{0}^2}\right) \left( \left( \frac{\frac{\alpha+1}{\pi} a_0^2}{t} \right)^{\frac{1}{\alpha+2}}-1 \right), \quad \forall t \,\, \in \left( \frac{\alpha+1}{\pi}c_0, \frac{\alpha+1}{\pi} \, a_0^2\right) ,
    \end{equation}
    where $K_0 = C / c_0^3$ and $\Tilde{c}_0 \in (0,1)$ is an explicit constant. 
\end{numlemma}

\begin{remark} We note that the constant $\Tilde{c}_0$ in the statement of Lemma \ref{lemma:BergmanSuperLevelSetsEstimate} may be taken to be around $0.6194$. It may likely be improved by further refining the proof below, but since this does not affect the rest of the proofs in this manuscript, and since this also does not seem to dramatically improve any of the main results in this manuscript, we have not attempted to do so here. 
\end{remark}

\begin{proof}
    We begin by setting
    \[\delta^2 =\frac{1}{T}\left(1-a_{0}^2\right) = \frac{1}{T} \left(\frac{\alpha+1}{\pi}-T \right) \frac{\pi}{\alpha+1}.\]
    Denote $\delta_{0}^2 = 1-a_{0}^2$, through a Möbius transformation and a change of phase we may assume that $f(0) = \sqrt{T}$. Then,
    \[\frac{f(z)}{\sqrt{ T }} = 1+R(z) , \quad z \in \mathbb{D},\]
    and we can differentiate to find
    \[
    R(z) = \sum_{k \geq 2} \frac{a_{k}}{\sqrt{c_{k} T}} z^k ,\quad  R'(z) = \sum_{k \geq 2} k\frac{a_{k}}{\sqrt{c_{k} T}} z^{k-1}, \quad R''(z) = \sum_{k \geq 2} k(k-1)\frac{a_{k}}{\sqrt{c_{k} T}} z^{k-2} .
    \]
    Here, we have used the fact that $u$ attains its maximum at $0$ in order to conclude that the $z-$coefficient of $R$ vanishes; see subsection \ref{subsec:SetStability} for more details on this normalization. Using the Cauchy-Schwarz inequality, we come to
    \begin{equation}\label{eqn:BoundsR} \lvert R(z) \rvert \leq \delta g_{1}(\lvert z \rvert)^{1/2} , \quad \lvert R'(z) \rvert \leq \delta g_{2}(\lvert z \rvert)^{1/2} , \quad \lvert R''(z) \rvert \leq \delta g_{3}(\lvert z \rvert)^{1/2} , 
    \end{equation}
    where the functions $g_{i}$ are the following:
    \[ g_{i}(x) =
    \begin{cases}
        \frac{\alpha+1}{\pi}\left[(1-x^2)^{-(\alpha+2)}-1-(\alpha+2)x^2\right], \quad &\text{if} \ i=1, \\
        \frac{\alpha+1}{\pi}(\alpha+2)[-1+(1+(\alpha+2)x^2)(1-x^2)^{-(\alpha+4)}], \quad &\text{if} \ i=2, \\
        \frac{\alpha+1}{\pi}\left[(\alpha+2)(\alpha+3)(1-x^{2})^{-\left(\alpha+6\right)}(2+(\alpha+3)x^{2}(4+(\alpha+4)x^{2}))\right], \quad &\text{if} \ i=3 .
    \end{cases} \]
    Setting $z=re^{i \theta}$, the condition $u(r e^{i\theta}) > t$ implies that
    \begin{equation}\label{eqn:boundary-condition}
    \frac{t}{T} (1-r^2)^{-(\alpha+2)} < 1 + \delta^2 g_{1}(r) + h(r e^{i \theta}),
    \end{equation}
    and $h(z) = 2 \text{Re}(R(z))$ is defined as in \cite{Inv2}. Note that this function satisfies, by definition and the Cauchy-Riemann equations, that $|\nabla h(z)| = 2 |R'(z)|$ and $|D^2 h(z)| = 2\sqrt{2} |R''(z)|,$ which imply that we have 
    \begin{equation}\label{eqn:Bounds-h}
    \lvert h(z) \rvert \leq 2\delta g_{1}(\lvert z \rvert)^{1/2} , \quad \lvert \partial_r h(r e^{i\theta}) \rvert \leq 2\delta g_{2}(r)^{1/2} , \quad \lvert \partial_r^2 h(re^{i\theta}) \rvert \leq 2\sqrt{2} \delta g_{3}(r)^{1/2}. 
    \end{equation}
    If we define
    \[ g_{\theta}(r,\sigma) = \frac{t}{T}(1-r^2)^{-(\alpha+2)} - \delta^2 g_{1}(r) - \sigma h(r e^{i\theta}),\]
    then \eqref{eqn:boundary-condition} may be rewritten as $g_{\theta}(r,1)<1$. At $\sigma=0$, this is a radial function, and by employing the same considerations as in \cite{Inv2}, it defines a function $r_{\sigma}(\theta)$ implicitly via
    \[ g_{\theta}(r_{\sigma}(\theta),\sigma)=1,\]
    which for each $\sigma \in (0,1)$ describes the radius of $E_{\sigma} = \{ r e^{i \theta} \colon 0 \leq r \leq r_{\sigma}(\theta) \}$.

    \vspace{2mm}

    \noindent\textbf{Properties of $E_{\sigma}$}. We can now introduce a few properties of the sets $E_{\sigma}$. In particular, they are bounded and star-shaped: letting $t > t_0 = \frac{\alpha+1}{\pi} c_0$ for some $c_0 \in (0,1)$, we can estimate $\sigma$ by $1$, $h(z)$ by $\lvert R(z) \rvert$ and $\frac{t}{T} \geq t_{0} \frac{\pi}{\alpha+1} = c_0$ to reach
    \begin{equation}\begin{split}\label{eqn:InProofFirstSubtraction}
        g_{\theta}(r,\sigma) & \geq c_{0}(1-r^2)^{-(\alpha+2)} - \delta^2 \frac{\alpha+1}{\pi}\left[ (1-r^2)^{-(\alpha+2)}-1-(\alpha+2)r^2 \right] \\
        &\hspace{5cm}-2 \delta \sqrt{ \frac{\alpha+1}{\pi} \left[ (1-r^2)^{-(\alpha+2)}-1-(\alpha+2)r^2 \right]} .
    \end{split}\end{equation}
    Notice that the first expression in brackets can be bounded by $(1-r^2)^{-(\alpha+2)}$, and imposing that $\delta^2 \frac{\alpha+1}{\pi}$ be small in terms of $c_{0}$ effectively translates into a lower bound on $T$, giving rise to the condition $T > T_0$, or equivalently $a_0^2 > C_0$. Write
    \[ \delta^2 \frac{\alpha+1}{\pi} = \frac{1}{T} \left[ \frac{\alpha+1}{\pi} - T \right]  = \frac{\delta_{0}^2}{T} \frac{\alpha+1}{\pi}. \]
    The first subtraction in \eqref{eqn:InProofFirstSubtraction} can be bounded from below by requiring $\delta_{0}$ to be small in terms of $c_{0}$. The second subtraction can be dealt with similarly: we can estimate the function in brackets inside the square root from above by $(1-r^2)^{-(\alpha+2)}$, and we can therefore bound
    \begin{align*}
        g_{\theta}(r,\sigma) & \geq \left( c_{0} - \frac{\delta_{0}^2}{T} \frac{\alpha+1}{\pi} -2 \frac{\delta_{0}}{\sqrt{ T }} \sqrt{ \frac{\alpha+1}{\pi}} \right) (1-r^2)^{-(\alpha+2)}.
    \end{align*}
    Assume now that $\delta_{0} < c_{0}/8$. This means that,
    \[\frac{T\pi}{\alpha+1} > 1- \frac{c_{0}^2}{16} , \quad \text{i.e.} \quad \frac{1}{a_0^2} < \frac{1}{1- \frac{c_{0}^2}{16}} < \frac{16}{15}, \]
    and we reach
    \begin{align*}
        g_{\theta}(r,\sigma) & \geq \left( c_{0} - \frac{c_{0}}{16} \frac{\alpha+1}{T\pi} - \frac{c_{0}}{4} \sqrt{ \frac{\alpha+1}{T\pi}} \right) (1-r^2)^{-(\alpha+2)} \geq \frac{2}{3} c_{0} (1-r^2)^{-(\alpha+2)}.
    \end{align*} 
    We can finally bound the radius $r_{\sigma}$ of $E_{\sigma}$ defined by $g_{\theta}(r_{\sigma},\sigma)=1$ by the solutions to
    \[ \frac{2}{3} c_{0}(1-r^2)^{-(\alpha+2)} = 1 ,\]
    reaching
    \[ r_{\sigma}^2 \leq 1-\left( \frac{3}{2c_{0}} \right)^{-\frac{1}{a+2}} .\]
    Moreover, when $\sigma =0$, we have a sharper bound for $r_{0}$,
    \begin{equation}\label{eqn:BergmanBoundr0}
        r_{0}^2 \leq 1- \left[ \frac{c_{0}-\delta_{0}^2}{a_{0}^2-\delta_{0}^2} \right]^{\frac{1}{\alpha+2}} .
    \end{equation}

    To show that $E_{\sigma}$ is star-shaped, we can argue similarly. We just need to check that $g_{\theta}(r,\sigma)$ grows as $r$ increases. We differentiate with respect to $r$ and use \eqref{eqn:Bounds-h} to find 
    \begin{align*}
        \frac{ \mathrm{d}  }{ \mathrm{d} r } g_{\theta}(r,\sigma) &\ge 2\frac{t}{T}(\alpha+2)r(1-r^2)^{-(\alpha+3)} - 2\delta_{0}^2 \frac{\alpha+1}{\pi T} (\alpha+2)r ((1-r^2)^{-(\alpha+3)}-1)\\
        &\hspace{4cm}- 2e^{i\theta}\sigma \delta_{0} \sqrt{ \frac{\alpha+1}{\pi T} } \left[ 2(\alpha+2) r ((1-r^2)^{-(\alpha+3)}-1) \right]^{1/2} .
    \end{align*}
    The same arguments as in the previous discussion yield
    \[\frac{ \mathrm{d}  }{ \mathrm{d} r } g_{\theta}(r,\sigma) \geq 2r(\alpha+2)\left( c_{0} - {\delta_{0}^2} \frac{\alpha+1}{\pi T} -2 {\delta_{0}} \sqrt{ \frac{\alpha+1}{\pi T}} \right) (1-r^2)^{-(\alpha+3)}, \]
    a lower bound that remains positive as long as we let $\delta_{0} < c_{0}/8$ again.

    \vspace{2mm}
    
    \noindent\textbf{Strategy to show \eqref{eqn:BergmanSuperLevelSetsEstimate}}. The condition that $u(z) \geq t$ implies that $z \in E_{1}$, meaning that
    \[ \mu(\{ u > t \}) \leq \mu(E_{1}) =\frac{1}{2} \int_{0}^{2 \pi} \int_{0}^{r_{\sigma}(\theta)} \frac{r}{(1-r^2)^2} \, \mathrm{d}r  \, \mathrm{d}\theta , \quad \text{at} \ \sigma=1 .\]
    This suggests considering $F(\sigma) = \mu(E_{\sigma})$,
    \[ F(\sigma) = \int_{0}^{2 \pi} \int_{0}^{r_{\sigma}(\theta)} \frac{r}{(1-r^2)^2} \, \mathrm{d}r  \, \mathrm{d}\theta = \frac{1}{2} \int_{0}^{2\pi} \frac{r_{\sigma}(\theta)^2}{1-r_{\sigma}(\theta)^2} \, \mathrm{d}\theta .\]
    Through Taylor's formula, we get
    \[ F(s)= F(0)+F'(0)s+F''(\sigma) \frac{s^2}{2} ,\]
    for some $\sigma \in (0,s)$. Now, note that we have
    \[
    F'(\sigma) = \int_0^{2\pi} \frac{r_\sigma (\theta)}{(1-r_{\sigma}(\theta)^2)^2} \, \partial_{\sigma} r_{\sigma}(\theta) \, \diff \theta.
    \]
    In order to compute the derivative of $r_{\sigma}$ with respect to the parameter $\sigma,$ we differentiate implicitly the condition defining $r_{\sigma}(\theta)$, obtaining
    \begin{equation}\label{eqn:derivative-r-sigma}
    r_{\sigma}' = \frac{h(r_{\sigma}e^{i \theta})}{2 r_{\sigma} \frac{t}{T} (\alpha+2)(1-r_{\sigma}^2)^{-(\alpha+3)} - 2 r_{\sigma} \delta_{0}^2 \frac{\alpha+1}{\pi T} (\alpha+2)((1-r_{\sigma}^2)^{-(\alpha+3)}-1) -\sigma e^{i\theta}\partial_{r}h(r_{\sigma}e^{i\theta})}.
    \end{equation}
    At $\sigma = 0$, we obtain that $\partial_\sigma r_\sigma \Big|_{\sigma = 0}$ is a radial multiple of $h(r_0 e^{i\theta})$. Hence, 
    \[
    F'(0) = \varphi(r_0) \int_0^{2\pi} h(r_0 e^{i\theta}) \, \diff \theta.
    \]
    Thanks to the harmonicity of $h$ and the fact that $h(0) = 0$, we conclude that $F'(0) = 0$, which reduces the situation to
    \begin{align}\label{eqn:BergmanAreaTaylor}
        F(s) = F(0)+F''(\sigma) \frac{s^2}{2}.
    \end{align}

    \vspace{2mm}
    
    \noindent\textbf{Estimates on $r_{\sigma}$}. We now wish to exploit this by transferring the estimates for $r_{\sigma}$ and its derivatives with respect to $\sigma$ through to $\sigma=0$. Using that $\delta_{0} < c_{0}/8$ on the $\delta^2$ term in the denominator on \eqref{eqn:derivative-r-sigma}, together with \eqref{eqn:Bounds-h}, we have
    \[ r_{\sigma}' \leq \frac{2 \delta_{0} \sqrt{ \frac{\alpha+1}{\pi T} } [(1-r_{\sigma}^2)^{-(\alpha+2)}-1-(\alpha+2)r_{\sigma}^2]^{1/2} }{r_{\sigma} (\alpha+2) c_{0}(1-r_{\sigma}^2)^{-(\alpha+3)} - 2 \delta_{0} \sqrt{ \frac{\alpha+1}{\pi T} } [(\alpha+2)(1-r_{\sigma}^2)^{-(\alpha+4)}(1+(\alpha+2)r_{\sigma}^2 -(1-r_{\sigma}^2)^{\alpha+4})]^{1/2} } .\]
    The term in brackets on the right hand side can be rewritten and bounded as
    \[ (\alpha+2)[-1+(1+(\alpha+2)r_{\sigma}^2)(1-r_{\sigma}^2)^{-(\alpha+4)}] \leq (\alpha+2)[-1+(1-r_{\sigma}^2)^{-2(\alpha+3)}] ,\]
    where we simply used $1+(\alpha+2)r_{\sigma}^2 \leq (1-r_{\sigma}^2)^{-(\alpha+2)}$. Now we use the estimate
    \[ -1+(1-r_{\sigma}^2)^{-2(\alpha+3)} \leq  2(\alpha+3) r_{\sigma}^2 (1-r_{\sigma}^2)^{-2(\alpha+3)} ,\]
    which is equivalent to $1-2(\alpha+3)r_{\sigma}^2 \leq (1-r_{\sigma}^2)^{-2(\alpha+3)}$, and valid in our case. Indeed, it suffices to check that both sides are equal at $r_{\sigma}=0$ and that, as functions of $r_{\sigma}$, the derivative of the left-hand side is bounded by that of the right-hand side, which is direct. We then reach
    \[ (\alpha+2)[-1+(1+(\alpha+2)r_{\sigma}^2)(1-r_{\sigma}^2)^{-(\alpha+4)}] \leq 2(\alpha+2)(\alpha+3)r_{\sigma}^2(1-r_{\sigma}^2)^{2(\alpha+3)}. \]
    Upon taking the square root in as the denominator, we find
    \[  r_{\sigma}' \leq \frac{2 \delta_{0} \sqrt{ \frac{\alpha+1}{\pi T} } }{ \left[ c_{0} - 2 \delta_{0} \sqrt{ \frac{\alpha+1}{\pi T} } \sqrt{ \frac{2(\alpha+3)}{\alpha+2} }\right] } \cdot \frac{\left[(1-r_{\sigma}^2)^{-(\alpha+2)}-1-(\alpha+2)r_{\sigma}^2\right]^{1/2}}{(\alpha+2)r_{\sigma}^2(1-r_{\sigma}^2)^{-(\alpha+3)}} r_{\sigma}. \]
    The rightmost fraction on the right-hand side is bounded by $1$. Indeed, this can be checked to be true whenever 
    \[ (1-x^2)^{-(\alpha+2)}-1-(\alpha+2)x^2 \leq (\alpha+2)^2 x^4 (1-x^2)^{-2(\alpha+3)} , \quad x = r_{\sigma}, \]
    which can be easily verified for all $x \geq 0$ by means of a series expansion and coefficient comparison. Then, the remaining fraction is bounded by
    \[ \frac{2\delta_{0} \sqrt{ \frac{\alpha+1}{\pi T} }}{c_{0}-4\delta_{0} \sqrt{ \frac{\alpha+1}{\pi T} }} \leq \frac{2\delta_{0} \sqrt{ \frac{C}{C-1} }}{c_{0}\left(1-4\frac{1}{C}\sqrt{ \frac{C}{C-1} }\right)} \leq \frac{4\delta_{0}}{c_{0} \left( 1-\frac{8}{C} \right)} ,\]
    whenever $\delta_{0}< c_{0}/C$. Choosing $C=16$, for instance, we reach
    \[ r_{\sigma}' \leq K \frac{\delta_{0}}{c_{0}}r_{\sigma}, \quad \text{where} \ K=8. \]
    For the second derivative, we differentiate $r_{\sigma}'$ with respect to $\sigma$ to obtain
    \begin{align*}
        r_{\sigma}'' &= \frac{e^{i\theta} \partial_{r} h(r_{\sigma}e^{i \theta})}{2 r_{\sigma} \frac{t}{T} (\alpha+2)(1-r_{\sigma}^2)^{-(\alpha+3)} - 2 r_{\sigma} \delta_{0}^2 \frac{\alpha+1}{\pi T} (\alpha+2)((1-r_{\sigma}^2)^{-(\alpha+3)}-1) -\sigma e^{i\theta}\partial_{r}h(r_{\sigma}e^{i\theta})} r_{\sigma}' \\
        &- \frac{2 \frac{t}{T} (\alpha+2) r_{\sigma}' \left[ (1-r_{\sigma}^2)^{-(\alpha+3)}+2r_{\sigma}^2(\alpha+3)(1-r_{\sigma}^2)^{-(\alpha+4)} \right]}{\left[ 2 r_{\sigma} \frac{t}{T} (\alpha+2)(1-r_{\sigma}^2)^{-(\alpha+3)} - 2 r_{\sigma} \delta_{0}^2 \frac{\alpha+1}{\pi T} (\alpha+2)((1-r_{\sigma}^2)^{-(\alpha+3)}-1) -\sigma e^{i\theta}\partial_{r}h(r_{\sigma}e^{i\theta}) \right]^2} h(r_{\sigma }e^{i\theta})  \\
        &+ \frac{2 \delta_{0}^2 \frac{\alpha+1}{\pi T} (\alpha+2) \left[\left((1-r_{\sigma}^2)^{-(\alpha+3)}-1)r_{\sigma}' + r_{\sigma}(2r_{\sigma} r_{\sigma}' (\alpha+3) (1-r_{\sigma}^2)^{-(\alpha+4)}\right)\right]}{\left[ 2 r_{\sigma} \frac{t}{T} (\alpha+2)(1-r_{\sigma}^2)^{-(\alpha+3)} - 2 r_{\sigma} \delta_{0}^2 \frac{\alpha+1}{\pi T} (\alpha+2)((1-r_{\sigma}^2)^{-(\alpha+3)}-1) -\sigma e^{i\theta}\partial_{r}h(r_{\sigma}e^{i\theta}) \right]^2} h(r_{\sigma }e^{i\theta})\\
        &+ \frac{e^{i\theta} \left[ \partial_{r}h(r_{\sigma}e^{i\theta}) + \sigma e^{i\theta} r_{\sigma}' \partial_{r}^{2} h(r_{\sigma}e^{i\theta}) \right]}{\left[ 2 r_{\sigma} \frac{t}{T} (\alpha+2)(1-r_{\sigma}^2)^{-(\alpha+3)} - 2 r_{\sigma} \delta_{0}^2 \frac{\alpha+1}{\pi T} (\alpha+2)((1-r_{\sigma}^2)^{-(\alpha+3)}-1) -\sigma e^{i\theta}\partial_{r}h(r_{\sigma}e^{i\theta}) \right]^2} h(r_{\sigma}e^{i\theta}) \\
        & =: T_1 + T_2 +T_3 + T_4. 
    \end{align*}
    We address each term separately.

    \vspace{2mm}

    \noindent\underline{\textsc{First term}}: This can be estimated similarly to how we tackled $r_{\sigma}'$. Indeed, by using \eqref{eqn:Bounds-h} again, we see that
    \begin{align*}
        &|T_1| \le \frac{2 \delta_{0} \sqrt{ \frac{\alpha+1}{\pi T} } [(\alpha+2)(1-r_{\sigma}^2)^{-(\alpha+4)}(1+(\alpha+2)r_{\sigma}^2 -(1-r_{\sigma}^2)^{\alpha+4})]^{1/2}}{\left(c_{0}-4\delta_{0}\sqrt{ \frac{\alpha+1}{\pi T} }\right) r_{\sigma} (\alpha+2) (1-r_{\sigma}^2)^{-(\alpha+3)}} r_{\sigma}'\\
        &\hspace{2cm}\leq \frac{4 \delta_{0} \sqrt{ \frac{\alpha+1}{\pi T}}}{c_{0}-4\delta_{0}\sqrt{ \frac{\alpha+1}{\pi T} }} \cdot \frac{ (1-r_{\sigma}^2)^{-(\alpha+3)}}{(1-r_{\sigma}^2)^{-(\alpha+3)}} \leq K \frac{\delta_{0}}{c_{0}} \cdot K\frac{\delta_{0}}{c_{0}} r_{\sigma} = C \frac{\delta_{0}^2}{c_{0}^2} r_{\sigma},
    \end{align*}
    where $C$ can be taken as $8^{2}$ and we assume $\delta_{0} < c_{0}/16$ as before.

    \vspace{2mm}

    \noindent\underline{\textsc{Second term}}: The same estimate from before allows us to bound this term from above by
    \[ |T_2| \le \frac{4 C \frac{\delta_{0}^2}{c_{0}} \sqrt{ \frac{\alpha+1}{\pi T} } (\alpha+2) r_{\sigma} \left[(1-r_{\sigma}^2)^{-(\alpha+3)}+2r_{\sigma}^2(\alpha+3)(1-r_{\sigma}^2)^{-(\alpha+4)}\right] \left[ (1-r_{\sigma}^2)^{-(\alpha+2)}-1-(\alpha+2)r_{\sigma}^2 \right]^{1/2} }{ \left[ \left( c_{0}-4\delta_{0}\sqrt{ \frac{\alpha+1}{\pi  T} } \right) r_{\sigma} (\alpha+2) (1-r_{\sigma}^2)^{-(\alpha+3)}\right]^2} .\]
    Rewriting, the upper bound on the right-hand side above becomes
    \[ \frac{\frac{\delta_{0}^2}{c_{0}} C\sqrt{ \frac{\alpha+1}{\pi T} }r_{\sigma}}{\left( c_{0}-4\delta_{0}\sqrt{ \frac{\alpha+1}{\pi  T} } \right)^2} \cdot \frac{\left[(1-r_{\sigma}^2)^{-(\alpha+3)}+2r_{\sigma}^2(\alpha+3)(1-r_{\sigma}^2)^{-(\alpha+4)}\right] \left[ (1-r_{\sigma}^2)^{-(\alpha+2)}-1-(\alpha+2)r_{\sigma}^2 \right]^{1/2} }{ r_{\sigma}^2 (\alpha+2) (1-r_{\sigma}^2)^{-2(\alpha+3)}} . \]
    We now claim that 
    \[
    \frac{\left[(1-r_{\sigma}^2)^{-(\alpha+3)}+2r_{\sigma}^2(\alpha+3)(1-r_{\sigma}^2)^{-(\alpha+4)}\right] \left[ (1-r_{\sigma}^2)^{-(\alpha+2)}-1-(\alpha+2)r_{\sigma}^2 \right]^{1/2} }{ r_{\sigma}^2 (\alpha+2) (1-r_{\sigma}^2)^{-2(\alpha+3)}} \le 3
    \]
    for all $\alpha>-1$. We can check this by estimating $(1-r_{\sigma}^2)^{-(\alpha+3)} \leq (1-r_{\sigma}^2)^{-(\alpha+4)}$, which results in checking the bound
    \[ \frac{\left[1+2r_{\sigma}^2(\alpha+3)\right] \left[ (1-r_{\sigma}^2)^{-(\alpha+2)}-1-(\alpha+2)r_{\sigma}^2 \right]^{1/2} }{ r_{\sigma}^2 (\alpha+2) (1-r_{\sigma}^2)^{-(\alpha+2)}} \leq 3 ,\]
    and this can be done by using power series. Indeed, by distributing the first factor we may follow the same strategy as in the bound for $r_{\sigma}'$, coupled with the fact that $(\alpha+3)/(\alpha+2) \leq 2$ for $\alpha > -1$ and the expansion of the term $(1-r_{\sigma}^2)^{-(\alpha+2)}$.
    This, and the previous arguments, leads us to the upper bound
    $ C\frac{\delta_{0}^2}{c_{0}^3}r_{\sigma} $
    for $T_2$. 

    \vspace{2mm}

    \noindent \underline{\textsc{Third term}}: We once more employ the same argument in order to bound the denominator in $T_3$. This yields that we need to bound
    \[ \frac{2 \delta_{0}^2 \frac{\alpha+1}{\pi T} (\alpha+2) r_{\sigma}' \left[\left((1-r_{\sigma}^2)^{-(\alpha+3)}-1\right) + 2r_{\sigma}^2  (\alpha+3) (1-r_{\sigma}^2)^{-(\alpha+4)}\right]}{\left[ \left( c_{0}-4\delta_{0}\sqrt{ \frac{\alpha+1}{\pi  T} } \right) r_{\sigma} (\alpha+2) (1-r_{\sigma}^2)^{-(\alpha+3)}\right]^2} h(r_{\sigma }e^{i\theta}) ,\]
    which in turn may be bounded by 
    \[ \frac{4 K \delta_{0}^3 \left(\frac{\alpha+1}{\pi T}\right)^{3/2} \left( \frac{\delta_{0}}{c_{0}} \right) r_{\sigma}}{\left( c_{0}-4\delta_{0}\sqrt{ \frac{\alpha+1}{\pi  T} } \right)^2} \cdot \frac{ (\alpha+2)  \left[\left((1-r_{\sigma}^2)^{-(\alpha+3)}-1\right) + 2r_{\sigma}^2  (\alpha+3) (1-r_{\sigma}^2)^{-(\alpha+4)}\right]}{\left[ r_{\sigma} (\alpha+2) (1-r_{\sigma}^2)^{-(\alpha+3)}\right]^2} \sqrt{ \frac{\pi}{\alpha+1} }g_{1}(r_{\sigma})^{1/2} .\]
    The factor 
    \[
    \frac{ (\alpha+2)  \left[\left((1-r_{\sigma}^2)^{-(\alpha+3)}-1\right) + 2r_{\sigma}^2  (\alpha+3) (1-r_{\sigma}^2)^{-(\alpha+4)}\right]}{\left[ r_{\sigma} (\alpha+2) (1-r_{\sigma}^2)^{-(\alpha+3)}\right]^2} \sqrt{ \frac{\pi}{\alpha+1} }g_{1}(r_{\sigma})^{1/2}
    \]
    is bounded by $3$ for all $\alpha>-1$ since it is bounded by the corresponding function in the analysis of $T_2$, and thus $T_3$ is bounded by $K \frac{\delta_{0}^3}{c_{0}^3} r_{\sigma} ,$
    where $\delta_{0}<c_{0} /16$ and $K$ is a computable constant.

    \vspace{2mm}

    \noindent \underline{\textsc{Fourth term}}: Again, we employ the same strategies as we did in the previous parts, so as to find the upper bound
    \begin{align*}
        &|T_4| \le \frac{\delta^2 \of{\frac{\alpha+1}{\pi}}^2 \abs{(1-r_{\sigma}^2)^{-(\alpha+2)}-1-(\alpha+2)r_{\sigma}^2}^{1/2} \abs{-1+(1+(\alpha+2)r_{\sigma}^2)(1-r_{\sigma}^2)^{-(\alpha+4)}}^{1/2}}{\left[ 2 r_{\sigma} \frac{t}{T} (\alpha+2)(1-r_{\sigma}^2)^{-(\alpha+3)} - 2 r_{\sigma} \delta_{0}^2 \frac{\alpha+1}{\pi T} (\alpha+2)((1-r_{\sigma}^2)^{-(\alpha+3)}-1) -\sigma e^{i\theta}\partial_{r}h(r_{\sigma}e^{i\theta}) \right]^2} \\
        &\hspace{0.5cm}+ \frac{K \delta^2 \of{\frac{\alpha+1}{\pi}}^2 \frac{\delta_0}{c_0} r_{\sigma} \abs{(1-r_{\sigma}^2)^{-(\alpha+2)}-1-(\alpha+2)r_{\sigma}^2}^{1/2}}{\left[ 2 r_{\sigma} \frac{t}{T} (\alpha+2)(1-r_{\sigma}^2)^{-(\alpha+3)} - 2 r_{\sigma} \delta_{0}^2 \frac{\alpha+1}{\pi T} (\alpha+2)((1-r_{\sigma}^2)^{-(\alpha+3)}-1) -\sigma e^{i\theta}\partial_{r}h(r_{\sigma}e^{i\theta}) \right]^2} \\
        &\hspace{0.6cm}\cdot\abs{(\alpha+2)(\alpha+3)(1-x^{2})^{-\left(\alpha+6\right)}(2+(\alpha+3)x^{2}(4+(\alpha+4)x^{2}))}^{1/2} .
    \end{align*}
    Using the same techniques as for estimating the previous terms, we eventually reach
    \[ C_{1}\frac{\delta_{0}^2}{c_{0}^2}r_{\sigma} + C_{2} \frac{\delta_{0}^3}{c_{0}^3} r_{\sigma} \leq C \frac{\delta_{0}^2}{c_{0}^3} r_{\sigma .}\]
    Combining all these estimates, we reach
    \[ r_{\sigma}' \leq C \frac{\delta_{0}}{c_{0}}r_{\sigma}, \quad r_{\sigma}'' \leq C \frac{\delta_{0}^2}{c_{0}^3} r_{\sigma} ,\]
    where $C$ is a universal, computable constant. 

    \vspace{2mm}

    \noindent\textbf{Concluding the proof}. Coming back to \eqref{eqn:BergmanAreaTaylor}, and using the bounds for $r_{\sigma}'$ and $r_{\sigma}''$ that we just found, we can estimate
    \[ F''(\sigma) \leq \int_{0}^{2\pi} \left[ (3 (K')^2 + K'') \frac{r_{\sigma}^4}{(1+r_{\sigma}^2)^2} + 4(K')^2 \frac{r_{\sigma}^6}{1+r_{\sigma}^2} \right] \, \mathrm{d}\theta \leq C \frac{\delta_{0}^2}{c_{0}^3} F(\sigma) , \]
    where $C$ is a positive, computable constant, and $K'$ and $K''$ are simply the constants we found such that $r_{\sigma}' \leq K' r_{\sigma}$ and $r_{\sigma}'' \leq K'' r_{\sigma}$. Now we need to compare $F(\sigma)$ to $F(0)$ for $\sigma \in (0,1)$. For this, start by noting that
    \[ \log r_{\sigma} - \log r_{0} \leq \int_{0}^{\sigma} \frac{\lvert r_{s}' \rvert}{r_{s}} \, \mathrm{d}s \leq \sigma \frac{\delta_{0}}{c_{0}} \leq \log \sqrt{ 2 } \]
    whenever $\delta < \frac{c_{0}}{16}$. Hence $r_{\sigma}^2 \leq 2r_{0}^2 \leq 4 r_{\sigma}^2$. We then find
    \[ F(s) \leq \left( 1+C \frac{\delta_{0}^2}{c_{0}^3} s^2 \right) F(0),\]
    and therefore
    \begin{equation}\label{eqn:estimateonE1}
        \mu(E_{1}) = F(1) \leq \pi \left( 1+C \frac{\delta_{0}^2}{c_{0}^3} \right) \frac{r_{0}^2}{1-r_{0}^2} .
    \end{equation}
    Now $r_{0}$ is defined via
    \[ 1=\frac{t}{T}(1-r_{0}^2)^{-(\alpha+2)}  -\delta_{0}^2 \left( \frac{\alpha+1}{\pi T} \right) \left[ (1-r_{0}^2)^{-(\alpha+2)} -1-(\alpha+2)r_{0}^2 \right] .\]
    We can rephrase it so that it reads
    \[ \left( \left( \frac{T}{t} \right)^{\frac{1}{\alpha+2}}-1 \right) = (1-r_{0}^2)^{-1} \left[ 1+ \delta_{0}^2 \left( \frac{\alpha+1}{\pi T} \right)((1-r_{0}^2)^{-(\alpha+2)}-1-(\alpha+2)r_{0}^2) \right]^{-\frac{1}{\alpha+2}}-1 .\]
    We now claim that the right-hand side can be bounded from below by $(1-\delta_{0}^2) \frac{r_{0}^2}{1-r_{0}^2}$ for all $\alpha>-1$ provided that $c_{0}$ is sufficiently close to $1$. 

    To show this, we remark that by using the series representation for $(1-t^2)^{-(\alpha+2)}$, we can rephrase the inequality
    \begin{equation}\label{eqn:lower-bound-Prelim}
    (1-r_{0}^2)^{-1} \left[ 1+ \delta_{0}^2 \left( \frac{\alpha+1}{\pi T} \right)((1-r_{0}^2)^{-(\alpha+2)}-1-(\alpha+2)r_{0}^2) \right]^{-\frac{1}{\alpha+2}}-1 \geq (1-\delta_{0}^2) \frac{r_{0}^2}{1-r_{0}^2} 
    \end{equation}
    as
    \begin{equation}\label{eqn:DesiredBound-r_0}
    \sum_{n \geq 1} \frac{(\alpha+3)_{(n)}}{(n+1)!} r_{0}^{2n} \left( \frac{1}{a_{0}^2} -\delta_{0}^{2n} \right) \leq 1,
    \end{equation}
    where we let $a_{(n)} \coloneqq a \cdot (a-1) \cdots (a+n-1)$. Indeed, \eqref{eqn:lower-bound-Prelim} rewrites easily as 
    \[
    1 + \delta_0^2 \left( \frac{\alpha+1}{\pi T}\right) \cdot \left( (1-r_0^2)^{-(\alpha+2)} - 1 - (\alpha+2)r_0^2\right) \le \left(1- \delta_0^2 r_0^2\right)^{-(\alpha+2)}.
    \]
    Using the aforementioned series expansion of $(1-t^2)^{-(\alpha+2)}$ on both sides and rearranging, together with the fact that $T = \frac{\alpha+1}{\pi} a_0^2$, we get that \eqref{eqn:lower-bound-Prelim} is equivalent to 
    \[
    a_0^{-2} \sum_{n \ge 2} \frac{(\alpha+2)_{(n)}}{n!} r_0^{2n} \le (\alpha+2)r_0^2 + \sum_{n \ge 2} \frac{(\alpha+2)_{(n)}}{n!} r_0^{2n} \delta_0^{2n}. 
    \]
    Cancelling out the respective terms, we obtain \eqref{eqn:DesiredBound-r_0}, as claimed. 
    
    We then move on to show \eqref{eqn:DesiredBound-r_0}. The definition of $r_{0}$ can be read in terms of its series description as
    \begin{equation}\label{eqn:Definition-r_0Alternative} 1= \frac{t}{T}(1+(\alpha+2)r_{0}^2) + \left( \frac{t}{T} - \frac{\delta_{0}^2}{a_{0}^2} \right) \sum_{n \geq 2} \frac{(\alpha+2)_{(n)}}{n!} r_{0}^{2n} .
    \end{equation} 
    After we plug \eqref{eqn:Definition-r_0Alternative} into \eqref{eqn:DesiredBound-r_0}, we eventually reach that the inequality we wish to prove rewrites as 
    \begin{equation}\label{eqn:DesiredBound-2}
    \sum_{n \geq 2} \frac{(\alpha+3)_{(n)}}{(n+1)!} \left[ \frac{1}{a_{0}^2}-\delta_{0}^{2n}-\frac{(n+1)(\alpha+2)}{\alpha+2+n}\left( \frac{t}{T}-\frac{\delta_{0}^2}{a_{0}^2} \right) \right] r_{0}^{2n} \leq \frac{t}{T} + \left[ \frac{t}{T}(\alpha+2) -\frac{\alpha+3}{2}\left( \frac{1}{a_{0}^2}-\delta_{0}^2 \right) \right] r_{0}^2 .
    \end{equation}
    Furthermore, \eqref{eqn:DesiredBound-2} can be reduced to checking that
    \[ \sum_{n \geq 2} \frac{(\alpha+3)_{(n)}}{(n+1)!} \left[ 1- \frac{(n+1)(\alpha+2)}{\alpha+2+n} (c_{0}-\delta_{0}^2) \right] r_{0}^{2n} \leq c_{0} + \left[c_{0}(\alpha+2)-\frac{\alpha+3}{2}(1-a_{0}^2 \delta_{0}^2)\right] r_{0}^2 .\]
    In order to prove this, we first note the estimates
    \[ \frac{(n+1)(\alpha+2)}{\alpha+2+n} \geq 1, \quad (n+1)! \geq n! \quad \text{and} \quad 0 \leq a_{0}^2 \delta_{0}^2,\]
    which are valid for all $n \in \N$ and $\alpha > -1.$ Upon using these inequalities, it suffices to prove 
    \[ \left[ 1- c_{0}+\delta_{0}^2 \right] \sum_{n \geq 2} \frac{(\alpha+3)_{(n)}}{n!} r_{0}^{2n} \leq c_{0} + \left[c_{0}(\alpha+2)-\frac{\alpha+3}{2}\right] r_{0}^2 . \]
    Now we use the fact that the left-hand side corresponds to the series expression of $(1-r_{0}^2)^{-(\alpha+3)}$ except for the first two terms, whereby we equivalently write the desired assertion as
    \[ \left[ 1- c_{0}+\delta_{0}^2 \right] (1-r_{0}^2)^{-(\alpha+3)} \leq 1+\delta_{0}^2 + \left[(1+\delta_{0}^2)(\alpha+2)+1+\delta_{0}^2-c_{0}-\frac{\alpha+3}{2}\right] r_{0}^2 . \]
    We may observe that the coefficient of $r_{0}^2$ in brackets in the right-hand side is actually positive, as it simply amounts to 
    \[
    (\alpha+3)\left( \frac{1}{2} + \delta_0^2\right) - c_0 > 0,
    \]
    which is true for $\alpha > -1$ as long as $c_0 < 1.$ We thus further reduce matters to showing
    \[ (1-r_{0}^2)^{-(\alpha+3)} \leq \frac{1+\delta_{0}^2}{1+\delta_{0}^2-c_{0}} ,\]
    which on the other hand is equivalent to proving that
    \begin{equation}\label{eqn:upper-Bound-r_0-wanted}
    r_{0}^2 \leq 1-\left( \frac{1+\delta_{0}^2}{1+\delta_{0}^2-c_{0}} \right)^{-\frac{1}{\alpha+3}} .
    \end{equation}
    We now impose that the right-hand side of \eqref{eqn:upper-Bound-r_0-wanted} should be at least the bound we already have for $r_{0}^2$ from \eqref{eqn:BergmanBoundr0}, as this is clearly enough to conclude. This, on the other hand, is ensured as long as $c_0$ is sufficiently close to $1$. Indeed, write
    \[ 1-\left( \frac{c_{0}-\delta_{0}^2}{a_{0}^2-\delta_{0}^2} \right)^{\frac{1}{\alpha+2}} \leq 1-\left( \frac{1+\delta_{0}^2}{1+\delta_{0}^2-c_{0}} \right)^{-\frac{1}{\alpha+3}} ,\]
    which we can ensure by checking that
    \[ \left( \frac{c_{0}-\delta_{0}^2}{a_{0}^2-\delta_{0}^2} \right)^{2} \geq \frac{1+\delta_{0}^2-c_{0}}{1+\delta_{0}^2} .\]
    Upon using the properties on $\delta_{0}^2$ and $a_{0}^2$ that we have, we reach
    \[ \left( c_{0}-\frac{c_{0}^2}{16^2} \right)^2 \geq 1-\frac{c_{0}}{1+\frac{c_{0}^2}{16^2}} ,\]
    which is satisfied whenever $c_{0} \geq \tilde{c}_{0}$, and $\tilde{c}_{0}$ is the only root of the equation $\left(t - t^2/16^2\right)^2 - 1 + \frac{t}{1 + t^2/16^2} = 0$ inside $(0,1)$, whose numerical value equals approximately $0.6194$. Coming back to \eqref{eqn:estimateonE1}, we finally conclude that
    \[ \mu(E_{1}) \leq \pi \left( \frac{ 1+\delta_{0}^2 \frac{C}{c_{0}^3}}{a_{0}^2} \right) \left( \left( \frac{T}{t} \right)^{\frac{1}{\alpha+2}}-1 \right) ,\]
    where $c_{0} \geq \tilde{c}_{0}$, $\delta < c_{0}/16$, $t>t_{0} = c_{0} \frac{\alpha+1}{\pi}$ and $C$ is a positive, computable constant that does not depend on any parameters.
\end{proof}

Having shown \eqref{eqn:BergmanSuperLevelSetsEstimate}, the next step is to find an upper bound for
\[ \int_0^{s^*} \left(v^*(s) - u^*(s)\right) \diff s \]
in terms of the deficit $\delta(f;\Omega,\alpha)$. Here, $s^*$ denotes the minimal solution of the equation $v^*(s) =u^*(s).$

\begin{numlemma}\label{lemma:BergmanLowerBoundDifference}
    For every $f \in \Berg$, there holds
    \begin{equation}\label{eqn:BergmanLowerBoundDifference}
        C \frac{\alpha+1}{\alpha+2}(1-a_0^2) \leq \int_0^{s^*} \left(v^*(s) - u^*(s)\right) \diff s
    \end{equation}
\end{numlemma}
\begin{proof}
    We first let $t^* = v^*(s^*)$. We claim then that $t^* \le \frac{\alpha + 1}{\pi} c, $ for $c> 0$ an absolute constant. Indeed, if $t^* \le \frac{\alpha+1}{\pi} c_0$, where $c_0$ is given by Lemma \ref{lemma:BergmanSuperLevelSetsEstimate}, we are done. Otherwise, using \eqref{eqn:BergmanSuperLevelSetsEstimate} we have that 
    \begin{align}\label{eqn:level-sets-almost-bound}
    \pi \left( \frac{1+ \delta_0^2 \frac{C}{c_0^3}}{a_0^2}\right) \left( \left(\frac{T}{t^*}\right)^{\frac{1}{\alpha+2}} - 1 \right) \ge \rho(t^*) = s^* =  \pi \left( \left( \frac{\frac{\alpha+1}{\pi}}{t^*} \right)^{\frac{1}{\alpha+2}} -1 \right).
    \end{align}
    Denoting $\kappa(a_0) \coloneqq \frac{1+ \delta_0^2 \frac{C}{c_0^3}}{a_0^2}$, we have from \eqref{eqn:level-sets-almost-bound} that 
    \[
    \frac{\frac{\alpha+1}{\pi}}{t^*} \ge \left( \frac{1 - \frac{1}{\kappa(a_0)}}{a_0^{\frac{2}{\alpha+2}} - \frac{1}{\kappa(a_0)}} \right)^{\alpha+2}.
    \]
    We now claim that, whenever $a_0$ is sufficiently close to $1$ and $\alpha > -1$, then 
    \begin{equation}\label{eqn:upper-bound-uniform}
    \left( \frac{a_0^{\frac{2}{\alpha+2}} - \frac{1}{\kappa(a_0)}}{1 - \frac{1}{\kappa(a_0)}} \right)^{\alpha+2} < 1 - \epsilon_0,
        \end{equation}
    for some $\epsilon_0 > 0$. Indeed, a direct computation shows that the left-hand side of \eqref{eqn:upper-bound-uniform} defines an \emph{increasing} function of $a_0 \in (0,1)$, and its limit as $a_0 \to 1$ is simply
    \[
    \left( 1 - \frac{1}{\left(1 + \frac{C}{c_0^3}\right) (\alpha+2)}\right)^{\alpha+2} \le e^{-\frac{1}{1 + \frac{C}{c_0^3}}},
    \]
    which proves \eqref{eqn:upper-bound-uniform}, and concludes our claim on boundedness of $t^*$. 
    
    We now move on to the proof of \eqref{eqn:BergmanLowerBoundDifference}. Calling $\omega(t)$ the right-hand side of \eqref{eqn:BergmanSuperLevelSetsEstimate}, and if $\nu$ is the inverse of $v^*$, given by
    \[ \nu(t) = \pi \left( \left( \frac{\frac{\alpha+1}{\pi}}{t} \right)^{\frac{1}{\alpha+2}} -1 \right), \]
    we may bound
    \begin{align*}
        \int_{0}^{s^*} \left(v^{*}(s) -u^{*}(s) \right) \, \mathrm{d}s &\geq \int_{\tau_{1}}^{T} \left( \nu(t)-\omega(t) \right) \, \mathrm{d}t\\
        &\geq \pi \int_{\tau_{1}}^{T} \left[ \left( \left( \frac{\frac{\alpha+1}{\pi}}{t} \right)^{\frac{1}{\alpha+2}} -1 \right) -\frac{\left( 1+C \frac{\delta_{0}^2}{c_{0}^3} \right)}{a_{0}^2} \left( \left( \frac{T}{t} \right)^{\frac{1}{\alpha+2}}-1 \right) \right] \, \mathrm{d}t\\
        &\geq \pi \frac{\alpha+1}{\pi} \int_{r_{1}}^{a_{0}^2} \left[ r^{-\frac{1}{\alpha+2}}-1- \frac{\left( 1+C \frac{\delta_{0}^2}{c_{0}^3} \right)}{a_{0}^2} \left( \left( \frac{a_{0}^2}{r} \right)^{\frac{1}{\alpha+2}}-1 \right) \right] \, \mathrm{d}r \\ 
        &\geq \frac{\alpha+1}{\alpha+2} \int_{r_{1}}^{a_{0}^2} \sum_{n \geq 1} \frac{1}{n!} \frac{\prod_{j=0}^{n-1} (j\alpha+2j+1)}{(\alpha+2)^{(n-1)}} \left[ (1-r)^n -\left( 1+ \frac{C}{c_{0}^3} \delta_{0}^2\right) \frac{1}{a_{0}^2} \left( \frac{a_{0}^2-r}{a_{0}^2} \right)^n \right]  \, \mathrm{d}r .
    \end{align*}
    We changed variables to set $t=\frac{\alpha+1}{\pi}r$, $\tau_{1} = \frac{\alpha+1}{\pi} r_{1}$ and $T=a_{0}^2 \frac{\alpha+1}{\pi}$. We also used the fact that that $1 > r_{1}  \geq c_{0} > \tilde{c}_{0}$ holds for $\tau_1 > t$ that we chose.
    
    We now need to find a lower bound for the terms inside the integral in terms of $\delta_{0}^2$. Labelling
    \[ f_{n}(r) = (1-r)^n -\left( 1+ \frac{C}{c_{0}^3} \delta_{0}^2\right) \frac{1}{a_{0}^2} \left( \frac{a_{0}^2-r}{a_{0}^2} \right)^n ,\]
    notice that this this function is negative for $r<r_{n}$,
    \[ r_{n}= 1-\frac{\delta_{0}^2 \sqrt[n]{ \frac{1}{a_{0}^2} \left( 1+ \frac{C}{c_{0}^3} \delta_{0}^2 \right) }}{\sqrt[n]{ \frac{1}{a_{0}^2} \left( 1+ \frac{C}{c_{0}^3} \delta_{0}^2 \right) } -a_{0}^2} ,\]
    and positive between $r_{n}$ and $a_{0}^2$. The crucial observation is that the $r_{n}$ form a \emph{decreasing} sequence, meaning that $f_{n}(r)$ becomes positive before $f_{n-1}(r)$. In order to see this, note that $r_n$ is decreasing if, for instance, the sequence 
    \[
    \beta_n(y) = \frac{y^{1/n}}{y^{1/n} - a_0^2} = 1 + \frac{a_0^2}{y^{1/n} - a_0^2} 
    \]
    is increasing in $n$, for each $y > 1$. But this is easy to see from the fact that $n \mapsto y^{1/n}$ is decreasing.
    
    These considerations imply that we can bound all the $f_{n}(r) \geq 0$ between $r_{1}$ and $a_{0}^2$, i.e.
    \begin{align*}
        \int_{0}^{s^*}\left( v^{*}(s) -u^{*}(s) \right) \, \mathrm{d}s &\geq \int_{\tau_{1}}^{T} \left( \nu(t)-\omega(t) \right) \, \mathrm{d}t\\
        &\geq \frac{\alpha+1}{\alpha+2} \int_{r_{1}}^{a_{0}^2} f_{1}(r) \, \mathrm{d}r = \frac{\alpha+1}{\alpha+2} \int_{a_{0}^2 \frac{C/c_{0}^3}{1+C/c_{0}^3}}^{a_{0}^2} \frac{\delta_{0}^2}{a_{0}^2} \left[ r-(a_{0}^2-r) \frac{C}{c_{0}^3} \right]  \, \mathrm{d}r \\
        &\geq \frac{\alpha+1}{\alpha+2} \delta_{0}^2 a_{0}^2 \int_{\frac{C/c_{0}^3}{1+C/c_{0}^3}}^{1} \left( \left( 1+\frac{C}{c_{0}^3} \right) u - \frac{C}{c_{0}^3} \right)\, \mathrm{d}u \geq \frac{1}{2} \cdot \frac{\alpha+1}{\alpha+2} \cdot \delta_{0}^2 a_{0}^2 \cdot \frac{1}{1+C/c_{0}^3} .
    \end{align*}
    Since $a_{0}$ can be bounded from below in terms of $c_{0}$ (given our assumption that $\delta_{0}<c_{0}/16$) and $c_{0}$ can be bounded from below by $\tilde{c}_{0}$, we conclude that the right hand side behaves like
    \[ K \frac{\alpha+1}{\alpha+2} \delta_{0}^2 ,\]
    where $K$ is an absolute, computable constant.
\end{proof}

At this point we are finally ready to show the stability result for the function in terms of the closeness of $a_0$ to $1$. Coupled with Lemma~\ref{lemma:RKHSQuantitative} below, we will conclude the proof of \eqref{eqn:BergmanStabilityFunction}.

\begin{numlemma}
    There is $M_{\alpha}(s)$ such that for every $f \in \Berg$ and every $\Omega \subset \disk$ with $s = \mu(\Om)$,
    \begin{equation}\label{eqn:BergmanFunctionStabilityBeforeRKHS}
        1-a_0^2 \leq M_{\alpha}(s) \delta(f;\Om,\alpha) .
    \end{equation}
\end{numlemma}
\begin{proof}
    We will denote $\mu(\Omega) = s_0$ and split the proof in three steps. In the first one, we use Lemma~\ref{lemma:BergmanLowerBoundDifference} to treat the case $s_0 > s^*$. In the second step, we treat the remaining cases, while in the last step, we give the function $M_{\alpha}(s)$ explicitly.

    \vspace{2mm}
    
    \noindent\underline{\textsc{Step I}.} To begin with, we let $\dso = \delta(f; \inset{u > u^*(s_0)}, \alpha) \leq \delta(f; \Om, \alpha)$, and we define
    \begin{equation}\label{eqn:InProofFEDefnEps}
        \eps \coloneqq \dso \int_0^{s_0} v^*(s) \diff s = \int_{0}^{s_0} \left( v^*(s) - u^*(s) \right) \diff s ,
    \end{equation}
    and the ratio
    \begin{equation}
        r(s) \coloneqq \frac{u^*(s)}{v^*(s)},
    \end{equation}
    which is increasing on $[0,\infty)$ - this is an equivalent formulation of the main result from \cite{RamosTilli}. 

    \vspace{2mm}

    \noindent\textit{\underline{Case $s_0 > s^*$}}. Notice that $r(s^*) = 1$, and from \eqref{eqn:InProofFEDefnEps},
    \[ \eps = \int_{s_0}^{\infty} u^*(s) \of{1- \inv{r(s)}} \diff s \geq \of{1- \inv{r(s_0)}} \int_{s_0}^{\infty} u^*(s) \diff s .\]
    For the same reason,
    \[ \int_{s^*}^{s_0} \left( u^*(s) - v^*(s) \right) \diff s = \int_{s^*}^{s_0} u^*(s) \of{1-\inv{r(s)}} \diff s \leq \of{1-\inv{r(s_0)}} \int_{s^*}^{s_0} u^*(s) \diff s .\]
    Combining these two estimates yields
    \[ \int_{s^*}^{s_0} \left( u^*(s) - v^*(s) \right) \diff s \leq \eps \frac{\int_{s^*}^{s_0} u^*(s) \diff s}{\int_{s_0}^{\infty} u^*(s) \diff s} ,\]
    and recalling \eqref{eqn:InProofFEDefnEps},
    \begin{align*}
        \int_0^{s^*} \left( v^*(s)-u^*(s) \right) \diff s &\leq \eps + \eps \frac{\int_{s^*}^{s_0} u^*(s) \diff s}{\int_{s_0}^{\infty} u^*(s) \diff s} = \eps \frac{\int_{s^*}^{\infty} u^*(s) \diff s}{\int_{s_0}^{\infty} u^*(s) \diff s} \\
        &\leq \frac{\eps}{\int_{s_0}^{\infty} v^*(s) \diff s} 
        = \frac{\int_{0}^{s_{0}} v^*(s)  \diff s }{\left( 1+\frac{s_{0}}{\pi} \right)^{-(\alpha+1)}}\delta_{s_{0}} = \left[ \left( 1+\frac{s_{0}}{\pi} \right)^{\alpha+1} -1 \right] \delta_{s_{0}} .
    \end{align*}
    Invoking Lemma~\ref{lemma:BergmanLowerBoundDifference}, we conclude that
    \begin{equation}\label{eqn:InProofFSEstimate1}
        1-a_0^2 \leq C_1 \frac{\alpha+2}{\alpha+1} \left[ \left( 1+\frac{s_{0}}{\pi} \right)^{\alpha+1} -1 \right] \delta(f;\Om,\alpha) , \quad s_0 > s^* .
    \end{equation}

    \vspace{2mm}
    
    \noindent\textit{\underline{Case $s_0 < s^*$}}. Again from \eqref{eqn:InProofFEDefnEps}, we find
    \[ \eps = \int_{0}^{s_0} v^*(s) (1-r(s)) \diff s \geq (1-r(s_0)) \int_0^{s_0} v^*(s) \diff s . \]
    Arguing similarly,
    \[ \int_{s_0}^{s^*} \left( v^*(s) - u^*(s) \right) \diff s = \int_{s_0}^{s^*} v^*(s) (1-r(s)) \diff s \leq (1-r(s_0)) \int_{s_0}^{s^*} v^*(s) \diff s ,\]
    and combining the two,
    \[ \int_{s_0}^{s^*} \left( v^*(s) - u^*(s) \right)  \diff s \leq \eps \frac{\int_{s_0}^{s^*} v^*(s) \diff s}{\int_{0}^{s_0} v^*(s) \diff s} . \]
    This, in turn, yields
    \begin{equation*}
        \int_0^{s^*} \left( v^*(s) - u^*(s) \right) \diff s \leq \eps + \eps \frac{\int_{s_0}^{s^*} v^*(s) \diff s}{\int_{0}^{s_0} v^*(s) \diff s} = \eps \frac{\int_{0}^{s^*} v^*(s) \diff s}{\int_{0}^{s_0} v^*(s) \diff s} = \left[ 1- \left(1+\frac{s^*}{\pi}\right)^{-(\alpha+1)} \right] \dso .
    \end{equation*}
    Using Lemma~\ref{lemma:BergmanLowerBoundDifference} again,
    \begin{equation}\label{eqn:InProofFSEstimate2}
        1-a_0^2 \leq C_2 \frac{\alpha+2}{\alpha+1} \left[ 1- \left(1+\frac{s^*}{\pi}\right)^{-(\alpha+1)} \right] \delta(f;\Om,\alpha) , \quad s_0 < s^* .
    \end{equation}

    \vspace{2mm}
    
    \noindent\underline{\textsc{Step II}.} Although the previous step already shows the stability estimate for $\alpha$ away from $-1$, the \emph{key problem} we face now is that \eqref{eqn:InProofFSEstimate2} is not sharp as $\alpha \to -1$. Indeed, we would need to make sure that $s^*$ behaves like a constant that is independent of $\alpha$. We may achieve a lower bound on that value coming from the upper bound on $t^*$ in the proof of Lemma~\ref{lemma:BergmanSuperLevelSetsEstimate}, but we lack an upper bound. To circumvent this issue, we exploit Lemma~\ref{lemma:BergmanSuperLevelSetsEstimate} again in order to compare $v^*$ and $u^*$ when $s_0 < s^*$.
    
    Effectively: we may rephrase \eqref{eqn:BergmanSuperLevelSetsEstimate} as
    \begin{equation*}
        u^*(s) \leq \frac{\alpha+1}{\pi}a_0^2 \left( 1+ \left( \frac{a_0^2}{1+K_0 \delta_0^2} \right) \frac{s}{\pi} \right)^{-(\alpha+2)} = w(s),
    \end{equation*}
    which holds for all $s$ such that $w(s)>t_0$, and we define $\Tilde{s}$ as the unique solution to $w(\Tilde{s})=v^*(\Tilde{s})$, which reads
    \begin{equation}\label{eqn:def-tilde-s}
    \Tilde{s} = \pi (1+K_0 (1-a_0^2)) \frac{1-(a_0^2)^{-\frac{1}{\alpha+2}}}{(a_0^2)^{\frac{\alpha+1}{\alpha+2}} - 1-K_0 (1-a_0^2)} ,
    \end{equation}
    and is furthermore bounded from above by some constant $\tilde{s} \leq C$. To check this, notice first that the maximum of the right-hand side of \eqref{eqn:def-tilde-s} is achieved at $\alpha=-1$, which follows by simply noting that the derivative of the function $t \mapsto \frac{1-t^{-\frac{1}{\alpha+2}}}{t^{\frac{\alpha+1}{\alpha+2}} - 1 - K_0(1-t)}$ is given by 
    \[
    \frac{(1 + K_0) (1 - t) t^{-\frac{1}{2 + \alpha}}
    \log(t)}{(2 + \alpha)^2 \left(-1 + K_0 (-1 + t) + t^{\frac{1 + \alpha}{2 + \alpha}}\right)^2},
    \]
    which is clearly non-positive whenever $t<1$. The expression at $\alpha = -1$ evaluates explicitly to 
    \[
    \pi (1 + K_0(1-a_0^2)) \frac{1}{K_0 a_0^2},
    \]
    and hence since $a_0^2$ is supposed to be bounded away from zero as in Lemma~\ref{lemma:BergmanSuperLevelSetsEstimate}, we conclude that $\tilde{s} \leq C$, for some universal constant $C$, which is independent of $\alpha$, as desired. We now distinguish the cases $s^* \leq \tilde{s}$ and $s^* \geq \tilde{s}$.

    \vspace{2mm}

    \noindent\textit{\underline{Case 1: $s^* \leq \tilde{s}$}}. In this situation, we have the desired (uniform on $\alpha$) upper bound on $s^* \leq \tilde{s} \leq C$ that makes \eqref{eqn:InProofFSEstimate2} sharp as $\alpha \to -1$, and we conclude, since in that case we have 
\begin{equation}\label{eqn:first-bound-s*-uniform}
C_2 \frac{\alpha+2}{\alpha+1} \left[ 1- \left(1+\frac{s^*}{\pi}\right)^{-(\alpha+1)} \right] \le C'_2,
\end{equation} 
for some $C'_2$ absolute constant. 
    \vspace{2mm}

    \noindent\textit{\underline{Case 2: $s^* \geq \tilde{s}$}}. We need to distinguish cases again in terms of $s_0$. We first claim that
    \begin{equation*}
        y(s) = v^*(s)-w(s) = \frac{\alpha+1}{\pi} \left[ \left(1+\frac{s}{\pi}\right)^{-(\alpha+2)} - a_0^2 \left(1+\left(\frac{a_0^2}{1+K\delta_0^2}\right)\frac{s}{\pi}\right)^{-(\alpha+2)} \right] .
    \end{equation*}
    is globally bounded from below by
    \[ \Tilde{y}(s)=y(0)+y'(0)s = \frac{\alpha+1}{\pi} \left[ (1-a_0^2) + \frac{\alpha+2}{\pi} \left( \frac{a_0^4}{1+K_0 \delta_0^2} -1 \right)s \right] .\] 
    The basic idea in order to prove this is to use \emph{concavity properties} of $y$ to our advantage. As one can see from a direct inspection, $y$ is concave in a neighborhood of the origin, so we only need to check from the point where it stops being concave onward; the idea is that $y$ tends to $0$ as $s \to \infty$ and that, by the time it stops being concave, $\tilde{y}$ is already too negative. In effective terms, we want to show that 
    \begin{equation}\label{eqn:y-bound} 
    y(x) \geq y(0)+y'(0)x \Leftrightarrow \frac{\int_0^x y'(t) \diff t}{x} \geq y'(0) ,
    \end{equation} 
    so it suffices to show that $y'$ attains its global minimum at $s=0$. In order to prove this, we first note that the second derivative of $y$ is simply 
    \[
    y''(s) = \frac{(3+\alpha)(2+\alpha)(1+\alpha)}{\pi^3}\left( \left(1 + \frac{s}{\pi}\right)^{-4 -\alpha} - 
   \frac{a_0^6}{(1+K\delta_0^2)^2} \left( 1 + \left( \frac{a_0^2}{1+K \delta_0^2}\right) \frac{s}{\pi}\right)^{-4 - \alpha}\right), 
    \]
    which can vanish at most at one positive value, which we denote by $s_m.$ From that point on $y'$ is non-increasing. Let $\tilde{s}_m$ denote further the point where $x \mapsto \frac{1}{x} \int_0^x y'(t) \, \diff t$ attains its maximum. Clearly, from the definitions we have that $\tilde{s}_m$ is unique (and hence well-defined), and we have $\tilde{s}_m \ge s_m.$ We further have $\lim_{s \to \infty} y'(s) = 0$, yielding, for $x > \tilde{s}_m$,  
    \begin{align*} 
    \frac{1}{x} \int_0^x y'(t) \, \diff t &\ge \frac{1}{x} \left( \int_0^{\tilde{s}_m} y'(t) \, \diff t + \int_{\tilde{s}_m}^x y'(t) \, \diff t \right) \cr 
        & \ge \frac{\tilde{s}_m}{x} \cdot y'(0) \ge y'(0),
    \end{align*} 
    since $y'(0) < 0$. For $x < \tilde{s}_m$, we claim that $x \mapsto \frac{1}{x} \int_0^x y'(t) \, \diff t$ is \emph{increasing} on the interval $(0,\tilde{s}_m)$. By differentiating that function, this amounts to showing that $y'(x) > \frac{1}{x} \int_0^x y'(t) \, \diff t$ on that interval, which can be checked with a case analysis: if $x < s_m$, this follows from the fact that $y'$ is itself increasing on $(0,s_m)$. For $s_m \le x < \tilde{s}_m$, suppose this does not happen, and hence we may find $x_0 \in [s_m,\tilde{s}_m)$ such that $y'(x_0) = \frac{1}{x_0} \int_0^{x_0} y'(t) \, \diff t$. But the monotonocity of $y'$ in $(s_m,+\infty)$ implies
    \[
    \int_0^x y'(t) \, \diff t = \int_0^{x_0} y'(t) \, \diff t + \int_{x_0}^x y'(t) \, \diff t < x_0 \cdot y'(x_0) + (x-x_0) y'(x_0) = x \cdot y'(x_0),
    \]
    and hence $x_0$ is a global maximum point for $\frac{1}{x} \int_0^x y'(t) \, \diff t$, a contradiction to the uniqueness of $\tilde{s}_m$. We hence conclude the proof of \eqref{eqn:y-bound}.
    
    The solution $s'$ to $\Tilde{y}(s') = 0$ is then
    \[ s' = \frac{\pi}{\alpha+2} \cdot \frac{(1+K_0 \delta_0^2)(1-a_0^2)}{1+K_0 \delta_0^2 - a_0^4} \geq \frac{\pi}{\alpha+2} \cdot \frac{1}{K_0+2} ,\]
    and we distinguish two further subcases.

    \vspace{2mm}

    \noindent\textit{\underline{Case 2.1: $s_0 > s'$}}. We can estimate
    \[ C\frac{\alpha+1}{\alpha+2} (1-a_0^2) \leq \int_0^{s'} \tilde{y}(s) \diff s \leq \int_0^{s'} \left( v^*(s)-w(s) \right) \diff s \leq \int_0^{s_0} \left( v^*(s)-u^*(s) \right) \diff s ,\]
    reaching
    \begin{equation}\label{eqn:InProofFSEstimate3}
        1-a_0^2 \leq C_3 \frac{\alpha+2}{\alpha+1} \left[ 1- \of{1+\frac{s_0}{\pi}}^{-(\alpha+1)} \right] \delta(f;\Om, \alpha) , \quad s' < s_0 < s^* .
    \end{equation}

    \vspace{2mm}

    \noindent\textit{\underline{Case 2.2: $s_0 < s'$}}. Finally, in this situation we can easily estimate
    \[ C \frac{\alpha+1}{\pi} (1-a_0^2) s_0 \leq \int_0^{s_0} \tilde{y}(s) \diff s \int_0^{s_0} \left( v^*(s)-w(s) \right) \diff s \leq \int_0^{s_0} \left( v^*(s)-u^*(s) \right) \diff s = \theta_{\alpha}(s_0) \dso , \]
    which yields
    \begin{equation}\label{eqn:InProofFSEstimate4}
        1-a_0^2 \leq \frac{C_4}{\alpha+1} \left[ \frac{ 1- \of{1+\frac{s_0}{\pi}}^{-(\alpha+1)}}{s_0} \right] \delta(f;\Om, \alpha) , \quad s_0 < s' < s^* .
    \end{equation}

    \vspace{2mm}
    
    \noindent\underline{\textsc{Step III}.} Estimates \eqref{eqn:InProofFSEstimate1}, \eqref{eqn:InProofFSEstimate2}, \eqref{eqn:InProofFSEstimate3} and  \eqref{eqn:InProofFSEstimate4} are sharp as $\alpha \to -1$, but we wish to combine them into a more compact expression. To this end, notice that the right-hand side of \eqref{eqn:InProofFSEstimate3} can be bounded by that of \eqref{eqn:InProofFSEstimate1} up to choosing the constant $C_{13} = \max \inset{C_1,C_3}$. In other words, by the mean value theorem, we have that there is an absolute constant $C'_{4}$ such that 
    \[ \frac{C_4}{\alpha+1} \left[ \frac{ 1- \of{1+\frac{s_0}{\pi}}^{-(\alpha+1)}}{s_0} \right]  \leq C'_{4} .\]
     Therefore, we may choose $C \geq \max \inset{C_{13},C'_{2},C'_4}$ and define
    \begin{equation}\label{eqn:defnOfM} 
        M_{\alpha}(s) \coloneqq C \of{1 + \frac{\alpha+2}{\alpha+1} \left[ \of{1+\frac{s}{\pi}}^{\alpha+1} -1\right]} \geq \max\inset{C'_2,C'_4 , C_{13} \frac{\alpha+2}{\alpha+1} \left[ \of{1+\frac{s}{\pi}}^{\alpha+1} -1\right] }  ,
    \end{equation}
    concluding the proof of \eqref{eqn:BergmanFunctionStabilityBeforeRKHS}.
\end{proof}

To relate the left-hand sides of \eqref{eqn:BergmanFunctionStabilityBeforeRKHS} and \eqref{eqn:BergmanStabilityFunction}, we include the following well-known result from the theory of reproducing kernel Hilbert spaces -- see \cite{KNOT} for instance. For $f \in \mcalH$, a reproducing kernel Hilbert space of complex-valued functions with kernel $K$ and norm $\norm{\cdot}$, define
\[ \Delta_f(x) = \frac{\norm{f}^2 - \abs{f(x)}^2 K(x,x)^{-1}}{\norm{f}^2} ,\]
and recall that
\begin{equation}\label{eqn:RKHSBound}
    \abs{f(x)}^2 K(x,x)^{-1} \leq \norm{f}^2 .
\end{equation}

\begin{numlemma}\label{lemma:RKHSQuantitative}
    Let $(\mcalH,\norm{\cdot})$ be a reproducing kernel Hilbert space of complex-valued functions on $X$ with kernel $K$, and let $f \in \mcalH$. Then for every $x \in X$ there exists a constant $c \in \C$ with $\abs{c}=\norm{f}$ such that
    \begin{equation}\label{eqn:QuantitativeRKHS}
        \frac{\norm{f-cf_x}}{\norm{f}} \leq \sqrt{2} \Delta_f(x)^{1/2} ,
    \end{equation}
    where $f_x = K(x,x)^{-1/2} K_x$ is a multiple of the reproducing kernel.
\end{numlemma}
\begin{proof}
    Fix $x \in X$ and consider the following normalization of $K_x$,
    \begin{equation}\label{eqn:RKHSNormalizedKernel}
        f_x (y) \coloneqq  K(x,x)^{-1/2}K_x(y) .
    \end{equation}
    Given $f \in \mcalH$, assume that $f(x) \neq 0$ and choose $c = \norm{f} \frac{f(x)}{\abs{f(x)}}$. With this choice, and thanks to \eqref{eqn:RKHSBound},
    \begin{align*}
        \R \ni \inner{f,cf_{x}} &= \conj{c} K(x,x)^{-1/2} f(x) = \norm{f}K(x,x)^{-1/2}\abs{f(x)} \\
        &= \frac{\abs{f(x)}^2 K(x,x)^{-1}}{\abs{f(x)}K(x,x)^{-1/2}} \norm{f} \geq \norm{f}^2 (1-\Delta_f (x) ) .
    \end{align*}
    When we now look at the distance in the Hilbert space norm between $f$ and our candidate, $c f_x$, we find that
    \begin{align*}
        \norm{f-cf_{x}}^2 &= \norm{f}^2 + \abs{c}^2 \norm{f_{x}}^2 - 2 \tRe \inner{f, c f_x}\\
        &\leq 2\norm{f}^2 - 2 \norm{f}^2(1-\Delta_f(x)) = 2\norm{f}^2 \Delta_f(x) .
    \end{align*}
    Assume now that $f(x)=0$. Then $\Delta_f(x)=1$ and if $c=\norm{f}$, the computation above shows that $\norm{f-cf_x}^2 =2 \norm{f}^2$. One obtains \eqref{eqn:QuantitativeRKHS} by rearranging the terms.
\end{proof}

We are finally ready to prove \eqref{eqn:BergmanStabilityFunction}.

\begin{proof}[Proof of \eqref{eqn:BergmanStabilityFunction}]
    Let $f \in \Berg$ be such that $\norm{f}_{\Berg}=1$. Lemma~\ref{lemma:RKHSQuantitative} reads
    \begin{equation*}
        \inf_{\substack{\abs{c} = \norm{f}_{\Berg} ,\\ \om \in \disk}} \norm{f-c\ff_{\om}}_{\mathbf{B}_{\alpha}^2}^2 \leq 1-a_0^2 .
    \end{equation*}
    Putting this together with \eqref{eqn:BergmanFunctionStabilityBeforeRKHS} yields the desired \eqref{eqn:BergmanStabilityFunction}.
\end{proof}

\subsection{Set Stability}\label{subsec:SetStability}

We will now address the proof of \eqref{eqn:BergmanStabilitySet}. To this end, we begin by introducing a few normalizations, different from the ones used in the proof of \eqref{eqn:BergmanStabilityFunction}.

First, for $f \in \Berg$ such that $\norm{f}_{\Berg} = 1$, we explicitly compute:
\begin{align*}
\lVert f-c\ff_{z_{0}} \rVert_{\mathbf{B}_{\alpha}^2}^2 &= 2  \left( 1-\mathrm{Re}\left( \overline{c} f(z_{0})\sqrt{ \frac{\pi}{\alpha+1} } (1-\lvert z \rvert^2)^{\frac{\alpha+2}{2}} \right) \right), 
\end{align*}
and minimizing the above over ${c}$ such that $|c| = \|f\|_{\Berg}$ and $z_{0} \in \disk$ we find that
\begin{equation}\label{eqn:InSetProofM}
    \min_{z_{0} \in \mathbb{D}, |c|=\lVert f \rVert_{\mathbf{B}_{\alpha}^2}} \frac{\lVert f-c\ff_{z_{0}} \rVert_{\mathbf{B}_{\alpha}^2}^2}{\lVert f \rVert_{\mathbf{B}_{\alpha}^2}^2} = 2  \left( 1- \sqrt{ \frac{\pi}{\alpha+1} \max_{z \in \mathbb{D}} u(z) } \right) .
\end{equation}
We now further assume, without loss of generality, that
$$\min_{z_0 \in \disk, c \in \C} \frac{ \|f - c \cdot \ff_{z_0}\|_{\Berg}}{\|f\|_{\Berg}} =  \min_{c\in \C} \frac{\| f - c \|_{\Berg}}{\|f\|_{\Berg}}.
$$
Observe that the minimum from before is equivalent to the distance between $f$ and the set of extremals: indeed, we have
\begin{align}
    \label{eq:equivalentqtties}
   \min_{z_0 \in \disk, c \in \C} \frac{ \|f - c \cdot \ff_{z_0}\|_{\Berg}^2}{\|f\|_{\Berg}^2} & \leq \min_{z_0 \in \disk, |c|=\|f\|_{\Berg}} \frac{ \|f - c \cdot \ff_{z_0}\|_{\Berg}}{\|f\|_{\Berg}} \cr 
   & = \min_{|c|=\|f\|_{\Berg}} \frac{ \|f - c \|_{\Berg}}{\|f\|_{\Berg}} \leq \sqrt{2} \min_{z_0 \in \disk, c \in \C} \frac{ \|f - c \cdot \ff_{z_0}\|_{\Berg}^2}{\|f\|_{\Berg}^2},
\end{align}
by Lemma \ref{lemma:RKHSQuantitative}. We will now assume also that $u$ achieves its maximum at $z=0$, and we will consider it as given by a perturbation
\[ f(z) = \sqrt{ \frac{\alpha+1}{\pi} } + \varepsilon g(z) , \quad f(0) = \sqrt{ \frac{\alpha+1}{\pi} } ,\]
satisfying $\lVert g \rVert_{\mathbf{B}_{\alpha}^2} = 1$. We will sometimes denote such $\e = \e(f).$ Note that, according to this normalization and the fact that $u$ achieves its maximum at $0$, we have $f'(0) = 0$, which implies
\begin{equation}
    \label{eq:normG}
    \quad \langle g, 1\rangle_{\Berg} = \langle g, z\rangle_{\Berg} = 0.
\end{equation}
Additionally, \eqref{eqn:InSetProofM} now shows that
\[ \min_{\lvert c \rvert=\lVert f \rVert_{\mathbf{B}_{\alpha}^2}} \frac{\left\lVert  f-c\sqrt{ \frac{\alpha+1}{\pi} }  \right\rVert_{\mathbf{B}_{\alpha}^2}^2}{\lVert f \rVert_{\mathbf{B}_{\alpha}^2}^2} = 2\left( 1-\sqrt{ \frac{\pi}{\alpha+1}} \frac{\lvert f(0) \rvert}{\lVert f \rVert_{\mathbf{B}_{\alpha}^2}}  \right) = 2 \frac{\lVert f \rVert_{\mathbf{B}_{\alpha}^2} - 1}{\lVert f \rVert_{\mathbf{B}_{\alpha}^2}} .\]
Recalling \eqref{eqn:defnOfM}, the upper bound in \eqref{eqn:BergmanStabilityFunction} can now be used to reach
\[ 1 \leq \lVert f \rVert_{\mathbf{B}_{\alpha}^2} \leq \frac{2}{\left(2-M_{\alpha}(s) \delta(f,\alpha,\Omega) \right)^{1/2}} \leq 2 ,\]
whenever the deficit is small enough in terms of $s = \mu(\Om)$. Furthermore, recalling that
\[\varepsilon = \left\lVert  f-\sqrt{ \frac{\alpha+1}{\pi} }  \right\rVert_{\mathbf{B}_{\alpha}^2} ,\]
notice that
\[ \frac{\eps}{2} \leq \frac{\eps}{\norm{f}_{\Berg}} = \frac{\norm{f-\sqrt{\frac{\alpha+1}{\pi}}}_{\Berg}}{\norm{f}_{\Berg}} \]
and thanks to the proof of \eqref{eqn:BergmanStabilityFunction}, we may assume that this is as small as necessary, depending on $s = \mu(\Om)$ and $\alpha >-1$.

With the new normalizations at hand, we move onto discussing a few properties of the super-level sets of $u$, which will be featured throughout the proof of \eqref{eqn:BergmanStabilitySet}.
Given that the super-level sets of $u$ are unique, we can write
\[ A_{\Omega}= A_{u^*(\Omega)}.\]
We may estimate
\begin{align*}
    \abs{u-v} &= \abs{\left(\abs{f}^2-\frac{\alpha+1}{\pi}\right)} (1-\abs{z}^2)^{\alpha+2} \\
    &\leq \left(\abs{f-\sqrt{\frac{\alpha+1}{\pi}}}\right) \left( \abs{f}+\sqrt{\frac{\alpha+1}{\pi}}\right) (1-\abs{z}^2)^{\alpha+2} \leq \eps (\norm{f} + 1) \frac{\alpha+1}{\pi} \leq 3 \eps \frac{\alpha+1}{\pi} ,
\end{align*}
which we can rephrase as
\begin{equation}\label{eqn:Inclusions}
    L_{\Om} = \inset{v > u^*(s) + 3 \eps \frac{\alpha+1}{\pi}} \subset A_{\Om} = \inset{u > u^*(s)} \subset E_{\Om} = \inset{v > u^*(s)-3 \eps \frac{\alpha+1}{\pi}}.
\end{equation}
The sets $L_{\Om}$ and $E_{\Om}$ are balls of radii $r_-$ and $r_+$ respectively, where
\begin{equation}\label{eqn:InSetProofRadii}
    r_{\pm}^2 = 1- \left[ \frac{\pi}{\alpha+1} \left(u^*(s)\mp 3\eps\frac{\alpha+1}{\pi}\right) \right]^{\inv{\alpha+2}} .
\end{equation}

\begin{proof}[Proof of \eqref{eqn:BergmanStabilitySet}]
    Let $\mcalT$ be the map transporting $\mathbbm{1}_{A_{\Omega} \setminus \Omega} \mathrm{d} \mu$ into $\mathbbm{1}_{\Omega \setminus A_{\Omega}} \mathrm{d} \mu$, whose existence is granted by absolute continuity with respect to the Lebesgue measure \cite[p. 12]{FigalliGlaudo}. Define now
    \[ B = \inset{z \in A_{\Om} \colon \abs{\mcalT(z)}^2 - \abs{z}^2 > K_{\Om} \gamma} , \]
    and observe that
    \begin{equation}\label{eqn:InSetProofD1}
        \int_B u(z)-u(\mcalT(z)) \diff \mu(z) = \int_B u \diff \mu - \int_{\mcalT(B)} u \diff \mu \leq \int_{A_{\Om}} u \diff \mu - \int_{\Om} u \diff \mu = d(\Om) ,
    \end{equation}
    and
    \begin{equation}\label{eqn:InSetProofD2}
        d(\Om) \leq \norm{f}_{\mathbf{B}_{\alpha}^2}^2 \theta_{\alpha}(s) - \int_{\Om} u \diff \mu = \norm{f}_{\mathbf{B}_{\alpha}^2}^2 \theta_{\alpha}(s) \delta(f;\Om,\alpha) \leq 4 \theta_{\alpha}(s) \delta(f;\Om, \alpha) .
    \end{equation}
    We now wish to control the size of $B$ in terms of $d(\Om)$ in order to then exploit \eqref{eqn:Inclusions} when estimating the asymmetry of $A_{\Om}$. To this end, we divide the proof into three steps.

    \vspace{2mm}

    \noindent\underline{\textsc{Step I}. \textit{Control over $B$}.} Notice that thanks to \eqref{eqn:InSetProofD1} and \eqref{eqn:InSetProofD2}, in order to control the measure $\mu(B)$, we only need to find a lower bound for $u(z)-u(\mcalT(z))$,
    \begin{align*}
        u(z)-u(\mcalT(z)) \geq \frac{\alpha+1}{\pi} &(1-\abs{z}^2)^{\alpha+2} \of{1- \of{ \frac{1-\abs{\mcalT(z)}^2}{1-\abs{z}^2} }^{\alpha+2}} - \frac{\alpha+1}{\pi}\eps^2 -4 \frac{\alpha+1}{\pi} \eps ,
    \end{align*}
    and for $z \in B$, we find that $(\abs{\mcalT(z)}^2-|z|^2)/(1-|z|^2)$ is positive, and at most $1$. Using the definition of $B$ and assuming that $K_{\Om} \gamma \leq c_1 / (\alpha+2)$ for some constant $c_1 > 0$, we find
    \[ u(z)- u(\mcalT(z)) \geq \frac{\alpha+1}{\pi} (\alpha+2) (1-r_-^2)^{\alpha+1} K_{\Om}\gamma - \frac{\alpha+1}{\pi} \eps^2 - 4 \frac{\alpha+1}{\pi} \eps .\]
    On the other hand, if $K_{\Om} \gamma > c_1 / (\alpha+2)$, then 
    \begin{align*}
        u(z)- u(\mcalT(z)) &\geq \frac{c_1}{\pi} \cdot \frac{\alpha+1}{\alpha+2} (1-r_-^2)^{\alpha+1} - \frac{\alpha+1}{\pi} \eps^2 - 4 \frac{\alpha+1}{\pi} \eps .
    \end{align*}

    Letting $\gamma = \eps$ and assuming that
    \begin{equation}\label{eqn:SmallnessOfEps}
        \eps < c\left[ \frac{C(\alpha+2)}{1+C(\alpha+2)} \right] \of{1+\frac{s}{\pi}}^{-(\alpha+2)} , \quad K_{\Om} = \frac{6}{\alpha+2} (1-r_+^2)^{-(\alpha+1)} ,
    \end{equation}
    for $c>0$ a small absolute constant to be chosen later, we conclude that
    \[ u(z)-u(\mcalT(z)) \geq \frac{\alpha+1}{\pi} \gamma , \quad z \in B . \]
    We can couple this with \eqref{eqn:InSetProofD1} and \eqref{eqn:InSetProofD2} to estimate
    \begin{equation}\label{eqn:InSetProofBoundB}
        \mu(B) \leq \frac{\pi}{(\alpha+1)(\alpha+2)} \cdot \frac{d(\Om)}{\gamma} \leq \frac{4\pi}{(\alpha+1)(\alpha+2)} \cdot \frac{\theta_{\alpha}(s)\delta(f;\Om, \alpha)}{\gamma} .
    \end{equation}

    \vspace{2mm}
    
    \noindent\underline{\textsc{Step II}. \textit{Showing that $\Omega$ is close to $A_{\Omega}$}.} Note the identities
    \[ \mu(\Om) - \mu(B) = \mu(\Om) - \mu(\mcalT(B)) = \mu(\Om \setminus \mcalT(B)) = \mu(\Om) - \mu(A_{\Om} \setminus \Om) + \mu( (\Om \setminus \mcalT(B)) \setminus A_{\Om} ) ,\]
    from which we deduce that
    \begin{equation}\label{eqn:InSetProof3}
        \inv{2} \mu(\Om \Delta A_{\Om}) = \mu(A_{\Om} \setminus \Om) = \mu(B) + \mu( (\Om \setminus \mcalT(\Om)) \setminus A_{\Om} ).
    \end{equation}
    For an estimate in the right-hand side it suffices to bound the second term. We note that $\Om \setminus \mcalT(B)$ is contained in a neighborhood of $A_{\Om}$ that depends on $K_{\Om}\gamma$, and which is in turn nested between $L_{\Om}$ and $E_{\Om}$ of radii $r_{\pm}$ \eqref{eqn:InSetProofRadii}.
    
    We may therefore estimate
    \begin{align*}
        \mu((\Om \setminus \mcalT(B))\setminus A_{\Om}) &\leq \mu \left(B_{\sqrt{K_{\Om} \gamma + r_+^2}} \setminus B_{r_-}\right) = \pi \of{ \frac{r_+^2 + K_{\Om} \gamma - r_-^2}{(1-r_+^2 - K_{\Om}\gamma)(1-r_-^2)}} \\
        &\leq 4 \pi \of{1+\frac{s}{\pi}}^{2(\alpha+2)} \of{r_+^2-r_-^2 + K_{\Om} \gamma} \\
        &\leq 4 \pi \of{1+\frac{s}{\pi}}^{2(\alpha+2)} \left[ \frac{24}{\alpha+2} \of{1+\frac{s}{\pi}}^{\alpha+1} \eps + \frac{6}{\alpha+2} (1-r_+^2)^{-(\alpha+1)} \gamma \right]
    \end{align*}
    assuming that $\eps$ is small as in \eqref{eqn:SmallnessOfEps}. Indeed, we may bound
    \begin{align*}
        r_+^2 - r_-^2 &\leq \vhi_{s,\alpha}(\eps) \coloneqq \left[ \of{1+\frac{s}{\pi}}^{-(\alpha+2)} + 6\eps \right]^{\inv{\alpha+2}} - \left[ \of{1+\frac{s}{\pi}}^{-(\alpha+2)} - 6\eps \right]^{\inv{\alpha+2}} \leq \frac{24\eps}{\alpha+2} \of{1+\frac{s}{\pi}}^{\alpha+1}
    \end{align*}
    by considering
    \[ \phi_{s,\alpha}(\eps) \coloneqq \vhi_{s,\alpha}(\eps) - \frac{24\eps}{\alpha+2} \of{1+\frac{s}{\pi}}^{\alpha+1} .\]
    Notice that $\phi_{s,\alpha}(0) = 0$ and the derivative $\phi_{s,\alpha}'(\eps)$ is negative if
    \begin{equation}\label{eqn:InSetProof4}
        \left[ \of{1+\frac{s}{\pi}}^{-(\alpha+2)} +6\eps \right]^{-\frac{\alpha+1}{\alpha+2}} + \left[ \of{1+\frac{s}{\pi}}^{-(\alpha+2)} -6\eps \right]^{-\frac{\alpha+1}{\alpha+2}} < 4 \of{1+\frac{s}{\pi}}^{\alpha+1} ,
    \end{equation}
    but the function on the left-hand side is increasing for all $\eps > 0$, and checking that \eqref{eqn:InSetProof4} is satisfied for $\eps$ as in \eqref{eqn:SmallnessOfEps} equates to choosing $c >0$ small enough such that
    \[ \of{1+6c \frac{C(\alpha+2)}{1+C(\alpha+2)}}^{-\frac{\alpha+1}{\alpha+2}} + \of{1-6c \frac{C(\alpha+2)}{1+C(\alpha+2)}}^{-\frac{\alpha+1}{\alpha+2}} < 4 .\]
    For instance, letting $c \leq 1/10$ suffices.  Setting $\eps = M_{\alpha}(s)^{1/2} \delta(f;\Om, \alpha)^{1/2}$ (recall \eqref{eqn:defnOfM}) we get
    \[ \mu(\Om \Delta A_{\Om}) \leq C \of{\frac{\pi}{\alpha+1}M_{\alpha}(s)^{-1/2}\theta_{\alpha}(s) + \frac{1}{\alpha+2}\of{1+\frac{s}{\pi}}^{3\alpha+5}M_{\alpha}(s)^{1/2}} \delta(f;\Om,\alpha)^{1/2} \]
    for some explicitly computable constant $C$. The choice of $\eps$, however, imposes a \emph{small deficit condition} on $\delta(f;\Om,\alpha)$ when coupled with the requirement \eqref{eqn:SmallnessOfEps}. We will deal with this at the end of the proof.

    \vspace{2mm}

    \noindent\underline{\textsc{Step III}. \textit{Concluding the proof in the small deficit regime}.} We finally compare $\Om$ to the ball 
    $$S_{\Om} = \inset{v > v^*(s)}.$$
    By \eqref{eqn:Inclusions} we have $S_{\Om} \subset E_{\Om}$ and
    \[ \mu(E_{\Om} \setminus S_{\Om}) \leq \frac{C}{\alpha+2} \of{1+\frac{s}{\pi}}^{3 \alpha+5} M_{\alpha}(s)^{1/2} \delta(f;\Om,\alpha)^{1/2} .\]
    Then,
    \begin{align*}
        \mu(\Om \Delta S_{\Om}) &\leq \mu(\Om \setminus E_{\Om}) + \mu(E_{\Om} \setminus S_{\Om}) + \mu(S_{\Om} \setminus \Om) \\
        &\leq \mu(\Om \setminus E_{\Om}) + \mu(E_{\Om} \setminus S_{\Om}) + \mu(E_{\Om} \setminus \Om) \\
        &\leq \mu(E \Delta \Om) + \frac{C}{\alpha+2} \of{1+\frac{s}{\pi}}^{3\alpha+5} M_{\alpha}(s)^{1/2} \delta(f;\Om,\alpha)^{1/2}.
    \end{align*}
    We then estimate
    \begin{align*}
        \mu(E_{\Om} \Delta \Om) &\leq \mu(E_{\Om} \setminus A_{\Om}) + \mu(A_{\Om} \setminus \Om) + \mu(\Om \setminus A_{\Om}) \\
        &\leq \of{ \frac{C}{\alpha+2} \of{1+\frac{s}{\pi}}^{3\alpha+5} +C \of{\frac{\pi}{\alpha+1} \cdot \frac{\theta_{\alpha}(s)}{M_{\alpha}(s)} + \inv{\alpha+2}\of{1+\frac{s}{\pi}}^{\alpha+3}} } M_{\alpha}(s)^{1/2} \delta(f;\Om,\alpha)^{1/2} \\
        &\leq C \of{\frac{\pi}{\alpha+1} \cdot \frac{\theta_{\alpha}(s)}{M_{\alpha}(s)} + \inv{\alpha+2}\of{1+\frac{s}{\pi}}^{3\alpha+5}} M_{\alpha}(s)^{1/2} \delta(f;\Om,\alpha)^{1/2} .
    \end{align*}
    We therefore conclude that
    \[ \mu(\Om \Delta S_{\Om}) \leq C \of{\frac{\pi}{\alpha+1} \cdot \frac{\theta_{\alpha}(s)}{M_{\alpha}(s)} + \inv{\alpha+2}\of{1+\frac{s}{\pi}}^{3\alpha+5}} M_{\alpha}(s)^{1/2} \delta(f;\Om,\alpha)^{1/2} ,\]
    and thus
    \[ \mcalA_{\disk} (\Om) \leq C s^{-1} \of{\frac{\pi}{\alpha+1} \cdot \frac{\theta_{\alpha}(s)}{M_{\alpha}(s)} + \inv{\alpha+2}\of{1+\frac{s}{\pi}}^{3\alpha+5}} M_{\alpha}(s)^{1/2} \delta(f;\Om,\alpha)^{1/2} .\]
    
    \noindent\underline{\textsc{Step IV}. \textit{Concluding the proof in the large deficit regime}.} The choice $\eps = M_{\alpha}(s)^{1/2} \delta(f;\Om,\alpha)^{1/2}$ together with \eqref{eqn:SmallnessOfEps} implies that
    \[ \delta(f;\Om,\alpha)^{1/2} < c M_{\alpha}(s)^{-1/2} \left[ \frac{C(\alpha+2)}{1+C(\alpha+2)} \right] \of{1+\frac{s}{\pi}}^{-(\alpha+2)} = \tilde{\delta}^{1/2}.\]
    Now, since $\mcalA_{\disk}(\Om) \leq 2$, we can estimate
    \[ \mcalA_{\disk}(\Om) \leq 2 \of{\frac{\delta(f;\Om,\alpha)}{\tilde{\delta}}}^{1/2} ,\]
    whenever $\delta(f;\Om,\alpha) > \tilde{\delta}$. We complete the proof by combining the estimates in both the large and small deficit regimes. Renaming
    \[ C(s,\alpha) = C s^{-1} \of{\frac{\pi}{\alpha+1} \cdot \frac{\theta_{\alpha}(s)}{M_{\alpha}(s)} + \inv{\alpha+2}\of{1+\frac{s}{\pi}}^{3\alpha+5}} ,\]
    we can set $K(s,\alpha) = \max \inset{C(s,\alpha), 2 \tilde{\delta}^{-1/2}}$ to find
    \[ \mcalA_{\disk}(\Om) \leq K(s,\alpha) \delta(f;\Om,\alpha)^{1/2} . \qedhere\]
\end{proof}

\begin{remark}\label{rmk:K-value} The proof above shows us that we have a bound of the form 
\begin{align}\label{eqn:upperboundK}
K(s,\alpha) \le C \cdot s^{-1} &\left(\frac{\pi}{\alpha+1} \frac{\theta_{\alpha}(s)}{M_{\alpha}(s)} + \frac{1}{\alpha+2}\left( 1 + \frac{s}{\pi}\right)^{3\alpha+5}\right) \cr 
&\hspace{1cm} + 2 \cdot M_{\alpha}(s)^{1/2} \left[ \frac{1 + C(\alpha+2)}{C(\alpha+2)}\right] \left( 1 + \frac{s}{\pi}\right)^{\alpha +2},
\end{align}
where we recall that $M_\alpha$ is defined as in \eqref{eqn:defnOfM}. 
\end{remark}

\section{Sharpness}\label{sec:Sharpness}

In this section, we discuss the sharpness of Theorems \ref{thm:main-wavelet} and \ref{thm:BergmanStability}. 

\subsection{Variational analysis}\label{subsec:GeometryOfSuperLevelSets}

In this section, we shall lay out a basic road map for an alternative proof to Theorems \ref{thm:main-wavelet} and \ref{thm:BergmanStability}. We mention that the basic idea here is similar to the one employed in \cite[Section~5]{Inv2}: proving that `reasonable' level sets are given by perturbation of those of a Gaussian under the hypothesis that our functions are themselves perturbations of a Gaussian, and then employing calculus of variations tools in order to explicitly compute the second variation around the nearest extremals. 

We start by first summarizing the relevant results in our context regarding the geometry of level sets of functions of the form $u(z) = |f(z)|^2 (1-|z|^2)^{\alpha+2},$ in analogy to \cite[Section~3]{Inv2}. 

More explicitly, for a fixed number $t>0$, the next result studies geometric properties of the super-level sets $\{z\in \disk:u(z)>t\}$. In particular, we show that, above a certain threshold, all level sets of the function $u$ behave like those of $\sqrt{\frac{\alpha+1}{\pi}}(1-|z|^2)^{\alpha+2}$, as long as $\e(f)$ is sufficiently small. Since its proof is a standard adaptation of the proof in \cite[Lemma~3.2]{Inv2}, we decided not to include it here - we refer the reader to that result, and in general to \cite[Section~3]{Inv2}, for further details. 

\begin{numlemma}\label{lemma:star-shaped-levels}
Let $f \in \Berg$ satisfy the normalizations in the beginning of subsection~\ref{subsec:SetStability}. One may find a small constant $c_2> 0$ such that, for $t > \e(f)^{c_2},$ the level sets 
$$\{z \in \disk \colon u(z) > t \}$$
are all star-shaped with respect to the origin. Moreover, for such $t,$ the boundary $\partial\{ u > t\} = \{ u = t \}$ is a smooth, closed curve, which may be parametrized as a graph over a disk centered at the origin. 
\end{numlemma} 

Given this information on level sets, fix now $s>0$ and consider the functional
\[\K_\alpha \colon \Berg \to \R, \qquad \K_\alpha [f] := \frac{I_f(s)}{\|f\|_{\Berg}^2},\]
where we let, for fixed $f \in \Berg,$ 
\[ I_f(s) = \int_{\{u = u^*(s)\}} u(z) \diff \mu(z),\]
where $u(z) = |f(z)|^2(1-|z|^2)^{\alpha+2}$. We will prove the following result:

\begin{numthm}\label{thm:optimalperturbative}
For a given $s\in (0,\infty)$, there are explicit constants $\e_0(s;\alpha),C(s;\alpha)>0$ such that, for all $\e \in (0,\e_0),$ we have
\[ \K_\alpha [\ff_{0}] - \K_\alpha [\ff_{0}+ \e g]  \geq C(s;\alpha)\e^2, \quad \ff_{0} = \sqrt{\frac{\alpha+1}{\pi}} , \]
whenever $\|g\|_{\Berg}=1$ satisfies \eqref{eq:normG}.    
\end{numthm}

The proof of Theorem \ref{thm:optimalperturbative} is based on the following technical result:

\begin{numlemma}\label{lemma:control3rdder}
    There is $\e_0=\e_0(s,\alpha)$ and a modulus of continuity $\eta$, depending only on $s$, such that
        \[|\K_\alpha [\ff_{0}+ \e g] - \K_\alpha [\ff_{0}]| \le  \left| \frac{\e^2}{2}\nabla^2 \K_\alpha [\ff_{0}](g,g)\right| + \eta(\e)\e^2 \]
    for all $0\leq \e \leq \e_0(s,\alpha)$ and $g\in \Berg$ such that $\|g\|_{\Berg}=1$ and which satisfy \eqref{eq:normG}. Here we have defined
    \begin{equation*}\label{eq:definition-second-variation-K}
        \nabla^2 \K_\alpha [\ff_{0}](g,g) \coloneqq \frac{\diff^{\,2}} 
        {\diff \e^{\,2}}  \K_\alpha [\ff_{0}+\e g]\Big|_{\e = 0}.
    \end{equation*}
\end{numlemma}

The proof of Lemma \ref{lemma:control3rdder} is purely technical, for which reason we refer the reader to Appendix A in \cite{Inv2} for the techniques needed in order to prove it. Our next result shows, additionally, that the second variation is \emph{uniformly} negative whenever $g \in \Berg$ with $\|g\|_{\Berg} = 1$. 

\begin{numprop}\label{prop:negdef}
For all $g \in \Berg$ such that $\|g\|_{\Berg}=1$ and which satisfy \eqref{eq:normG}, we have
    $$\frac 1 2\nabla^2 \K_\alpha[\ff_0](g,g)\leq - \frac{\pi^{\alpha+2}}{2} s (\pi + s)^{-3 - 
  \alpha} (s+2\pi) .$$
\end{numprop}

It is clear that Theorem \ref{thm:optimalperturbative} is an immediate consequence of the above two results (cf. \cite[Proof~of~Theorem~5.1]{Inv2}). The rest of this section is dedicated to the proof of Proposition \ref{prop:negdef}. We start by first computing the second variation of $\mc K_\alpha$. We hence consider - in analogy to \cite[Section~5]{Inv2} - the sets
\begin{equation}
    \label{eq:defOmegae}
    \Omega_\e := \{u_\e>u_\e^*(s)\}, \qquad u_\e :=u_{\ff_0+\e g} = |\ff_0+\e g|^2 (1-|\cdot|^2)^{\alpha+2},
\end{equation}
and write 
\[ \Omega_\e = \Phi_\e(\Omega_0), \]
for a suitable volume-preserving flow $\Phi_\e$. In order to construct such a flow, we recall a general lemma which allows us to build a flow that deforms the unit disk into a given family of graphical domains over the unit circle (see, for instance, \cite[Theorem~3.7]{Acerbietal} for a more general statement, and \cite[Lemma~5.4]{Inv2} for a proof of the result below). 

\begin{numlemma}\label{lemma:flows} 
Suppose that we are given a one-parameter family $\{D_\e\}_{\e \in [0,\e_0]}$ of domains, whose boundaries are given by smooth graphs over the unit circle:
$$\p D_\e = \{(1+g_\e(\omega))\omega:\omega\in \mb S^1\}.$$
Suppose, additionally, that the family $\{g_\e\}_{\e \in [0,\e_0]}$ depends smoothly on $\e.$ 

Then there exists a family $\{Y_\e\}_{\e \in [0,\e_0]}$ of smooth vector fields, which depends smoothly on the parameter $\e,$ such that, if $\Psi_\e$ denotes the flow associated with $Y_\e$, i.e. if
$$\frac{\diff}{\diff \e} \Psi_\e = Y_\e(\Psi_\e),$$
 then $\Psi_\e(\disk) = D_\e.$ In addition, $Y_\e$ is such that $\ddiv(Y_\e) = 0$ in a neighborhood of $\mb S^1.$
\end{numlemma}

We are now able to apply Lemma \ref{lemma:flows} to the sets $\Omega_\e$:

\begin{numlemma}
    \label{lemma:flows-specific}
Let $g\in \Berg$ satisfy \eqref{eq:normG}. 
There is $\e_0 =\e_0(s,\|g\|_{\Berg},\alpha)> 0$ such that, for all $\e\in [0,\e_0],$ there are globally defined smooth vector fields  $X_\e$, with associated flows $\Phi_\e$, such that 
$$\Omega_\e = \Phi_\e(\Omega_0).$$ 
Moreover, $X_\e$ depends smoothly on $\e$ and is divergence-free in a neighborhood of $\partial \Omega_0$. We also have
\begin{equation}
\label{eq:zeromeanXeps}
\int_{\partial \Omega_\e} \langle X_{\varepsilon} , \nu_\e \rangle (1-|z|^2)^{-2} \, \diff \mathcal{H}^1(z) = 0,
\end{equation}
where $\nu_\e$ denotes the outward-pointing unit vector field on $\p \Omega_\e.$
\end{numlemma}

\begin{proof}
Suppose that $\Omega_0 = \tau \cdot \disk,$ for some $\tau \in (0,1).$ Lemma \ref{lemma:star-shaped-levels} shows that, if $\e_0$ is chosen sufficiently small, the boundaries $\p\Omega_\e$ are smooth and the sets $\Omega_\e$ are star-shaped with respect to zero, hence they can be written as graphs over $\mb S^1$:
\begin{equation}
    \label{eq:deffeps}
    \p \Omega_\e = \{ (\tau+\rho_\e(\omega)) \, \omega \colon \omega \in \mb S^1\}.
\end{equation}
We now claim that the function $(\e,\omega) \mapsto \rho_\e(\omega)$ is \emph{smooth} as long as $\e$ is sufficiently small. 

Indeed, for fixed $\e,$ the function $\omega \mapsto \rho_{\e}(\omega)$ is smooth, by Lemma \ref{lemma:star-shaped-levels}, since it is implicitly defined by $u_{\e}((\tau+\rho_\e(\omega))\cdot \omega) = u_\e^*(s).$ Moreover, since $\nabla u_{\e}$ is \emph{bounded} by a constant depending only on $s, \alpha$ when restricted to $\{u_\e = u_\e^*(s)\}$ (this follows from an explicit computation, similar to the ones undertaking in \cite[Proof~of~Lemma~3.1]{Inv2}), a careful quantified version of the implicit function theorem (cf. \cite{Krantz}) implies that there is a universal $\e_0(s,\alpha) > 0$ such that, if $\e < \e_0(s,\alpha),$ then $\e \mapsto \rho_\e(\omega)$ is smooth for any fixed $\omega \in \mathbb{S}^1.$ This proves the desired smoothness claim.

 In order to prove \eqref{eq:zeromeanXeps}, we note that since $\Omega_{\varepsilon}$ has constant \emph{hyperbolic} measure equal to $s$ for all $\varepsilon\in [0,\e_0]$,  we have that 
\begin{align*}
0 = \frac{\diff}{\diff \e} \left(\mu\left(\Om_\e\right)\right) = \int_{\partial \Om_\e} \langle X_\e, \nu_\e \rangle (1-|z|^2)^{-2} \diff \mathcal{H}^1(z),
\end{align*}
where we used Reynold's formula in order to compute derivatives above. Thus \eqref{eq:zeromeanXeps} follows at once, concluding our proof. 
\end{proof}

We are now able to explicitly compute the second variation of $\K_\alpha$ along a direction given by a fixed function $g \in \Berg$:  

\begin{numlemma}\label{lemma:second-variation-K}
For all $g\in \Berg$ which satisfy \eqref{eq:normG}, we have
    \begin{align}\label{eq:second-variation} 
 \frac 1 2\nabla^2 \K_\alpha [\ff_{0}](g,g)
 & = \int_{\Omega_0} |g(z)|^2(1-|z|^2)^{\alpha+2}\, \diff \mu(z) - \frac{(\alpha+1) \|g\|_{\Berg}^2}{\pi} \int_{\Omega_0} (1-|z|^2)^{\alpha+2} \, \diff \mu(z) \cr
 & + \frac{\pi}{\alpha+2} \fint_{\p \Omega_0} |g(z)|^2 (1-|z|^2)^{\alpha+1} \diff \mc H^1(z).
\end{align}
\end{numlemma}


\begin{proof}
We first set $\diff \nu_\alpha(z) := (1-|z|^2)^{\alpha} \diff z$ for brevity. We thus introduce the auxiliary functions
\begin{equation}
    \label{eq:defIJ}
    I_\e^{\alpha} := \int_{\Omega_\e} |\ff_{0}+\e g|^2 \diff \nu_\alpha, \qquad J_{\e}^{\alpha}:=\int_\disk |\ff_{0}+\e g|^2 \diff \nu_\alpha;
\end{equation}
we write additionally $K_\e^{\alpha} \coloneqq \K_\alpha[\ff_{0}+\e g]$, where we recall  that $\Omega_\e \coloneqq \{u_\e>u_\e^*(s)\}$. We will always take $\e\leq \e_0$, where $\e_0$ is given by Lemma \ref{lemma:flows-specific}. We need to compute the derivatives of those quantities: arguing similarly as in the proof of \cite[Lemma~5.6]{Inv2}, we have 
\begin{align}
\begin{split}
\label{eq:derK}
\pe K_\e^\alpha & = \Big(\pe I_\e^\alpha - I_\e^\alpha \frac{\pe J_\e^\alpha }{J_\e^\alpha}\Big) \frac{1}{J_\e^\alpha},\\
\pe^2 K_\e^\alpha &= \Big(\pe^2 I_\e^\alpha - I_\e^\alpha \frac{\pe^2 J_\e^\alpha}{J_\e^\alpha}\Big) \frac{1}{J_\e}- \frac{2 \pe J_\e^\alpha}{(J_\e^\alpha)^2} \Big(\pe I_\e^\alpha -I_\e^\alpha \frac{\pe J_\e^\alpha}{J_\e^\alpha}\Big),
\end{split}
\end{align}
and using Reynold's theorem, we have explicitly
\begin{align}
\label{eq:derIJ}
\pe I_\e^\alpha & = 2 \int_{\Omega_\e} \text{Re}(g \cdot \overline{\ff_{0}+\e g}) \diff \nu_\alpha,  \\ 
\pe J_\e^\alpha  & = 2\int_{\disk} \text{Re} (g \cdot \overline{\ff_{0}+ \e g}) \diff \nu_\alpha,\\
\label{eq:derIJsecond}
\pe^2 I_\e^\alpha & = 2 \int_{\Omega_\e} |g|^2 \diff \nu_\alpha + 2\int_{\p \Omega_\e} \text{Re}(g \cdot \overline{\ff_{0}+\e g}) \langle X_\e,\nu_\e\rangle (1-|z|^2)^{\alpha}, \\
 \pe^2 J_\e^\alpha  & = 2 \int_{\disk} |g|^2 \diff \nu_\alpha.
\end{align}
Here, $X_\e$ denotes the vector fields built in Lemma \ref{lemma:flows-specific}. Note that, in order to obtain \eqref{eq:derIJ}, we used that $u_\e$ is constant on $\p \Omega_\e$, together with \eqref{eq:zeromeanXeps}. 

Since $\langle g, 1\rangle_{\Berg}=0$, $\Omega_0$ is a ball and $g$ is holomorphic, from \eqref{eq:derIJ} it follows that 
\begin{equation}
    \label{eq:derK0}
    \pe\Big|_{\e = 0} I_\e^\alpha =\pe\Big|_{\e = 0} J_\e^\alpha =0 \qquad \implies \qquad \frac{\diff}{\diff \e} \K_\alpha[\ff_{0}+\e g] \Big|_{\e=0} = \pe\Big|_{\e = 0} K_\e = 0,
\end{equation}
where the implication follows from the first equation in \eqref{eq:derK}. Combining thus \eqref{eq:derK}--\eqref{eq:derK0}, we obtain that
\begin{align}\label{eqn:second-variation-first-formula} 
 \frac{\diff^{\,2}}{\diff \e^{\,2}}\mc K_\alpha[\ff_{0}+\e g]\Bigg|_{\e = 0 } 
 & =2\left( \int_{\Omega_0} |g|^2(1-|z|^2)^{\alpha} -  \frac{(\alpha+1) \|g\|_{\Berg}^2}{\pi} \int_{\Omega_0} (1-|z|^2)^\alpha \right. \cr 
 & \left. + \int_{\p \Omega_0} \Re(g(z) \ff_{0})  \langle X_0, \nu\rangle (1-|z|^2)^{\alpha}  \right).
\end{align}
Our main task is to simplify the last term: we first claim that
\begin{equation}
\label{eq:dermu}
\frac{\diff}{\diff \e}\Big|_{\e=0} \rho_\e(t) = 0,
\end{equation}
where $\rho_\e(t) \coloneqq \rho_{\ff_{0}+\e g}(t) =\mu\left(\{u_\e>t\}\right)$.
To prove this claim we build, as in Lemma \ref{lemma:flows-specific}, a family of vector fields $Y_\e$ with associated flows $\Psi_\e$ such that $\Psi_\e(\{u_0>t\}) = \{u_\e>t\}$. It follows from Lemma \ref{lemma:star-shaped-levels} that these sets have smooth boundaries, and hence we are able to apply Lemma \ref{lemma:flows}.  We compute, for $z \in \partial \{ v > t \} = \{v = t\},$
\begin{equation}\label{eqn:boundary-vector-expansion} 
t \equiv u_\e(\Psi_\e(z))= \left(\ff_{0}^2 + 2 \e\ff_{0}\Re g(z)- 2 \ff_0^2 \e \frac{ (\alpha+2) \langle Y_0(z),z\rangle}{1-|z|^2} + O(\e^2)\right )(1-|z|^2)^{\alpha+2},
\end{equation}
and thus, since the first order term in $\e$ vanishes, we have
\begin{equation}\label{eqn:condition-vector-field-boundary} 
\Re g(z) = \ff_{0}\cdot (\alpha+2) \langle Y_0,z\rangle \cdot \left((1-|z|^2)\right)^{-1}. 
\end{equation}
We can now prove \eqref{eq:dermu}: again by Reynold's formula, we have
\begin{equation}\label{eqn:condition-derivative-measure}
\frac{\diff}{\diff \e}\Big|_{\e=0} \rho_\e(t) = \int_{\p\{v>t\}} \langle Y_0,\nu\rangle 
= 2\pi \fint_{\p\{v>t\}} \langle Y_0,z\rangle 
= \frac{2\ff_{0}\pi}{\alpha+2} \fint_{\p\{v>t\}} \left( \Re g(z) \right) \cdot \left((1-|z|^2)\right)^{-1} = 0,
\end{equation}
since $\p\{v>t\}$ is a centered circle and  $\Re g$ is harmonic and vanishes at $0$, where the last equality follows from \eqref{eq:normG}.

Now let us fix $s>0$ and recall that $g(0)=0$. Using the smoothness of $\e \mapsto \rho_\e$ first and then the smoothness of $\rho_0$ on a neighborhood of $v^*(s),$ we obtain:  
\begin{align*}
s & = \rho_\e(u_\e^*(s))  = \rho_0(u_\e^*(s))+ 2\e \ff_{0} \Re(g(0)) + O(\e^2)\\
 & = s + (u_\e^*(s)-v^*(s))\frac{\diff}{\diff t}\Big|_{t=v^*(s)} \rho_0(t) + O(\e^2) \\
 & = s - (u_\e^*(s) - v^*(s)) \cdot \frac{\diff}{\diff t} \Big|_{t=v^*(s)} \left( \frac{1-t^{\frac{1}{\alpha+2}}}{t^{\frac{1}{\alpha+2}}} \right) + O(\e^2) \\ 
 & = s + \frac{u_\e^*(s) - v^*(s)}{(\alpha+2)\left(v^*(s)\right)^{\frac{\alpha+3}{\alpha+2}}} + O(\e^2)
\end{align*}
where we used the fact that $g(0)=0$ by \eqref{eq:normG}.
Thus, after rearranging, we find
$$u_\e^*(s) =  (1 + O_s(\e^2))v^*(s).$$
Since $\Phi_\e$ is the flow of $X_\e$, we have
\begin{equation}
    \label{eq:expansionPhieps}
    \Phi_\e(z) = \Phi_0(z) + \e X_0(\Phi_0(z)) + O(\e^2) = z+ \e X_0(z) + O(\e^2).
\end{equation}
Comparing the expressions
\begin{align*}
u_\e(\Phi_\e(z)) & = \left(\ff_{0}^2 + 2 \e \ff_{0}\Re g(z)- 2\e \ff_{0}^2\frac{(\alpha+2) \langle X_0(z),z\rangle}{1-|z|^2} + O(\e^2) \right)(1-|z|^2)^{\alpha+2} ,\\
 u_\e^*(s) & = (1 + O(\e^2)) v^*(s),
\end{align*}
we deduce that, on $\p \{v>v^*(s)\} = \p \Omega_0$, the first order terms in $\e$ must be the same, thus
\begin{equation}\label{eq:vector-field-X_0} 
\ff_0(\alpha+2) \cdot \frac{\langle X_0,z\rangle}{1-|z|^2} = \Re g(z) \text{ on } \partial \Omega_0. 
\end{equation}
Finally, since $g$ is holomorphic and $g(0)=0$, we have
\begin{align*}
\fint_{\p \Omega_0} (\Re g)^2 \diff \mc H^1 
 = \fint_{\p \Omega_0} \Re(g^2) \diff \mc H^1+ \fint_{\p \Omega_0} (\Im g)^2 \diff \mc H^1
=  \fint_{\p \Omega_0} (\Im g)^2 \diff \mc H^1.
\end{align*}
Now, we simply note that
\begin{align*}
   \int_{\p \Omega_0} (\Re g) \langle X_0,\nu \rangle \cdot (1-|z|^2)^{\alpha} \diff \mc H^1 
    & = \frac{2 \pi}{\ff_0(\alpha+2)} \fint_{\p \Omega_0} (\Re g)^2 \cdot (1-|z|^2)^{\alpha+1} \diff \mc H^1 \\
    & = \frac{\pi}{\ff_0(\alpha+2)} \fint_{\p \Omega_0} |g(z)|^2 (1-|z|^2)^{\alpha+1} \diff \mc H^1,
\end{align*}
and plugging back into the \eqref{eqn:second-variation-first-formula} yields the claim, as desired. 
\end{proof}

We are finally ready to prove Proposition \ref{prop:negdef}. We start off by noting that 
 \eqref{eq:second-variation} can be rewritten using the power series for $g$ and \eqref{eq:normG}: 
\begin{align}\label{eq:expansion-second-variation-coeff} 
\begin{split}
 \frac 1 2\nabla^2 \K[\ff_{0}](g,g) =& \sum_{k=2}^\infty |a_k|^2 W_k(s),
\end{split}
\end{align} 
where
\begin{align}\label{eq:definition-W-k} 
W_k(s) & \coloneqq \frac{1}{c_k} \int_{B(0,\sqrt{\frac{s}{s+\pi}})} |z|^{2k} (1-|z|^2)^{\alpha+2} \, \diff \mu(z)
- \frac{\pi}{\alpha+1} \int_{B(0,\sqrt{\frac{s}{s+\pi}})} (1-|z|^2)^{\alpha+2} \, \diff \mu(z) \cr 
&+ \frac{\pi}{(\alpha+2)c_k} \cdot \left( 1 + \frac{s}{\pi}\right)^{-(\alpha+1)}\cdot \left( \frac{s}{s+\pi}\right)^{k}. 
\end{align} 
The first term in the right-hand side of \eqref{eq:definition-W-k} may be explicitly computed to be 
\[
\frac{\pi}{ c_k} \mathbb{B}(r^2;k+1,\alpha+1),
\]
where $\mathbb{B}$ denotes the incomplete Beta function: 
\[
\mathbb{B}(\sigma;a,b) = \int_0^\sigma t^{a-1} \cdot (1-t)^{b-1} \, \diff t,
\]
and $r \coloneqq \sqrt{\frac{s}{s+\pi}}$. Now we use the following relation satisfied by the incomplete Beta function (see \cite[(8.17.9)]{NIST:DLMF}): 
\begin{align*} 
\mathbb{B}(r^2;k+1,\alpha+1) & = \frac{\mathbb{B}(k+1,\alpha+1)}{\mathbb{B}(k,\alpha+1)} \cdot \mathbb{B}(r^2;k,\alpha+1) - \frac{\mathbb{B}(k+1,\alpha+1)}{k \mathbb{B}(k,\alpha+1)} \cdot r^{2k}(1-r^2)^{\alpha+1} \cr 
    & = \frac{\Gamma(k+1)\Gamma(\alpha+1)}{\Gamma(k)\Gamma(\alpha+1)} \cdot \frac{\Gamma(k+\alpha+1)}{\Gamma(k+\alpha+2)} \cdot \mathbb{B}(r^2;k;\alpha+1) \cr 
    & - \frac{\Gamma(k+1)\Gamma(\alpha+1)}{k \Gamma(k) \Gamma(\alpha+1)} \cdot \frac{\Gamma(k+\alpha+1)}{\Gamma(k+\alpha+2)} \cdot r^{2k}(1-r^2)^{\alpha+1}.
\end{align*} 
Here, we have denoted the usual Beta function by $\mathbb{B}(a,b) \coloneqq \mathbb{B}(1;a,b).$ Now, since $c_k = \frac{\Gamma(\alpha+1)\Gamma(k+1)}{\Gamma(k+\alpha+2)} \pi = \pi \mathbb{B}(\alpha+1,k+1)$, we have 
\begin{align*}
\frac{\pi}{c_k} \mathbb{B}(r^2;k+1,\alpha+1) = \frac{\pi}{c_{k-1}} \mathbb{B}(r^2;k,\alpha+1) - \frac{\pi}{k c_{k-1}} \cdot r^{2k}(1-r^2)^{\alpha+1}. 
\end{align*}
Hence, we have that
\begin{align*}
W_k(s) & = \frac{\pi}{c_{k-1}} \mathbb{B}(r^2;k,\alpha+1) - \frac{\pi}{\alpha+1} \int_{B(0,\sqrt{\frac{s}{s+\pi}})} (1-|z|^2)^{\alpha+2} \, \diff \mu(z) \cr 
 & + \frac{\pi}{(\alpha+2)c_{k-1}} \left( \frac{k-1}{k} \right)  \left( 1 + \frac{s}{\pi}\right)^{-(\alpha+1)}\cdot \left( \frac{s}{s+\pi}\right)^{k} \cr 
 & < \frac{\pi}{c_{k-1}} \mathbb{B}(r^2;k,\alpha+1) - \frac{\pi}{\alpha+1} \int_{B(0,\sqrt{\frac{s}{s+\pi}})} (1-|z|^2)^{\alpha+2} \, \diff \mu(z) \cr 
 & + \frac{\pi}{(\alpha+2)c_{k-1}}  \left( 1 + \frac{s}{\pi}\right)^{-(\alpha+1)}\cdot \left( \frac{s}{s+\pi}\right)^{k-1} = W_{k-1}(s). 
\end{align*}
In order to conclude the desired bound, notice that it suffices to evaluate $W_2(s)$ explicitly, since all other coefficients are smaller than that. But this is a direct computation: 
\[
W_2(s) = - \frac{\alpha+1}{2} \pi^{\alpha+1} s (\pi + s)^{-3 - 
  \alpha} (s+2\pi).
\]
The conclusion of Proposition \ref{prop:negdef} follows then directly from \eqref{eq:expansion-second-variation-coeff}. 

\subsection{Stability Analysis}

We now prove that Theorems \ref{thm:main-wavelet} and \ref{thm:BergmanStability} are \emph{sharp}, both in terms of the order of stability obtained, and in terms of the growth of the stability constant with respect to the hyperbolic measure of the sets under consideration. Moreover, we shall see that the results are aslo \emph{sharp} when $\alpha \to -1$, in the sense that the constants obtained are best possible in terms of their dependence on $\alpha$ when $\alpha \to -1$. 

In order to prove the desired sharpness results, we start with a key proposition, which is directly related to the latter issue of sharpness as the Bergman parameters converge to $-1$. 

\begin{numprop}\label{prop:ComputationsSharpness}
    Let $\alpha > -1, \, s>0$ and $B \subset \disk$ a ball centered at the origin with $\mu(B)=s$. Then there is a constant $c(s) > 0$ independent of $\alpha$ such that the function
    \[ f_{\alpha}(z) = \frac{a_{0}}{\sqrt{ c_{0} }} + \frac{a_{2}}{\sqrt{ c_{2} }}z^2 \]
    satisfies
    \[
    \frac{\left\| f_\alpha - \sqrt{\frac{\alpha+1}{\pi}}\right\|_{\Berg}^2}{\|f_\alpha\|_{\Berg}^2 \cdot \delta(f_\alpha; B,\alpha)}  > c(s), \quad \text{as} \ \frac{a_2}{a_0} \to 0. 
    \]
\end{numprop}
\begin{proof}
    The norm of $f_{\alpha}$ is
    $\lVert f_{\alpha} \rVert_{\mathbf{B}_{\alpha}^2}^2 = a_{0}^2+a_{2}^2$,
    while its asymmetry reads
    \[ \frac{{\left\lVert  f_{\alpha} - \sqrt{ \frac{\alpha+1}{\pi} }  \right\rVert_{\mathbf{B}_{\alpha}^{2}}^2}}{\lVert f_{\alpha} \rVert_{\mathbf{B}_{\alpha}^2}^2} = \frac{{\left( a_{0}-\sqrt{ a_{0}^2 + a_{2}^2 } \right)^2 + a_{2}^2}}{a_{0}^2+a_{2}^2} = \frac{\left(1-\sqrt{ 1+\frac{a_{2}^2}{a_{0}^2} }\right)^2 + \frac{a_{2}^2}{a_{0}^2}}{1+\frac{a_{2}^2}{a_{0}^2}} = u^2 + o(u^2) ,\]
    where $ u= a_{2}/a_{0}$. On the other hand, in order to compute the deficit, we first need to compute the integral
    \[ I_{f_{\alpha}}(s) = \int_{B} \lvert f_{\alpha}(z) \rvert^2 (1-\lvert z \rvert^2)^{\alpha}  \, \mathrm{d}z , \quad \mu(B)=s ,\]
    where the radius of $B$ is denoted by $r$. To this end, notice that
    \[ \lvert f_{\alpha}(z) \rvert^2 = \frac{a_{0}^2}{c_{0}} + \frac{a_{0}^2}{c_{2}}\lvert z \rvert^4 + 2\frac{a_{0}a_{2}}{\sqrt{ c_{0}c_{2}}} \mathrm{Re}(z^2),\]
    and we can reduce $I_{f_{\alpha}}(s)$ to computing the integrals
    \begin{align*}
        I_{1} &= \int_{B} (1-\lvert z \rvert^2)^{\alpha}  \, \mathrm{d}z = 2 \pi \int_{0}^r \rho (1-\rho^2)^{\alpha} \, \mathrm{d}\rho ,\\
        I_{2} &= \int_{B} \lvert z \rvert^4 (1-\lvert z \rvert^2)^{\alpha}  \, \mathrm{d}z = 2 \pi \int_{0}^r \rho^5 (1-\rho^2)^{\alpha} \, \mathrm{d}\rho , \\
        I_{3} &= \int_{B} \mathrm{Re}(z^2) (1-\lvert z \rvert^2)^{\alpha}  \, \mathrm{d}z = 0 .
    \end{align*}
    This results in
    \begin{align*}
        I_{1} &= \frac{\pi}{\alpha+1} \left( 1-\left( 1+\frac{s}{\pi} \right)^{-(\alpha+1)} \right), \\
        I_{2} &= \pi \left[ \frac{1-\left( 1+\frac{s}{\pi} \right)^{-(\alpha+1)}}{\alpha+1} -2 \frac{1-\left( 1+\frac{s}{\pi} \right)^{-(\alpha+2)}}{\alpha+2} + \frac{1-\left( 1+\frac{s}{\pi} \right)^{-(\alpha+3)}}{\alpha+3} \right] ,
    \end{align*}
    and putting it all together,
    \[ I_{f_{\alpha}}(s)= \frac{a_{0}^2}{c_{0}} I_{1} + \frac{a_{2}^2}{c_{2}}I_{2} .\]
    Now, the deficit reads
    \[ \delta(f_{\alpha}; B, \alpha) = \frac{\lVert f \rVert_{\mathbf{B}_{\alpha}^2}^2 \theta_{\alpha}(s) - I_{f_{\alpha}}(s)}{\lVert f \rVert_{\mathbf{B}_{\alpha}^2}^2 \theta_{\alpha}(s)}, \quad \text{ where we recall that }\theta_{\alpha}(s) = 1-\left( 1+\frac{s}{\pi} \right)^{-(\alpha+1)} ,\]
    and by computing this we reach
    \begin{align}\label{eqn:DeficitExplicit}
        \delta(f_{\alpha}; B, \alpha) &= \frac{\lVert f \rVert_{\mathbf{B}_{\alpha}^2}^2 \theta_{\alpha}(s) - I_{f_{\alpha}}(s)}{\lVert f \rVert_{\mathbf{B}_{\alpha}^2}^2 \theta_{\alpha}(s)} \cr
        &= \frac{a_{2}^2}{a_{0}^2+a_{2}^2}\left( 1-\frac{(\alpha+2)(\alpha+3)}{2} \right) - \frac{a_{2}^2 \frac{(\alpha+1)(\alpha+2)(\alpha+3)}{2 \pi} \left(\frac{1-\left( 1+\frac{s}{\pi} \right)^{-(\alpha+3)}}{\alpha+3} -2 \frac{1-\left( 1+\frac{s}{\pi} \right)^{-(\alpha+2)}}{\alpha+2} \right)}{(a_{0}^2+a_{2}^2)\left( 1-\left( 1+\frac{s}{\pi} \right)^{-(\alpha+1)} \right)}\cr
        &= (\alpha+1)\frac{a_2^2}{a_0^2+a_2^2} \of{\frac{(\alpha+4)}{2} + \frac{(\alpha+2)(\alpha+3)}{2\pi} \cdot \frac{\left(2 \frac{1-\left( 1+\frac{s}{\pi} \right)^{-(\alpha+2)}}{\alpha+2} - \frac{1-\left( 1+\frac{s}{\pi} \right)^{-(\alpha+3)}}{\alpha+3} \right)}{\left( 1-\left( 1+\frac{s}{\pi} \right)^{-(\alpha+1)} \right)}}
    \end{align}
    For any fixed $\alpha > -1$, the deficit is then of the order of $a_2^2$. Moreover, if we write
    \begin{equation}\label{eqn:DeficitComputation}
        \delta(f_{\alpha};B,\alpha) = \frac{a_2^2}{a_0^2+a_2^2} c_{\alpha}(s) ,
    \end{equation}
    then $c_{\alpha}(s)$ converges to a bounded function $0 \leq c(s) \leq 1$ as $\alpha \to -1$. Indeed, the right-hand side of the last line of \eqref{eqn:DeficitExplicit} converges to 
    \[
    \frac{a_2^2}{a_0^2 + a_2^2} \frac{1}{\pi \log\left( 1+ \frac{s}{\pi}\right)} \cdot\left( 1 - \left( 1+\frac{s}{\pi}\right)^{-1} - \frac{1}{2}\left( 1- \left(1+\frac{s}{\pi}\right)^{-2}\right)\right) =: \frac{a_2^2}{a_0^2 + a_2^2} c(s),
    \]
    which is bounded by $1$ after a simple computation - as a matter of fact, this function is explicitly given by 
    \[
    c(s) = \frac{s^2}{2 \pi (s+\pi)^2 \cdot \log \left( 1 + \frac{s}{\pi}\right)},
    \]
    and it can be show that it has exactly one maximum point $s_0 \in (0,+\infty)$, and $c(s_0) < 0.1$. 
\end{proof}

We are now able to state and prove the main result of this section. 

\begin{numcor}[Sharpness]\label{cor:sharpness}
    The following hold true:
    \begin{enumerate}[font=\normalfont]
        \item For each fixed $\alpha > -1,$ the factor $\delta(f;\Om,\alpha)^{1/2}$ in \eqref{eqn:BergmanStabilityFunction} and \eqref{eqn:BergmanStabilitySet} cannot be replaced by $\delta(f;\Om,\alpha)^{\beta}$ for any $\beta > 1/2$. 
        \item The constant in \eqref{eqn:BergmanStabilityFunction} has a sharp dependency on $\alpha$, in the sense that there is no $c < 1$ such that \eqref{eqn:BergmanStabilityFunction} holds with constant proportional to 
        \begin{equation}\label{eqn:FakeConstant} 
        \Sigma(c,\alpha) \coloneqq \left( 1 + \frac{\alpha+2}{\alpha+1} \left[ \left( 1 + c \cdot \frac{s}{\pi}\right)^{\alpha+1} - 1\right]\right)^{1/2} 
        \end{equation} 
        for all $\alpha$ sufficiently large. More generally, for a given $\alpha \in (-1,+\infty)$, \eqref{eqn:BergmanStabilityFunction} can only hold with constant \eqref{eqn:FakeConstant} for some $c<1$ if $c \ge \left(\frac{\pi(\alpha+2)}{\alpha+1}\right)^{-\frac{1}{\alpha+1}}$. 
        
        \item Theorem \ref{thm:BergmanStability} is sharp when $\alpha \to -1$, in the sense that, there is a sequence of functions $f_{\alpha,\varepsilon}$ such that the right-hand and left-hand sides of \eqref{eqn:BergmanStabilityFunction} both behave as $C \varepsilon$, for some absolute constant $C>0$, as $\alpha \to -1, \varepsilon \to 0$.
    \end{enumerate}
\end{numcor}
\begin{proof}
First note that, if we take $f_{\alpha,\varepsilon} = \frac{1}{\sqrt{c_0}} + \frac{\varepsilon}{\sqrt{c_2}} z^2$, then Proposition \ref{prop:ComputationsSharpness} yields that there is $c(s) > 0$ such that 
    \[
    \frac{\left\|f_{\alpha,\varepsilon} - \sqrt{\frac{\alpha+1}{\pi}} \right\|_{\Berg}^2}{\|f_{\alpha,\varepsilon}\|_{\Berg}^2} \ge c(s) \delta(f_{\alpha,\varepsilon};B,\alpha),
    \]
    for all $\alpha > -1, \varepsilon>0$. Now, note additionally that, since
    \begin{align*}
     \inf_{\substack{\abs{c} = \norm{f_{\varepsilon,\alpha}}_{\Berg} ,\\ \om \in \disk}} \|f_{\alpha,\varepsilon} - c \ff_\om\|_{\Berg}^2 = 2\|f_{\alpha,\varepsilon}\|_{\Berg}^2 - 2 \sup_{\om \in \disk} \|f_{\alpha,\varepsilon} \|_{\Berg} \left| \langle f_{\alpha,\varepsilon}, \ff_\om \rangle_{\Berg} \right|,
    \end{align*} 
it is enough to show that $\sup_{\om \in \disk} \left| \langle f_{\alpha,\varepsilon}, \ff_\om \rangle_{\Berg} \right|$ is attained at $\om = 0$. Since $f_\om$ is an explicit multiple of the reproducing kernel for the Bergman space $\Berg$, we have to prove that
\[
\left| \langle f_{\alpha,\varepsilon}, \ff_\om \rangle_{\Berg} \right| = \sqrt{\frac{\pi}{\alpha+1}} |f_{\alpha,\varepsilon}(\om)| (1-|\om|^2)^{\frac{\alpha+2}{2}} \le 1,
\]
with equality if and only if $\om = 0$. But this follows from a Taylor expansion argument: more explicitly, since 
\begin{align*}
    |f_{\alpha,\varepsilon}(\om)|^2 &= \frac{1}{c_0} + 2 \frac{\varepsilon}{\sqrt{c_0 c_2}} \text{Re}(\om^2) + \frac{\varepsilon^2}{c_2}|\om|^4 \\
    &\leq \frac{\alpha+1}{\pi} + 2 \varepsilon \frac{\alpha+1}{\pi} \cdot \sqrt{\frac{(\alpha+2)(\alpha+3)}{2}} |\om|^2 + \varepsilon^2 \cdot \frac{\alpha+1}{\pi} \cdot \frac{(\alpha+2)(\alpha+3)}{2} |\om|^4, 
\end{align*}
and as we have through a simple Taylor expansion computation
\[
(1-|\om|^2)^{-(\alpha + 2)} > 1 + (2+\alpha) |\om|^2 + \frac{(2+\alpha)(3+\alpha)}{2} |\om|^4, 
\]
we conclude the desired inequality as long as $\varepsilon < \frac{1}{10}$, for instance. This finishes the proof of (iii), and incidentally also that of the part of (i) relative to the sharpness of \eqref{eqn:BergmanStabilityFunction}. 

In order to deal with the part of (i) related to the sharpness of \eqref{eqn:BergmanStabilitySet}, we again consider the functions $f_{\alpha,\e}$ as before, and take $\Omega_\e = \{ u_\e > u_\e^*(s) \}$, where $u_\e(z) = |f_{\alpha,\e}(z)|^2 \cdot (1-|z|^2)^{\alpha+2}$. Whenever $\e > 0$ is sufficiently small -- here, we allow the degree of how small $\e$ must be to depend on $\alpha$, since we are essentially interested only in the limit behavior of quantities as $\e \to 0$ -- we have that $\Om_\e$ is a perturbation of a centered disk. We are now able to employ the tools developed in Section \ref{subsec:GeometryOfSuperLevelSets}: we may use Lemma \ref{lemma:flows-specific} to write $\Omega_\e = \Phi_\e(\Omega_0)$, provided that $\e$ is small enough. As we saw in \eqref{eq:expansionPhieps} we may write 
$$\Phi_\e(z) = z + \e X_0(z) + O(\e^2),\qquad \text{where } X_0(z) = h_0(z) z,$$ for some scalar function $h_0\colon \C\backslash \{0\}\to \R$. Since $\frac{\alpha+2}{1-|z|^2} \cdot \langle X_0(z), z \rangle = \ff_{0} \Re(z^2)$ on $\partial \Omega_0$ by \eqref{eq:vector-field-X_0},  we have 
\[ h_0(z) = \ff_{0} \frac{(1-|z|^2) \cdot \Re(z^2)}{(\alpha+2) |z|^2} =
\frac{\cos(2 \theta)}{\pi}\]
for $z = r(\Omega_0)e^{i \theta} \in \partial \Omega_0,$ where $r(\Omega_0)$ denotes the radius of the ball $\Omega_0$. Hence, 
\[
\mu\left(\Omega_\e \triangle \Omega_0\right) \geq \mu\left( \Omega_\e \setminus \Omega_0 \right) \ge  \mu\left( \Big\{z=re^{i \theta} \colon r(\Omega_0)< r <r(\Omega_0)+\e\frac{\cos(2\theta)}{\pi} -C\e^2 \Big\}\right) > c \, r(\Omega_0)^2 \e, 
\]
which concludes the proof of this part.
    
Finally, for the proof of (ii), note that if \eqref{eqn:BergmanStabilityFunction} could hold with $\Sigma(c,\alpha)$ as in \eqref{eqn:FakeConstant} as constant, then using Proposition \ref{prop:negdef} and Lemma \ref{lemma:control3rdder} we would have that, for $f_{\alpha,\e}$ above,
\begin{align*} 
\Sigma(c,\alpha) \delta(f_{\alpha,\e};\Om_\e,\alpha)^{1/2} 
 & \ge  \inf_{\substack{\abs{c} = \norm{f}_{\Berg} ,\\ \om \in \disk}} \frac{\norm{f_{\alpha,\e} -c\ff_{\om}}_{\mathbf{B}_{\alpha}^2}}{\norm{f_{\alpha,\e}}_{\mathbf{B}_{\alpha}^2}} = c \cdot  \varepsilon \cr 
 & \ge \left( \frac{  \left( \K_\alpha[\ff_{0}] - \K_\alpha[f_{\alpha,\e}] \right)}{(\alpha+1)\pi^{\alpha+1} s(s+2\pi)(s+\pi)^{-3-\alpha}} \right)^{1/2} \cr 
 & = \left( \frac{(s+\pi)^2 \cdot \left( 1 + \frac{s}{\pi}\right)^{\alpha+1}}{(\alpha+1) s(s+2\pi)} \right)^{1/2} \theta_{\alpha}(s)^{1/2} \cdot \delta(f_{\alpha,\e};\Om_\e,\alpha)^{1/2}.
\end{align*} 
Here, Lemma \ref{lemma:control3rdder} and Proposition \ref{prop:negdef} were used in the second passage. We would then conclude that 
\begin{equation}\label{eqn:ContradictionConstants} 
\Sigma(c,\alpha)^2 \cdot \left(1+\frac{s}{\pi}\right)^{-\alpha-1}  \ge \frac{(s+\pi)^2}{s(s+2\pi)}  \cdot \frac{\theta_\alpha(s)}{\alpha+1}
\end{equation} 
holds for all $s > 0$. Taking $s \to \infty$ in \eqref{eqn:ContradictionConstants} we obtain
\[
(\alpha+2)c^{\alpha+1} \geq 1
\]
which is equivalent to $c \ge (\alpha+2)^{-\frac{1}{\alpha+1}}$. Taking $\alpha \to \infty$ in this last inequality implies the desired assertion, and finishes our proof. 
    
\end{proof}

\begin{remark} There is a clear difference between the limiting case of Part (ii) in Corollary \ref{cor:sharpness} and the bound obtained for each $\alpha > -1$. Effectively, it is not clear at the moment whether the the bound $c>(\alpha+2)^{-\frac{1}{\alpha+1}}$ is best possible for \emph{each} $\alpha > -1$ with the current techniques. 

It currently seems that a further refinement of the proof of Theorem \ref{thm:HardyStability} could be key in order to improve the constant until the threshold $(\alpha+2)^{-\frac{1}{\alpha+1}}$, but since it goes beyond the scope of this manuscript, we decided not to investigate that issue here. We believe, however, that the investigation of this matter would be on its own a very interesting further direction. 

\end{remark}

\section{Proof of Theorem \ref{thm:GGRT-main}}\label{sec:Fock}

\subsection{Concentration results for the short-time Fourier transform}

We start by showing how one can recover the following main result from \cite{NicolaTilli} as a limitting case of our results: 

\begin{theorem*}[Main~Theorem~in~\cite{NicolaTilli}] 
If $\Omega \subset \mathbb{R}^2$ is a measurable set with positive and finite Lebesgue measure, and if $F \in \mathcal{F}^2(\mathbb{C}) \backslash\{0\}$ is an arbitrary function, then
$$
\frac{\int_{\Omega}|F(z)|^2 e^{-\pi|z|^2} \mathrm{~d} z}{\|F\|_{\mathcal{F}^2}^2} \leq 1-e^{-|\Omega|}
$$

Moreover, equality is attained if and only if $\Omega$ coincides (up to a set of measure zero) with a ball centered at some $z_0 \in \mathbb{C}$ and, at the same time, $F=c F_{z_0}$ for some $c \in \mathbb{C} \backslash\{0\}$.
\end{theorem*}

\begin{proof} Fix a function $F \in \Fock$, which we shall write as 
\[
F(z) = \sum_{k \ge 0} \sqrt{\frac{\pi^k}{k!}} a_k z^k.
\]
We start by recalling some previous results. Indeed, the result in \cite{RamosTilli} establishes that for $f \in \mathbf{B}_{\alpha}^2$ and $\Omega \subset \mathbb{D}$ of hyperbolic measure $\mu(\Omega)=s$ satisfies
\[ \int_{\Omega} \lvert f(z) \rvert^2 (1-\lvert z \rvert^2)^{\alpha+2}  \, \mathrm{d} \mu(z) \leq \int_{D_{s}} (1-\lvert z \rvert^2)^{\alpha+2} \, \mathrm{d}z , \]
where $D_{s}$ is a disk centered at the origin, and of hyperbolic measure $\mu(D_{s})=s$. 
We then define, for each $\alpha$ sufficiently large, the auxiliary function $f_{\alpha} \in \Berg$ as
\begin{equation}\label{eqn:BergmanFunction} 
f_\alpha(z) = \sum_{k \geq 0} a_{k} \frac{z^k}{\sqrt{ c_{k} }} ,
\end{equation}
where $c_k$ are defined as in Section \ref{sec:Prelim}. Explicitly computing, we obtain
\[ \int_{D_s} \lvert f_{\alpha}(z) \rvert^2 (1-\lvert z \rvert^2)^{\alpha}  \, \mathrm{d}z = \int_{D_s} H(z)(1-\lvert z \rvert^2)^{\alpha} \, \mathrm{d}z + \sum_{k \geq 0} \lvert a_{k} \rvert^2 \int_{D_s} \frac{\lvert z \rvert^{2k}}{c_{k}} (1-\lvert z \rvert^2)^{\alpha} \, \mathrm{d}z ,\]
for some function $H(z)$. The only important property of $H$ for us for now is that, when $\Omega = D_{s}$, the integral term containing $H$ vanishes due to the orthogonality of monomials over centered circles, and we are left with only the right-most term. 

We begin the proof of the result in \cite{NicolaTilli} by letting $\Omega \subset \mathbb{R}^2$ be a bounded, measurable set of volume $\lvert \Omega \rvert >0$. Consider the rescaled set
\[ \Omega_{\alpha} = \sqrt{ \frac{\pi}{\alpha} } \Omega ,\]
then $\lvert \Omega_{\alpha} \rvert = \frac{\pi}{\alpha} \lvert \Omega \rvert$. For a large enough $\beta$, all the sets $\Omega_{\alpha}$ for $\alpha>\beta$ will be contained in $\mathbb{D}$, and therefore we can apply the result in \cite{RamosTilli}. This yields
\[ \int_{\Omega_{\alpha}} \lvert f_\alpha (z) \rvert^2 (1-\lvert z \rvert^2)^{\alpha}  \, \mathrm{d}z \leq \frac{\Gamma(\alpha+2)}{\Gamma(\alpha+1)} \cdot \frac{1}{\pi} \int_{D_{s_{\alpha}}} (1-\lvert z \rvert^2)^{\alpha}  \, \mathrm{d}z  .\]
The number $s_{\alpha}$ is simply the hyperbolic measure of $\Omega_{\alpha}$. The right-hand side can be integrated in polar coordinates, taking into account that the radius of the ball $D_{s}$ is $\left[1-\left( 1+\frac{s}{\pi} \right)^{-1}\right]^{1/2}$,
\begin{align*}
\int_{D_{s}} (1-\lvert z \rvert^2)^{\alpha} \, \mathrm{d}z &= 2\pi \int_{0}^{[1-(1+s/\pi)^{-1}]^{1/2}} (1-r^2)^{\alpha}r \, \mathrm{d}r \\
&= \pi \int_{0}^{1-(1+s/\pi)^{-1}} (1-t)^{\alpha} \, \mathrm{d}t = \frac{\pi}{ \alpha} \int_{0}^{\alpha[1-(1+s/\pi)^{-1}]} \left( 1-\frac{y}{\alpha} \right)^{\alpha}  \, \mathrm{d}y .
\end{align*}
Hence, the main result in \cite{RamosTilli} yields
\begin{equation}\label{eqn:ConsequenceRT}   \int_{\Omega_{\alpha}} \lvert f_\alpha (z) \rvert^2 (1-\lvert z \rvert^2)^{\alpha}  \, \mathrm{d}z \leq \frac{\Gamma(\alpha+2)}{\alpha \Gamma(\alpha+1)}  \int_{0}^{\alpha[1-(1+s_{\alpha}/\pi)^{-1}]} \left( 1-\frac{y}{\alpha} \right)^{\alpha}  \, \mathrm{d}y  .
\end{equation} 
Now, as we let $\alpha \to \infty$, the factor in front of the integral on the right-hand side converges to $1$, and the upper integration limit is asymptotic to $\alpha \cdot \frac{s}{s+\pi}$. At this point we need to ensure that for large values of $\alpha$, this agrees asymptotically with the volume $\lvert \Omega \rvert$. For that, since $\lvert z \rvert$ is small if $\alpha$ is large enough, we have
\[ s_{\alpha} = \mu(\Omega_{\alpha}) = \int_{\Omega_{\alpha}} \frac{1}{(1-|z|^2)^2}  \, \mathrm{d}z = \sum_{k \geq 1} k \int_{\Omega_{\alpha}} \lvert z \rvert^{2(k-1)} \, \mathrm{d}z =  \sum_{k \geq 1} k \left( \frac{\pi}{\alpha} \right)^{k} \int_{\Omega} \lvert z \rvert^{2(k-1)} \, \mathrm{d}z .\]
Then,
\[ \alpha \frac{s_{\alpha}}{\pi + s_{\alpha}} = \frac{\ \pi \lvert \Omega \rvert + \sum_{k \geq 2} \frac{\pi^k}{\alpha^{k-1}} \int_{\Omega} \lvert z \rvert^{2(k-1)} \, \mathrm{d}z }{\pi+s_{\alpha}} \overset{\alpha \to \infty}{\longrightarrow} \lvert \Omega \rvert .\]
Hence
\[ \frac{\Gamma(\alpha+2)}{\alpha \Gamma(\alpha+1)}  \int_{0}^{\alpha[1-(1+s_{\alpha}/\pi)^{-1}]} \left( 1-\frac{y}{\alpha} \right)^{\alpha}  \, \mathrm{d}y \overset{\alpha \to \infty}{\longrightarrow} \int_{0}^{\lvert \Omega \rvert} e^{-y} \, \mathrm{d}y = 1-e^{ -\lvert \Omega \rvert } .\]
We are now left to check that the left-hand side converges to the same thing that appears in \cite{NicolaTilli}. It amounts to computing what happens to the function $H$ from before when we take such limits. We note, however, that the left-hand side of \eqref{eqn:ConsequenceRT} can be written as
\[ \int_{\Omega_{\alpha}} \left[ \sum_{k \geq 0} \frac{1}{c_{k}} \sum_{l=0}^{k} a_{l}\overline{a_{k-l}} z^{l}\overline{z^{k-l}} \right] (1-|z|^2)^{\alpha} \, \mathrm{d}z , \]
where we used the explicit expansion of $f$. By performing a change of variables in order to undo the scaling and by letting $\alpha \to \infty$, we reach that the limit of the left-hand side of \eqref{eqn:ConsequenceRT} equals
\[ \int_{\Omega} \left[ \sum_{k=0}^{\infty} \sqrt{ \frac{\pi^k}{k!} } \sum_{l=0}^{k} a_{l}\overline{a_{k-l}} z^l \overline{z^{k-l}} \right] e^{ -\pi \lvert z \rvert^2 } \, \mathrm{d}z = \int_{\Omega} \lvert F(z) \rvert^2 e^{ -\pi \lvert z \rvert^2 } \, \mathrm{d}z .\]
The conclusion of the main result from \cite{NicolaTilli} then follows at once.  
\end{proof}

\begin{remark} One can also consider a hyperbolic scaling in order to scale the sets in hyperbolic space in a way that preserves the measure. This amounts to considering the transformation
\begin{equation}\label{eqn:HyperbolicScaling} 
\lvert z \rvert \mapsto \frac{(1+\lvert z \rvert)^r - (1-\lvert z \rvert)^r}{(1+\lvert z \rvert)^r + (1-\lvert z \rvert)^r} .
\end{equation} 
By employing such an estimate, one would not need to work with $s_{\alpha}$, and instead of the relationship we found between it and $\lvert \Omega \rvert$ at the limit, we would have an equality. 

Since, however, this entails on a more technical argument due to the not particularly linear nature of the scaling in \eqref{eqn:HyperbolicScaling}, and taking into account moreover the fact that the first order expansion of this hyperbolic scaling agrees with the euclidean one that we took, we decided not to perform those computations here. 
\end{remark} 

\subsection{Stability results for the short-time Fourier transform}\label{sec:Stability}

We now tackle the case of stability estimates, and we wish to deduce the main result from \cite{Inv2} from Theorem \ref{thm:BergmanStability}. We note that a similar method has been used in \cite{GarciaOrtega} for the polynomial case. 

Indeed, simply note that, from what we just proved above, the deficit in the Bergman case is shown to converge to the deficit in the Fock case, provided one looks at the rescaled sets $\Omega_{\alpha}$. On the other hand, note that we have the convergence
\[\frac{\alpha+1}{\alpha+2} \left( 1+\frac{s_{\alpha}}{\pi} \right)^{\alpha} \longrightarrow e^{\lvert \Omega \rvert} \]
as $\alpha \to \infty$. Finally, the stability for the function follows from the fact that the left-hand side of \eqref{eqn:BergmanStabilityFunction} is simply the $\ell^2$-norm of $f \in \mathbf{B}_{\alpha}^2$. Since this corresponds to $F \in \mathcal{F}^2$ through the same series representation, and the same $\ell^2$-norm, we obtain that the left-hand side of \eqref{eqn:BergmanStabilityFunction} is \emph{independent of $\alpha$}, and we therefore reach the result.

Estimate \eqref{eqn:BergmanStabilitySet} also implies the corresponding result in the Fock space from \cite{Inv2} through the same limiting argument. Indeed, the right-hand side of \eqref{eqn:BergmanStabilitySet} converges to the right-hand side of its Fock space counterpart. Moreover, if $\Om$ is a ball, then $\mcalA(\Om)=\mcalA_{\disk}(\Om)=0$. Otherwise, letting $B^n$ be a minimizing sequence of balls for $\mcalA_{\disk}(\Om)$ and $B^n_{\alpha}$ be their rescalings, we have
\[ \mu(\Om_{\alpha} \Delta B^n_{\alpha}) = \mu\of{(\Om \Delta B^n)_{\alpha}} ,\]
and therefore $B^n_{\alpha}$ also minimizes $\mcalA_{\disk}(\Om_{\alpha})$. Up to a subsequence, we have
\begin{align*}
    \mcalA_{\disk}(\Om_{\alpha}) + 1/n &\geq
    s_{\alpha}^{-1} \mu(\Om_{\alpha} \Delta B^n_{\alpha}) =
    s_{\alpha}^{-1} \mu\of{(\Om \Delta B^n)_{\alpha}}\\
    &=\frac{\sum_{k \geq 1} k \of{\frac{\pi}{\alpha}}^k \int_{\Om \Delta B^n} \abs{z}^{2(k-1)} \diff z}{\sum_{k \geq 1} k \of{\frac{\pi}{\alpha}}^k \int_{\Om} \abs{z}^{2(k-1)} \diff z}
    = \frac{\vol{\Om \Delta B^n} + \sum_{k \geq 2} \frac{\pi^k}{\alpha^{k-1}} \int_{\Omega \Delta B^n} |z|^{2(k-1)} \diff z}{\vol{\Om} + \sum_{k \geq 2} \frac{\pi^k}{\alpha^{k-1}} \int_{\Omega} |z|^{2(k-1)} \diff z} .
\end{align*}
Finally, by choosing $n = m / \mcalA(\Om)$ and letting $\alpha \to \infty$, we find
\[ \limsup_{\alpha \to \infty} \mcalA_{\disk}(\Om_{\alpha}) \geq \of{\frac{m-1}{m}} \mcalA(\Om) .\]
Since $m$ was arbitrary, we conclude that 
\begin{align*}
\mathcal{A}(\Omega) \le \limsup_{\alpha \to \infty} \mathcal{A}_{\disk}(\Omega_\alpha) \le \limsup_{\alpha \to \infty} K(s,\alpha) \delta(f_{\alpha};\Om,\alpha)^{1/2} = K(s) \delta_{\Fock}(F;\Om)^{1/2},
\end{align*}
as desired. 

\section{Proof of Theorem \ref{thm:HardyStability}}\label{sec:Hardy} We proceed in a similar manner as in the preceding section, but this time we do \emph{not} employ any rescaling on the sets under consideration. Indeed, for a fixed function $f \in H^2(\disk)$ with $\|f\|_{H^2} =1$, we have that $f$ may be written as 
\begin{equation}\label{eqn:HardyExpansion} 
f(z) = \sum_{k \ge 0} a_k z^k,
\end{equation}
where $a_k \in \ell^2(\N)$, with $\|a_k\|_{\ell^2(\N)} = 1$. Conversely, any function which possesses an expansion like \eqref{eqn:HardyExpansion} for some sequence $a_k \in \ell^2(\N)$ with norm $1$ belongs automatically to $H^2(\disk)$ and has $H^2(\disk)$-norm equal to 1. We thus define $f_{\alpha}(z) \in \Berg$ in the exact same fashion as in \eqref{eqn:BergmanFunction}, and apply Theorem \ref{thm:BergmanStability} to $f_{\alpha}$. This directly implies that 
\begin{equation}\label{eqn:DeficitGenBerg} 
 \inf_{\substack{\abs{c} = \norm{f_\alpha}_{\Berg} ,\\ \om \in \disk}} \frac{\norm{f_\alpha-c\ff_{\om}}_{\mathbf{B}_{\alpha}^2}}{\norm{f_\alpha}_{\mathbf{B}_{\alpha}^2}} \leq C \of{1 + \frac{\alpha+2}{\alpha+1} \left[ \of{1+\frac{s}{\pi}}^{\alpha+1} -1\right]}^{1/2} \delta(f_\alpha;\Omega,\alpha)^{1/2}. 
\end{equation} 
We deal first with the left-hand side: we from the definition \eqref{eqn:BergmanFunction}, it follows directly that $\lim_{\alpha \to -1} \|f_\alpha \|_{\Berg} = \|f\|_{H^2(\disk)}$ (see also \cite{Zhu-Translating}). In order to deal with the numerator, note that, similarly to what we did in Lemma \ref{lemma:RKHSQuantitative}, we have
\begin{equation}\label{eqn:DifferenceBargmanExtreme} \inf_{\substack{\abs{c} = \norm{f_\alpha}_{\Berg} ,\\ \om \in \disk}} \|f_\alpha - c \ff_\om\|_{\Berg}^2 = 2\|f_\alpha\|_{\Berg}^2 - 2 \sup_{\om \in \disk} \|f_\alpha \|_{\Berg} \left| \langle f_\alpha, \ff_\om \rangle_{\Berg} \right|. 
\end{equation} 
Since the first term on the right-hand side of \eqref{eqn:DifferenceBargmanExtreme} was already shown to converge, it suffices to deal with the second. Namely, we wish to show that 
\begin{equation}\label{eqn:LimsupLimit} 
\limsup_{\alpha \to -1} \sup_{\om \in \disk} |\langle f_\alpha, \ff_\om \rangle_{\Berg}| \ge  \sup_{\om \in \disk} |\langle f, \mathbf{g}_\om \rangle_{H^2}|.
\end{equation} 
We begin by showing a slightly weaker result: for each fixed $\om \in \disk$, we have
\begin{equation}\label{eqn:ConvergenceHardyInner}
\langle f_\alpha, \ff_\om \rangle_{\Berg} \to \langle f,\mathbf{g}_\om \rangle_{H^2} \quad \text{as } \alpha \to -1.
\end{equation} 
By definition of reproducing kernels, however, we have that 
\[
\langle f_{\alpha}, \ff_\om \rangle_{\Berg} = \frac{f_\alpha(\om)}{\sqrt{\frac{\alpha+1}{\pi}} (1-|\om|^2)^{-\frac{\alpha+2}{2}}}  
 \to (1-|\om|^2)^{1/2} f(\om),\]
 where the asserted convergence follows from the series expansion. Indeed, for a fixed $\omega \in \disk$, we have 
 \[
 \sqrt{\frac{\pi}{\alpha+1}} f_\alpha(\om) = \sum_{k \ge 0} \frac{a_k}{\sqrt{(\alpha+1)c_k/\pi}} \om^k.
 \]
 We then note that 
 \[
 \frac{(\alpha+1)c_n}{\pi} = \frac{(\alpha+1)\Gamma(\alpha+1)\Gamma(n+1)}{\Gamma(2+\alpha+n)} = \frac{\Gamma(\alpha+2) \Gamma(n+1)}{\Gamma(2+\alpha+n)}. 
 \]
 For each fixed $n$, it follows directly from this last identity that $\frac{(\alpha+1)c_n}{\pi} \to 1$. In order to better quantify this, we use Gautschi's inequality for the Gamma function (see \cite[(5.6.4)]{NIST:DLMF}). Using that for $x = n+\alpha+1$ and $s = -\alpha$ where $\alpha \in (-1,0)$, we obtain that 
 \begin{equation}\label{eqn:GautschiBoundGamma} 
\Gamma(\alpha+2) (n+\alpha+2)^{-\alpha-1} \le \left| \frac{(\alpha+1)c_n}{\pi}\right| \le \Gamma(\alpha+2) \cdot (n+\alpha+1)^{-\alpha -1}. 
 \end{equation} 
 Since any $\{a_k\}_k \in \ell^2(\N)$ is a bounded sequence, we conclude from \eqref{eqn:GautschiBoundGamma} that the coefficients of $f_{\alpha}$ converge to $1$, and that they are each bounded by at most a polynomially growing function, which is uniformly bounded in $\alpha \in (-1,0)$. Since $\omega \in \disk$, we have that $\frac{a_k}{\sqrt{(\alpha+1)c_k/\pi}} \omega^k$ converges exponentially to $0$ as $k$ goes to infinity, uniformly on $\alpha\in(-1,0)$. This allows us to conclude the asserted fact that $\sqrt{\frac{\pi}{\alpha+1}} f_{\alpha}(\om) \to f(\om)$, as desired. 

 We now claim that $(1-|\om|^2)^{1/2} f(\omega) = \langle f,\mathbf{g}_\om\rangle_{H^2}$. This is equivalent to 
 \[
 f(\om) = \frac{1}{2\pi} \int_0^{2\pi} \frac{f(e^{i\theta})}{1-\om e^{-i\theta}} \, \diff \theta. 
 \]
 This, however, follows at once from several different arguments: for instance, one may expand $(1-\om e^{-i\theta})^{-1} = \sum_{k \ge 0} \left( \frac{\om}{e^{i\theta}}\right)^k$, change the order of summation and compare coefficients. Another way to do it is directly from the Cauchy integral formula applied to a fixed radius $r \in (0,1)$ sufficiently close to $1$: indeed, we have 
 \[
 f(\om) = \frac{1}{2\pi} \int_0^{2\pi} \frac{f(re^{i\theta})}{r - \om e^{-i\theta}} \, \diff \theta \quad \text{for }  r \text{ close to } 1. 
 \]
 Taking limits $r \to 1^-$ lets us arrive at the same conclusion. This is enough to conclude \eqref{eqn:ConvergenceHardyInner}. 

 We now fix $\varepsilon>0$ and $\om_0 \in \disk$ such that $\sup_{\om \in \disk} \left| \langle f,\mathbf{g}_\om \rangle_{H^2} \right| \le \left|\langle  f,\mathbf{g}_{\om_0} \rangle_{H^2}\right| + \varepsilon.$ We then use \eqref{eqn:ConvergenceHardyInner}, which allows us to deduce that
 \[
 \sup_{\om \in \disk} \left| \langle f,\mathbf{g}_\om \rangle_{H^2}\right| \le \limsup_{\alpha \to -1} \left|\langle  f_\alpha,\ff_{\om_0} \rangle_{\Berg}\right| + \varepsilon \le \limsup_{\alpha \to -1} \sup_{\om \in \disk} \left| \langle f_{\alpha},\ff_\om \rangle_{\Berg}\right| + \varepsilon. 
 \]
 Taking $\varepsilon \to 0$ yields then \eqref{eqn:LimsupLimit}. This, on the other hand, shows that 
 \begin{equation}\label{eqn:LimsupBoundHardy} 
  \inf_{\substack{\abs{c} = \norm{f}_{H^2} ,\\ \om \in \disk}} \frac{\norm{f -c\mathbf{g}_{\om}}_{H^2}}{\norm{f}_{H^2}} \le \limsup_{\alpha \to -1}  \inf_{\substack{\abs{c} = \norm{f_\alpha}_{\Berg} ,\\ \om \in \disk}} \frac{\norm{f_\alpha-c \ff_{\om}}_{\mathbf{B}_{\alpha}^2}}{\norm{f_\alpha}_{\mathbf{B}_{\alpha}^2}}. 
 \end{equation} 
 We now deal with the right-hand side of \eqref{eqn:DeficitGenBerg}. Note that the constant 
 \[
 1 + \frac{\alpha+2}{\alpha+1} \left[ \left( 1 + \frac{s}{\pi} \right)^{\alpha+1} - 1 \right] \to 1 + \log\left( 1 + \frac{s}{\pi} \right) \quad \text{ as } \alpha \to -1, 
 \]
 and hence we only need to deal with the limit of the deficit factor. We write that explicitly as 
 \begin{align}
    \delta(f_\alpha;\Omega,\alpha) & = 1 - \frac{ \int_\Om |f_\alpha(z)|^2(1-|z|^2)^{\alpha+2} \, \diff \mu(z)}{\|f_\alpha\|_{\Berg} \theta_{\alpha}(s)} \cr 
    & = 1 - \frac{ \int_\Om \left| \sqrt{\frac{\pi}{\alpha+1}} f_\alpha(z)\right|^2(1-|z|^2)^{\alpha+2} \, \diff \mu(z)}{\|f_\alpha\|_{\Berg} \frac{\pi \cdot \theta_{\alpha}(s)}{\alpha+1}}.
 \end{align}
 We now use the dominated convergence theorem: indeed, the considerations above show that 
 \[
 \left| \sqrt{\frac{\pi}{\alpha+1}} f_\alpha(z)\right|^2(1-|z|^2)^{\alpha+2} = |\langle f_\alpha , \ff_\om \rangle_{\Berg}|^2 \le 1.
 \]
 Furthermore, we also have that 
 \[
 \left| \sqrt{\frac{\pi}{\alpha+1}} f_\alpha(z)\right|^2(1-|z|^2)^{\alpha+2} \to \left| f(z) \right|^2 (1-|z|^2) \quad \text{ pointwise},
 \]
 which, since $\mu(\Omega) < +\infty$, implies that 
 \[
 \int_\Om \left| \sqrt{\frac{\pi}{\alpha+1}} f_\alpha(z)\right|^2(1-|z|^2)^{\alpha+2} \, \diff \mu(z) \to \int_\Om |f(z)|^2 (1-|z|^2) \, \diff \mu(z) \quad \text{ as } \alpha \to -1. 
 \]
 For the denominator of the expression we wish to analyze, we note that it suffices to compute $$\lim_{\alpha \to -1} \frac{\theta_{\alpha}(s)}{\alpha+1},$$
 since we have already proved convergence in norm. But this is a completely explicit computation: 
 \[
 \frac{\theta_{\alpha}(s)}{\alpha+1} = \frac{1 - \left( 1 + \frac{s}{\pi}\right)^{-\alpha-1}}{\alpha +1} \to \log\left( 1 + \frac{s}{\pi}\right) \quad \text{ as } \alpha \to -1. 
 \]
 We conclude that the right-hand side of \eqref{eqn:DeficitGenBerg} converges to 
 \begin{equation}\label{eqn:RHSHardyDeficit} 
 C \cdot \left[ 1 + \log^{1/2} \left( 1 + \frac{s}{\pi} \right) \right] \cdot \delta_{H^2}(f;\Om)^{1/2} \quad \text{ as } \alpha \to -1.
 \end{equation} 
 Putting this together with \eqref{eqn:DeficitGenBerg} and \eqref{eqn:LimsupBoundHardy}, we obtain the first part of Theorem \ref{thm:HardyStability}. 

 In order to prove the assertion about the asymmetry, we use the explicit upper bound we have on $K(s,\alpha)$ as given by \eqref{eqn:upperboundK} in Remark \ref{rmk:K-value}: according to that, we have that 
 \begin{equation}\label{eqn:LimitConstantSetStab}
 \limsup_{\alpha \to -1} K(s,\alpha) \le C \left( 
s^{-1} \cdot \left( 1 + \frac{s}{\pi}\right)^2 + \log^{1/2} \left( 1 + \frac{s}{\pi}\right) \cdot \left( 1 + \frac{s}{\pi}\right)\right)=:N(s). 
 \end{equation}
 Since we have already shown that the right-hand side of \eqref{eqn:DeficitGenBerg} converges to \eqref{eqn:RHSHardyDeficit}, and  since $\mathcal{A}_\disk(\Om)$ does not change as $\alpha \to -1$, we conclude that upon taking $\limsup_{\alpha \to -1}$ of \eqref{eqn:BergmanStabilitySet}, taking \eqref{eqn:LimitConstantSetStab} into account, we obtain \eqref{eqn:AsymmetryHardy}, as desired, concluding thus the proof of Theorem \ref{thm:HardyStability}. 

 \begin{remark} As noted in the introduction, Theorem \ref{thm:HardyStability} may be interpreted as a (sharp stability version of) a novel concentration inequality for Poisson extensions of functions in $H^2$. It is immediate from a change of variables that these results transfer immediately to the upper half plane. 

 In that regard, it is worth noting that all results in this manuscript yield sharp stability versions of concentration inequalities in \emph{general} simply connected domains. Indeed, one simply uses a biholomorphism from the unit disk to any such domain, in the same vein as we did in order to transfer from the upper half space to the unit disk in Section \ref{sec:Prelim}. 

 In particular, by applying this procedure to Theorem \ref{thm:HardyStability}, one would obtain a result on Poisson extensions to such domains. It is important to note that the weights yielded by that result are inherently connected to the biholomorphism used, and that as such the definition of Hardy spaces on the domain is also intimately related to the one for the unit disk. 

 The latter remark motivates the following question: can one obtain a concentration inequality for Poisson extensions of \emph{general} $L^2$ functions on the boundary of the unit disk? This question seems to go partly beyond the scope of techniques used here: the functions under consideration are no longer analytic functions in the unit disk, but merely harmonic functions. Moreover, the crucial property used in \cite{NicolaTilli,KNOT,KalajRamos,Kalaj1,Kalaj2024,KalajMelentijevic} that the functions under consideration are \emph{log-subharmonic} is also not valid in general, and this problem hence seems to require new ideas. We wish to address this issue in a future work.  
 \end{remark}


\printbibliography[heading=bibintoc,title=References]

 \end{document}